\documentclass[preview,12pt]{article}

\usepackage{dgjournal}          

\usepackage{mathptmx}

\usepackage{graphics}

\usepackage[authoryear,comma,longnamesfirst,sectionbib]{natbib} 

\usepackage[utf8]{inputenc}
\usepackage[english]{babel}
 
\usepackage{blkarray}

\usepackage{url}
\usepackage{amsmath}
\usepackage[linesnumbered,ruled]{algorithm2e}
\usepackage{amsfonts}
\usepackage{amsthm}
\usepackage{mathtools}
\usepackage{float}
\usepackage{caption}
\usepackage{subcaption}
\usepackage{booktabs}

\usepackage{graphicx}
\usepackage{listings}
\usepackage{amsmath}

\DeclareMathOperator{\mmd}{MMD}
\DeclareMathOperator{\AR}{AR}
\DeclareMathOperator{\LL}{L}
\DeclareMathOperator{\ar}{AR}
\def\independenT#1#2{\mathrel{\rlap{$#1#2$}\mkern2mu{#1#2}}}
\DeclareMathOperator{\hsic}{HSIC}
\DeclareMathOperator{\whsic}{wHSIC}

\newcommand\independent{\protect\mathpalette{\protect\independenT}{\perp}}
\def\independenT#1#2{\mathrel{\rlap{$#1#2$}\mkern2mu{#1#2}}}
\newcommand{\tr}{\operatorname{tr}}
\newtheorem{definition}{Definition}[section]
\newtheorem{theorem}{Theorem}[section]
\newtheorem{lemma}{Lemma}[section]
\usepackage{xr}
\makeatletter
\newcommand*{\addFileDependency}[1]{
  \typeout{(#1)}
  \@addtofilelist{#1}
  \IfFileExists{#1}{}{\typeout{No file #1.}}
}
\makeatother

\begin{document}


\textbf{ \Large A kernel- and optimal transport- based test of independence between covariates and right-censored lifetimes}\\

{D. Rindt, D. Sejdinovic and D. Steinsaltz.}\\

{Department of Statistics, University of Oxford, UK}

\section{Abstract}
We propose a nonparametric test of independence, termed optHSIC, between a covariate and a right-censored lifetime. Because the presence of censoring creates a challenge in applying the standard permutation-based testing approaches, we use optimal transport to transform the censored dataset into an uncensored one, while preserving the relevant dependencies. We then apply a permutation test using the kernel-based dependence measure as a statistic to the transformed dataset. The type 1 error is proven to be correct in the case where censoring is independent of the covariate. Experiments indicate that optHSIC has power against a much wider class of alternatives than Cox proportional hazards regression and that it has the correct type 1 control even in the challenging cases where censoring strongly depends on the covariate.

\section{Introduction}
\label{sec1}

We propose a nonparametric test of independence between a possibly multidimensional covariate and a right-censored lifetime. Existing approaches to this problem suffer from several limitations: if we cluster the continuous covariate into groups, and then test for equality of
lifetime distributions among the groups, the results will
depend on the arbitrary choice of boundaries between the groups,
while the spread of covariates within groups reduces power. Alternatively, one might fit a (semi-)parametric regression model, and test whether the regression coefficient corresponding to the covariate differs significantly from zero. The most commonly used such method is the Cox proportional hazards (CPH) model, which makes two assumptions (\cite{cph}): first, the hazard function must factorize into a function of time and a function of the covariate (the {\em proportional hazards} or {\em relative risk} condition); second, the effect of a covariate on the logarithm of the hazard function must be linear. 
Although this is a flexible model, in some cases these assumptions are violated. More complicated hazards are found for example when studying the relationship between body mass index and mortality - \cite{zajacova2012shape} reports $U$- or $V$-shaped hazards - or between diastolic blood pressure and various health outcomes (\cite{lip2019diastolic}).

Since distance- and kernel-based approaches have been used successfully for independence testing on uncensored data (\cite{distancecovariance}, \cite{kernelteststatisticalindependence}), it is natural to investigate whether these methods can be extended to the case of right-censored lifetimes. To this end we propose applying optimal transport to transform the censored dataset into an uncensored dataset in such a way that, 1) the new uncensored dataset preserves the dependencies of the original dataset, and 2) we can apply a standard permutation test to the new dataset with test statistic given by Distance Covariance (DCOV) (\cite{distancecovariance}) or, equivalently, the Hilbert--Schmidt Independence Criterion (HSIC) (\cite{kernelteststatisticalindependence}). 

Progress in kernel-based independence testing for censored data is further motivated by the fact that in the simpler context of uncensored data the corresponding methods have been further developed into tests of conditional independence, mutual independence, and have been applied to causal inference (\cite{zhang2012kernel}, \cite{doi:10.1111/rssb.12235}), and in particular detection of confounders. While the present work does not propose methods for testing conditional or mutual independence, we believe independence testing is a first step towards those ends. Additionally, since our method allows for multidimensional covariates, one can first test for a dependency based on the full multidimensional covariate, and then test whether the dependency remains when certain sub-dimensions are omitted from the covariate.

Section \ref{sec:background_material} overviews relevant concepts in survival analysis, distance- and kernel-based independence testing, and optimal transport. Section \ref{sec:transformation} proposes a transformation of the data based on optimal transport. Section \ref{sec:opthsic} introduces our testing procedure named optHSIC. Although we have not yet been able to prove control of the type 1 error rate in full generality, we do show the type 1 error rate to be correct in the case where censoring is independent of the covariate. Furthermore we obtain very promising results in simulation studies, showing correct type 1 error control even under censoring that depends strongly on the covariate.  Section \ref{sec:alternative_methods_when_censoring_independent} explores alternative kernel-based approaches under the additional assumption that censoring is independent of the covariate. These methods serve as benchmarks for the power performance of optHSIC. Section \ref{sec:experiments} compares the power and type 1 error of all tests and CPH regression in simulated data.

\section{Background Material}\label{sec:background_material}

\subsection{Right-Censored Lifetimes}

Let $T \in \mathbb R_{\geq 0}$ be a lifetime subject to right-censoring, so that we do not observe $T$ directly, but instead observe $Z \coloneqq \min \{ T,C \}$ for some censoring time $C \in \mathbb R_{\geq 0}$, as well as the indicator $\Delta\coloneqq 1\{C>T\}$. We further observe a covariate vector $X \in \mathbb R^d$, where $\mathbb R^d$ is equipped with the Borel sigma algebra. In total, for an i.i.d. sample of size $n$, we thus obtain the data $D \coloneqq ((x_i,z_i, \delta_i))_{i=1}^n \in (\mathbb R^d \times \mathbb R_{\geq 0} \times \{0,1 \})^n$. 

The main goal of this paper is developing a test of $H_0: X\independent T$ versus $H_1:X \not \independent T$ based on the sample $D$. Throughout this paper we will make the following assumptions. \\

\textbf{Assumption 1:} We assume that conditional on $\{X_i\}_{i=1}^n$, the random variables $\{(T_i,C_i)\}_{i=1}^n$ are mutually independent. \\

\textbf{Assumption 2:} We assume that $C \independent T \vert X.$ \\

Denote $F_{T\vert X}(t\vert x)=P(T\leq t\vert X=x)$ and $F_{C\vert X}(t\vert x)=P(C\leq t\vert X=x)$. We assume $F_{T|X}$ has a density $f_{T|X}$. Let $S_{T|X}(t|x)=1-F_{T|X}(t|X)$. We define the hazard rate of an individual with covariate $x$ to be
$\lambda_{T|X}(t\vert x)=f_{T|X}(t\vert x)/S_{T|X}(t\vert x).$
The Cox proportional hazards model (CPH) assumes that the hazard rate can be written as $\lambda_{T|X}(t\vert x)=\lambda(t)\exp(\beta^T x)$ for some baseline hazard $\lambda(t)$ and a vector $\beta \in \mathbb R^d$. This model enables estimation of $\beta$ and testing the significance of the difference of entries of $\beta$ from zero. The CPH model is the most commonly used regression method in survival analysis. 

A last important concept is that of the `individuals at risk at a time $t$'. By this we mean the set $\{i:Z_i\geq t\}$. We also use the notation $[a_1,\dots,a_k]$ for a multiset (a set with potentially repeated elements) with elements $a_1,\dots,a_k$. We refer, for example, to the multiset $\AR_t=[x_i: i\in \{1,\dots, n\}, Z_i\geq t ]$, the multiset of `covariates at risk at time $t$'. 

\subsection{Independence testing using kernels \label{sec:backround_material_hsic}}

Kernel methods have been successfully used for nonparametric independence- and two-sample testing (\cite{kernelteststatisticalindependence}, \cite{kerneltwosampletest}). We now give some of the relevant background in kernel methods. 

\begin{definition}{ (Reproducing Kernel Hilbert Space)}(\cite{learningwithkernels}) Let $\mathcal X$ be a non-empty set and $ H$ a Hilbert space of functions $f:\mathcal X \to \mathbb R$ endowed with dot product $\langle \cdot,\cdot \rangle $. Then $ H$ is called a reproducing kernel Hilbert (RKHS) space if there exists a function $k: \mathcal X \times \mathcal X \to \mathbb R$ with the following properties.
\begin{enumerate}
\item $k$ satisfies the reproducing property 
\begin{equation}
\langle f,k(x,\cdot)\rangle=f(x) \ \text{for all} \ f \in  H, x \in \mathcal X.
\end{equation}
\item $k$ spans $ H$, that is, $ H= \overline{\mathrm{span}\{k(x,\cdot) \ \vert x \in {\mathcal X} \}}$ where the bar denotes the completion of the space.
\end{enumerate}  \end{definition}
Let $\mathcal X$ together with a sigma-algebra be a measurable space and let $H_{\mathcal X}$ be an RKHS on $\mathcal X$ with reproducing kernel $k$. Let $P$ be a probability measure on $\mathcal X$.  If $E_P\sqrt{k(X,X)}<\infty$, then there exists an element $\mu_P\in H_{\mathcal X}$ such that $E_Pf(X)=\langle f,\mu_P \rangle$ for all $f\in H_{\mathcal X}$ (\cite{kerneltwosampletest}), where the notation $E_P f(X)$ is defined to be $\int f(x) P(dx)$. The element $\mu_P$ is called the mean embedding of $P$ in $H_{\mathcal X}$. Given a second distribution $Q$ on $\mathcal X$, for which a mean embedding exists, we can measure the dissimilarity of $P$ and $Q$ by the distance between their mean embeddings in $H_{\mathcal X}$:
\begin{align*}
\mmd(P,Q)\coloneqq \vert \vert \mu_P - \mu_Q \vert \vert _{H_{\mathcal X}},
\end{align*}
which is also called the Maximum Mean Discrepancy ($\mmd$). The name comes from the following equality (\cite{kerneltwosampletest}), 
\begin{align*}
\vert \vert \mu_P - \mu_Q \vert \vert _{H_{\mathcal X}} = \sup_{f \in H_{\mathcal X}: \| f \| \leq 1} E_P f(X) -E_Qf(X)
\end{align*}
showing that MMD is an integral probability metric. Given a sample $\{x_i\}_{i=1}^n$ and the empirical distribution,  $\sum_{i=1}^n \delta({x_i})/n$, where $\delta(x)$ denotes the Dirac measure at $x$, the corresponding mean embedding is given by $ \sum_{i=1}^n k(x_i,\cdot)/n.$

Suppose now that $\mathcal Y$ together with some sigma algebra is a second measurable space, and let $H_{\mathcal Y}$ be an RKHS on $\mathcal Y$ with kernel $l$. Let $X$ be a random variable in $\mathcal X$ with law $P_X$ and similarly let $Y$ be a random variable in $\mathcal Y$ with law $P_Y$. Finally let $P_{XY}$ denote the joint distribution on $\mathcal X \times \mathcal Y$ equipped with the product sigma-algebra. We let $H$ denote the RKHS on $\mathcal X \times \mathcal Y$ with kernel
$$ K((x,y),(x',y'))\coloneqq k(x,x')l(y,y').$$

In \cite{kernelteststatisticalindependence} it was proposed that the dependence of $X$ and $Y$ could be quantified by the following measure:

\begin{definition}
The Hilbert--Schmidt independence criterion (HSIC) of $X$ and $Y$ is defined by
\begin{align*}
\hsic(X,Y)\coloneqq \vert \vert \mu_{P_{XY}}-\mu_{P_XP_Y} \vert \vert^2_{H}
\end{align*}
where $P_XP_Y$ denotes the product measure of $P_X$ and $P_Y$.
\end{definition}
Let $(X,Y)$,$(X',Y')$ and $(X'',Y'')$ be three mutually independent copies of the same random variable with law $P_{XY}$. Using the reproducing property and the definition of mean embeddings, it can be shown that 
\begin{align*}\begin{split} \hsic(X,Y) &=  E_{XY}E_{X'Y'}k(X,X')l(Y,Y') + E_{XX'}  k(X,X')E_{YY'}l(Y,Y') \\ &-2E_{XY} E_{X'Y''}k(X,X') l(Y,Y''). \end{split}
\end{align*}

Now assume we are given a sample $D=((x_i,y_i) )_{i=1}^n$ of independent observations of the random pair
$(X,Y)$.  An empirical estimate of HSIC$(X,Y)$ can be obtained by measuring the distance between the embedding of the empirical distribution of the data and the embedding of the product of the marginal empirical distributions. That is, we define $\hsic(D)$ by
\begin{align*}
\hsic(D) \coloneqq \bigg \vert \bigg \vert \frac 1n \sum_{i=1}^n K((x_i,y_i),\cdot)   - \frac{1}{n^2} \sum_{i=1}^n  \sum_{j=1}^n K((x_i,y_j),\cdot)\bigg \vert \bigg \vert_{H}^2. 
\end{align*}
Using the reproducing property of the kernel and the definition of $K$ in terms of $k$ and $l$, $\hsic(D)$ can be shown to equal
\begin{align*}
\hsic(D) &= \frac{1}{n^2} \sum_{i,j=1}^n k(x_i,x_j)l(y_i,y_j)+ \frac{1}{n^2} \sum_{i,j=1} k(x_i,x_j)\frac{1}{n^2}\sum_{q,r=1}^nl(y_q,y_r) \\ &- \frac{2}{n^3}\sum_{i,j,r=1}^nk(x_i,x_j)l(y_i,y_r).
\end{align*} 
While the biased $\hsic(D)$ defined above is the most commonly used estimator in the literature (\cite{distancecovariance}, \cite{kernelteststatisticalindependence}), unbiased estimators of $\hsic(X,Y)$ exist too (\cite{song2012feature}). The bias of $\hsic(D)$ is $O(n^{-1})$, (\cite{kernelteststatisticalindependence}) and for appropriate choices of kernels, the permutation test with the biased statistic and the permutation test with the unbiased statistic are consistent tests - see section \ref{sec:choice_of_kernel} for details. Both tests also have correct type 1 error rate. Because the biased $\hsic(D)$ is more commonly encountered in the literature and has a slightly easier analytic form, we use $\hsic(D)$ throughout our paper.  The following algorithm shows how $\hsic$ is commonly combined with a permutation test for independence testing. 
\begin{figure}[H]
\begin{algorithm}[H] \label{hsicpermutationalgorithm}
    \SetKwInOut{Input}{Input}   \SetKwInOut{Output}{Output}
	\Input{ Observed data $D=((x_i,y_i))_{i=1}^{n}$, significance level $\alpha$, number of permutations $B$. } 
	Sample $\pi_j, \ 1\leq j \leq B$  distributed uniformly and independently from the symmetric group $S_n$ of all permutations on $n$ elements. Denote $\pi D \coloneqq  ((x_{\pi (i)},y_i) )_{i=1}^n$.\;
	Breaking ties at random, compute the rank $R$ of $\hsic(D)$ in the vector $$\left(\hsic (D),\hsic (\pi_1 D),\hsic (\pi_2 D),\dots,\hsic (\pi_{B} D) \right)$$\
	where $R=1$ is the rank of the largest element and $R=B+1$ is the rank of the smallest element. \;
	\Output{Reject $H_0$ if $p\coloneqq R/(B+1)\leq  \alpha$, otherwise accept $H_0$. }
	
    \caption{Independence testing using HSIC and a permutation test}
\end{algorithm}
\vspace{.5cm}
\caption{The algorithm for independence testing using HSIC and a permutation test, resulting in a $p$-value and a rejection decision. }
\end{figure}

\subsubsection{Choice of kernel \label{sec:choice_of_kernel}}
Throughout this paper we assume the covariates take values in the Euclidean space $\mathcal X=\mathbb R^d$ for some $d\in \mathbb N$. Note that in our case $\mathcal Y=\mathbb R_{\geq 0}$. We let both $k$ and $l$ be instances of the covariance kernel of Brownian motion. That is, $$k(x,x')= \left( \| x\| + \| x' \| - \| x-x' \| \right),$$ and 
$$l(y,y')=\left( |y| + |y'| - | y-y' | \right),$$
where $\|\cdot \|$ denotes the Euclidean norm. See \cite{equivalencedistancerkhs} for a discussion of this kernel. \\

The reason for this choice is three-fold. Firstly, under this choice of kernels, $\hsic$ coincides with Distance Covariance (DCOV) (\cite{distancecovariance}). The equivalence between $\hsic$ and DCOV was proved in \cite{equivalencedistancerkhs}. DCOV is a  well studied measure of dependence and if $X,Y$ are random variables with compact support, then $\text{DCOV}(X, Y)=0$ if and only if $X\independent Y$. Furthermore, the permutation test with test statistic DCOV is consistent in the sense that, for each distribution $P_{XY}$ with compact support and such that $X\not \independent Y$, the probability that the permutation test rejects the null hypothesis converges to 1, as the sample size converges to infinity (\cite{rindt2020consistency}). Secondly, unlike other potential kernels, such as the Gaussian kernel, the covariance kernel of Brownian motion spares us the need to select a bandwidth. Thirdly, our test relies on optimal transport to transform the original data into a transformed dataset. The optimal transport procedure, discussed in the next section, requires a metric with respect to which similarity is preserved. It appears natural to measure the dependency in the transformed dataset using the same metric as the one underlying the transformation. 

\subsection{Optimal transport }\label{sec:background_optimal_transport}
Let $A =[a_1,\dots,a_{|A|}]$ and $B = [b_1,\dots,b_{|B|}]$ be multisets of length $|A|$ and $|B|$ with $a_i,b_j \in \mathbb R^d$. Let $v^A$ be a distribution vector on $A$, by which we mean: $v^A=(v^A_1,\dots,v^A_{|A|}) \in \mathbb R^{|A|}$ so that $v^A_i\geq 0$ and $\sum_{i=1}^{|A|} v^A_i=1$. Let $v^B$ be a distribution vector on $B$. Let $m(a,b)$ be a metric on $\mathbb R^d$. Then the earth mover's distance (EMD) between $v^A$ and $v^B$ is defined as
\begin{align}\label{eqn:objective_emd}
\min_{P_{ij}\in \mathbb R^{|A|\times|B|}}\sum_{i=1}^{|A|}\sum_{j=1}^{|B|} P_{ij}m(a_i,b_j)
\end{align}
subject to
\begin{align*}
P_{ij}&\geq 0 &i=1,\dots,|A|,& \ j=1,\dots,|B|,\\
\sum_{j=1}^{|B|}P_{ij}&=v^A_i &i=1,\dots,|A|&, \ \\
\sum_{i=1}^{|A|}P_{ij}&=v^B_j &j=1,\dots,|B|&. \ \\
\end{align*}
Let $U$ be a random variable taking values in $A$ with distribution $v^A$, and $U'$ a random variable taking values in $B$ with distribution $v^B$. In the next section we will write `let $P$ be the joint distribution that solves the optimal transport problem between $U$ and $U'$'. By that we mean $P(U=a_i,U'=b_j)\coloneqq P_{ij}$ for $i=1,\dots,|A|,j=1,\dots,|B|$, where $P_{ij}$ is the matrix that minimizes Equation (\ref{eqn:objective_emd}). Note that $P$ is a valid joint distribution of $U$ and $U'$. Furthermore, for any $i$ with $v^A_i>0$ it holds that $P(U'=b_j|U=a_i)=P_{ij}/v^A_i$.

\section{Data transformation based on optimal transport \label{sec:transformation}}

\subsection{Objective of the algorithm \label{sec:objective_algorithm}}

To use HSIC for an independence test of $X$ and $T$, one requires the explicit values $l(T_i,T_j)$ for $i,j=1,\dots,n$. Since for right-censored data $T_i$ may be known only in terms of a lower bound, there is no straightforward way to apply these methods to right-censored data. In this section we propose a transformation of the original, censored dataset, into a synthetic dataset consisting of $n$ observed events. The algorithm uses optimal transport and its goal is twofold: first, it should return a dataset to which we can apply a permutation test with test-statistic HSIC, and obtain correct $p$-values under the null hypothesis; second, it should return a dataset in which the dependence between $X$ and $T$ is similar to the dependence in the original dataset. Indeed, the transformation of the data is of independent interest: we use the standard permutation test with test statistic HSIC/DCOV, but other statistics could be considered. We do believe that the main value of the transformation lies in how it can be straightforwardly combined with permutation testing on the transformed data and we expect it may be difficult to use the transformation for other purposes.

\subsection{Definition of the transformation \label{sec:transformation_algorithm}}

To define the algorithm, we use the following notation. Let $A = [a_1,\dots,a_{|A|}]$ be a multiset where $a_i\in \mathbb R^d$ for $i=1,\dots,|A|$. We define $\text{Uniform}(A)=\sum_{i=1}^{|A|} \delta(a_i)/|A|$, to be the uniform distribution over the elements, where $\delta(a)$ is the Dirac measure at $a$. By $B=\text{Remove}(a,A)$ we set $B$ to be the multiset that remains when one instance of $a$ is removed from $A$, and by $B=\text{Add}(a,A)$ we set $B$ to be the multiset that consists of the elements of $A$, with an added element $a$. We further assume that we are given a sample $D=((x_i,z_i,\delta_i))_{i=1}^n$ so that $z_1<z_2<\dots<z_n$: that is, we have $n$ distinct event times and they are labeled in increasing order. For a variable $a$ we use $a\leftarrow b$ to change the value of $a$ to $b$. In computing the optimal transport coupling as defined in Section \ref{sec:background_optimal_transport}, we use the Euclidean metric $\| x-y\|=\sqrt{\sum_{i=1}^d (x_i-y_i)^2}$ for $x\in \mathbb R^d$. The proposed transformation is defined in Algorithm \ref{algorithm:transformation}.

\begin{figure}[H]
\begin{algorithm}[H]
    \SetKwInOut{Input}{Input}
    \SetKwInOut{Output}{Output}
	\Input{$D= ((x_i,z_i,\delta_i))_{i=1}^{n}$, with $z_1  < \dots < z_n.$ } 
	Define the multisets $\LL=\ar=[x_1,\dots,x_n ]$ and $\tilde D=[\ ]$. (Note $\tilde D$ is an empty multiset). Set $d=0$\;
	\For{$i=1,\dots,n-1$}{ \If{ $\delta_i=1$}{ $d\leftarrow d+1$\;
	Compute the joint distribution matrix $P \in \mathbb R^{|\ar|\times |\LL|}$ that solves the optimal transport problem between $U\sim\text{Uniform}(\ar)$ and $U'\sim\text{Uniform}(\LL)$ \;
	Condition $P$ on $U=x_i$ to get distribution vector $v\in \mathbb R^{|\LL|}$ (see Section \ref{sec:background_optimal_transport})\;
	Sample from $\LL$ with distribution $v$. Denote the chosen element by $\tilde x$  \;
	$\tilde D \leftarrow \text{Add}((\tilde x,z_i),\tilde D)$\;
	$\LL \leftarrow \text{Remove}(\tilde x,\LL)$\; }

	$\ar \leftarrow \text{Remove}(x_i,\ar)$\;}
\For{$i=d+1,\dots,n$}{ Let $\tilde x$ be any (it does not matter which) element of $\LL$.\;
$\tilde D \leftarrow \text{Add}((\tilde x,z_n),\tilde D)$\;
$\LL \leftarrow \text{Remove}(\tilde x,\LL)$ \;}
\Output{The transformed dataset $\tilde D$, consisting of $n$ points. }

    \caption{Optimal transport based transformation of $D$ to $\tilde D$.}\label{algorithm:transformation}
\end{algorithm}

\caption{The optHSIC algorithm, resulting in a transformed dataset $\tilde D$}
\end{figure}
\subsection{Comments on the transformation algorithm \label{sec:comments_to_algorithm}}
We now give a verbal explanation of Algorithm \ref{algorithm:transformation}.

\textit{Initialization:} 
The input is simply $D=((x_i,z_i,\delta_i) )_{i=1}^n$ where $z_1<\dots<z_n$. We initialize an empty transformed dataset $\tilde D$, to which we will add $n$ observations of the form $(x_i,t_j)$ for some $i,j\in \{1,\dots,n\}$ in the following way. We will loop over the times $z_i$ from $i=1,\dots,n$. At each time, the multiset $\ar$ lists the covariates at risk in the dataset $D$ and $\ar$ is initialized as the multiset containing all covariates. The multiset $\LL$ will list all covariates that have not been added to $\tilde D$ yet. Indeed, as $\tilde D$ is initialized empty, $\LL$ initially contains all $n$ covariates. The variable $d$ will count the number of observed $\delta=1$ events and is initialized 0. 

\textit{First loop from $i=1,\dots,n-1$:} At time $z_i$ we distinguish two cases: if $\delta_i=0$ we leave $\LL$ and $\tilde D$ unchanged, and simply remove one instance of $x_i$ from $\ar$. If $\delta_i=1$ we add $1$ to $d$ (to count the number of observed events). We also select a covariate $\tilde x$ from $\LL$ as follows: First a joint distribution of $\text{Uniform}(\ar)$ and $\text{Uniform}(\LL)$ is computed using optimal transport. This is a matrix $P$ of size $|\ar|\times |\LL|$. This distribution is then conditioned on the event that $\text{Uniform}(\ar)=x_i$, yielding a distribution over $\LL$. We sample from this distribution to obtain the covariate $\tilde x$. The pair $(\tilde x,z_i)$ is then added to $\tilde D$, and the covariate $\tilde x$  is removed from $\LL$. There are now $d$ observations in $\tilde D$ and there are $n-d$ covariates left in $\LL$. We also remove $x_i$ from $\ar$. 

\textit{When we finish the first loop:} It is now the case that $\ar=[x_n]$ (note the previous loop went up to $i=n-1$) and the multiset $\LL$ contains $n-d$ covariates. In the second loop, that runs from $d+1$ to $n$, each remaining covariate in $\LL$ is combined with the time $z_n$ and added to $\tilde D$. After this loop, $\LL$ is empty, and $\tilde D$ contains all $n$ covariates $x_1,\dots,x_n$, each associated with a time. The transformed dataset $\tilde D$ is returned.

Consider the special case in which there is no censoring and $\delta_i=1$ for $i=1,\dots,n$. Then it is easy to verify that in the first loop $\AR=\LL$ at each step $i$, so that the optimal transport algorithm returns a scaled identity matrix. Consequently, at step $i$ the covariate $\tilde x$ is equal to $x_i$ with probability 1, and the final output is $\tilde D=D$. A nontrivial example is worked out in Section \ref{sec:appendix:example_transformation}.

\begin{figure}
\noindent\makebox[\textwidth][c]{\includegraphics[width=0.5\textwidth]{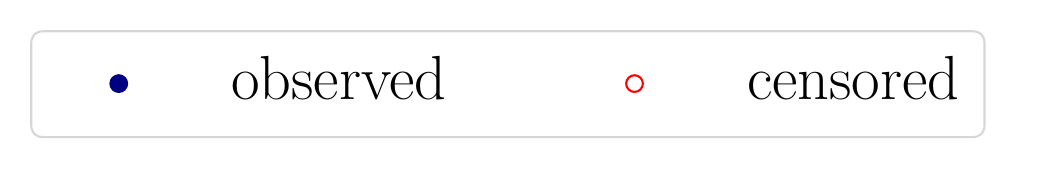}}%

\makebox[\textwidth][c]{\includegraphics[width=1\textwidth]{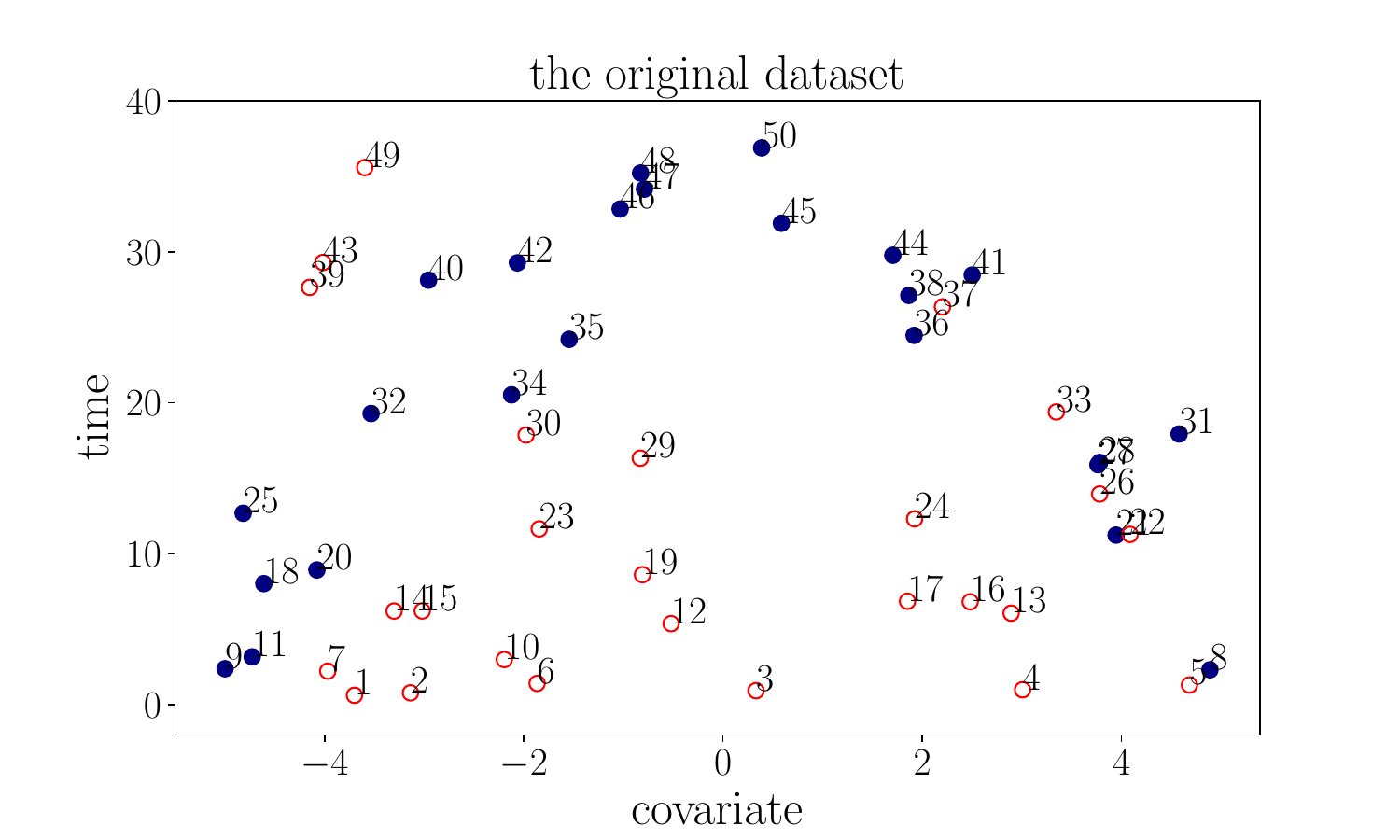}}%
\\
\makebox[\textwidth][c]{\includegraphics[width=1\textwidth]{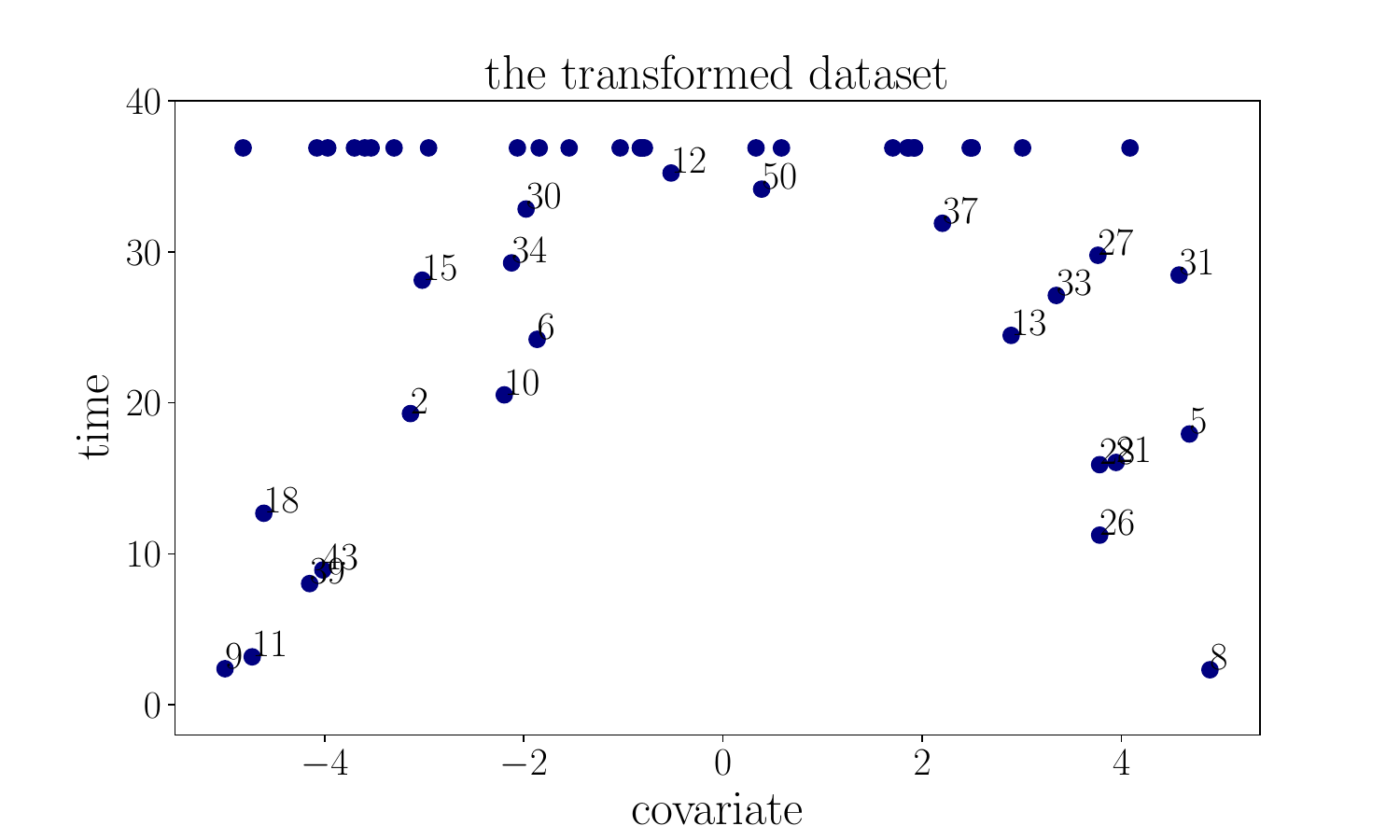}}

\caption{Above: The dataset as originally observed. Below: the synthetic dataset after applying the OPT-based transformation. The labels indicate which individual the observation corresponds to in each dataset. The data is sampled from a parabolic relationship between covariates and lifetimes.}
\label{Fig:figuretransformation}
\end{figure}

\subsection{Intuition behind the transformation \label{sec:intuition_of_algorithm}}

Before we prove properties of the proposed transformation, we briefly comment on the intuition behind the transformation. To this end, first consider a permutation test in the absence of censoring, when we simply observe $D=((x_i,t_i))_{i=1}^n$. The permuted datasets can be generated as follows: loop through the events in order of time, and to each time, associate a covariate that you have not associated to any earlier time. Comparing the original dataset with the datasets obtained in this manner is justified for the following reason: \textit{Under the null hypothesis} a sample can be generated by firstly sampling $t_i$ for $i=1,\dots n$ i.i.d. and independently sampling $x_i$ for $i=1,\dots n$ i.i.d., and secondly looping through the times in order, and associating to each time a covariate that has not yet been chosen at a previous time. The original dataset and the permuted datasets are thus equal in distribution: intuitively, the permutation test checks whether the dataset looks as if, at each time, a covariate is picked uniformly from those not chosen before.

It is not obvious how to translate this to censored data. Due to censoring, it may not be true that the $i$-th event covariate is chosen uniformly from the covariates that have not had an observed event time before. For this reason survival analysis often compares the $i$--th event covariate with the covariates at risk (not failed and not censored) just before the $i$--th event. If the null hypothesis holds, then intuitively it holds that the $i$-th event covariate is chosen uniformly \textit{from the covariates at risk} at time $z_i$. That is, if $\ar$ denotes the multiset of covariates at risk and $x_i$ is the event covariate, we would like to test if $x_i$ was chosen uniformly from $\ar$.

To do so using a permutation test, our algorithm couples $\text{Uniform}(\ar)$ to a uniform choice from those that have not yet been chosen in the synthetic dataset: $\text{Uniform}(\LL)$ (see Algorithm \ref{algorithm:transformation}). In this manner, when the null hypothesis is true, (intuitively) it holds that in the synthetic dataset the covariates are chosen uniformly from those not chosen before. But this last statement is exactly our intuition behind a permutation test: namely, we use a permutation test to see if the dataset looks as if, at each time, a covariate is picked uniformly from those not chosen before. 

The intuition discussed thus far related to the workings of our transformation under the null hypothesis, and had at its core that $\text{Uniform}(\ar)$ is coupled to $\text{Uniform}(\LL)$ in Algorithm \ref{algorithm:transformation}. However there are many couplings between these two distributions. The reason we choose the coupling that solves the optimal transport problem is the following: assume the alternative hypothesis holds and that given $\ar_i$, certain covariates have a higher hazard rate than others. Then we would like this bias towards these covariates to be visible in $\tilde D$. In other words we would like $\tilde x$ to be close to $\tilde x_i$ in Algorithm \ref{algorithm:transformation}. That is precisely what the optimal transport coupling achieves. 

Finally, including the remaining covariates in the transformed dataset at the last time ensures the permuted datasets correctly reflect all alternative choices that can be made when covariates are chosen at random from those not chosen before. The association of these remaining covariates to the last event time reflects they have not been selected at each of the earlier times, which may indicate they have had less risk of having an event up to that point.  

The goal of the transformation is thus twofold: firstly, ensuring that under the null hypothesis a permutation test is appropriate on the transformed dataset, which motivates the coupling of $\text{Uniform}(\ar)$ and $\text{Uniform}(\LL)$; secondly, given the first goal, ensure covariates in $\tilde D$ associated to times $z_i$ for $\delta_i=1$ are as close as possible to the covariates associated to $z_i$ in $D$, which motivates the chosen coupling to be the coupling that solves the optimal transport problem defined in Section \ref{sec:background_optimal_transport}. 

An interesting alternative approach would be to compare $x_i$ with $\text{Uniform}(\ar_i)$ directly, without first transforming the data. We do not see, however, how to combine that approach with a permutation test, or with another procedure to test for significance. 

\section{Applying HSIC to the transformed dataset: optHSIC \label{sec:opthsic}}

We have thus far described how to transform the dataset. For the hypothesis test of $H_0:X\independent T$, we propose to apply a permutation test with test statistic DCOV/HSIC to the transformed dataset. This approach is summarized in Algorithm \ref{algorithm:opthsic}.
 
\begin{figure}
\begin{algorithm}[H]
    \SetKwInOut{Input}{Input}
    \SetKwInOut{Output}{Output}
	\Input{ Observed data $D=((x_i,z_i,\delta_i))_{i=1}^{n}$, significance level $\alpha$, number of permutations $B$.} 
Apply Algorithm \ref{algorithm:transformation} to $D$ to obtain the transformed (and uncensored) dataset $\tilde D$.\;
Apply the standard HSIC permutation test of Algorithm \ref{hsicpermutationalgorithm} to $\tilde D$ with level $\alpha$ and number of permutations $B$, resulting in a $p$-value. \;
\Output{The $p$-value resulting from the permutation test. Reject $H_0$ if $p\leq \alpha$.}   
    \caption{optHSIC} \label{algorithm:opthsic}
\end{algorithm}
\caption{The optHSIC algorithm, resulting in a $p$-value and rejection decision of $H_0:X\independent T$.}
\end{figure}

\subsection{Computational cost of optHSIC \label{sec:computational_time}}

The algorithm optHSIC consists of two parts: First, the dataset is transformed. Second, a permutation test with test statistic HSIC/DCOV is performed. The computational cost of the transformation may be high: the solution of the earth mover's distance has $n^3\log(n)$ complexity (\cite{shirdhonkar2008approximate}). Since the earth mover's distance is computed for each $i=1,\dots,n$ such that $\delta_i=1$, complexity of the transformation is excessively high. Luckily, fast approximations of the earth mover's distance exist that are linear in time (e.g. \cite{shirdhonkar2008approximate}), making the transformation $O(n^2)$. Also in the case where the covariates are 1-dimensional, the optimal transport problem has a simple solution (\cite{cohen1999finding}). Different algorithms and/or approximations also exist for the EMD with other metrics (\cite{shirdhonkar2008approximate}). The second step, computation of HSIC is also an $O(n^2)$-operation for which large-scale approximations can be made (\cite{largescaleindependence}). Furthermore, both the computation of HSIC and the permutation test computations can  be easily parallelized. In our simulations we did not use approximations of the earth mover's distance nor of HSIC, as for samples up to about 500 the optHSIC test can be performed in about a second on an ordinary PC.

\subsection{Theoretical results on optHSIC \label{sec:opthsic_theory}}
We say a test with $p$-value $p$ has correct type 1 error rate if under the null hypothesis $P_{H_0}(p\leq \alpha)\leq \alpha$. The main theoretical result we obtained for optHSIC is that the type 1 error rate is correct when $C\independent X$. Unfortunately, we have not been able to prove other important results such as (asymptotically) correct type 1 error rate when $C\not \independent X$, or consistency (power converging to 1 for each alternative hypothesis). However, as Section \ref{sec:experiments} details, extensive simulations demonstrate that optHSIC achieves correct type 1 error rate also when $C\not \independent X$. Simulations further show that optHSIC is  able to detect a wide range of dependencies between $X$ and $T$ without losing much power, relative to the CPH likelihood ratio test, even when the CPH model assumptions holds. Obtaining further theoretical results is an important future challenge. In the uncensored case a permutation test with covariance kernel of Brownian motion is consistent for distributions with compact support and has correct type 1 error rate (\cite{rindt2020consistency}). The difficulty of extending these proofs to optHSIC is that optHSIC requires sequential analysis of the optimal transport distributions, breaking independence between observations in the transformed dataset. 

We first prove an auxiliary result: namely, although we propose to permute the transformed dataset, this is equivalent to permuting the original dataset, and then transforming the permuted datasets. 

\begin{lemma} \label{lemma:permutedtransformedvstransformpermuted}
Let $\pi_1,\dots,\pi_B$ be independent uniform random permutations, and let $T,T_1,\dots,T_B$ be independent optHSIC transformations.
\begin{align*}
\left[ T(D),\pi_1(T(D)),\dots,\pi_B(T(D)) \right]  \stackrel{d}{=} \left[ T(D),T_1(\pi_1D),\dots,T_B(\pi_BD) \right]
\end{align*}
\end{lemma}
\begin{proof}
See Appendix.
\end{proof}

Note that this implies that in the definition of optHSIC we could have equally permuted $D$ first, and then applied the transformation to each of the permuted datasets. However this is computationally more expensive. Lemma \ref{lemma:permutedtransformedvstransformpermuted} enables us to show that optHSIC has correct type 1 error rate when $C\independent X$.

\begin{theorem} \label{thm:correct_type1_opt_independent_censoring} Let $D= ((X_i,Z_i,\Delta_i))_{i=1}^n$ be an i.i.d sample. Assume that $C \independent X$. Let $p$ be the $p$-value resulting from applying optHSIC (Algorithm \ref{algorithm:opthsic}) to $D$ with $B$ permutations and level $\alpha \in [0,1]$. If the null hypothesis holds, i.e. $T\independent X$,  then $p\sim \text{Uniform}[\frac{1}{B+1},\dots,\frac{B+1}{B+1}]$ and in particular it holds that
\begin{align*}
P_{H_0}\bigl(p\leq \alpha )\leq \alpha.
\end{align*}
\end{theorem}
\begin{proof}
See Appendix.
\end{proof}

Table \ref{table:overview_type_1_error} summarizes our findings of the type 1 error of optHSIC.

\begin{table}[H] 
\begin{tabular}{@{}llll@{}}
\toprule
                       & \begin{tabular}[c]{@{}l@{}} \textbf{Correct type}\\  \textbf{1 error rate}\\ \textbf{in theory:}\\ $P_{H_0}(p\leq \alpha)\leq \alpha.$ \end{tabular}    & \begin{tabular}[c]{@{}l@{}} \textbf{Distribution $p$} \\ \textbf{under $H_0$} \\ \\ \\\end{tabular}                                                                                                                                     & \begin{tabular}[c]{@{}l@{}} \textbf{Correct type 1} \\ \textbf{error rate}\\ \textbf{in simulations }\\ \\ \end{tabular} \\ \midrule
 \begin{tabular}[c]{@{}l@{}}$ C\independent X$ \\ \\ \end{tabular}    & \begin{tabular}[c]{@{}l@{}}Yes \\ (Theorem  \ref{thm:correct_type1_opt_independent_censoring})\end{tabular} & \begin{tabular}[c]{@{}l@{}}$\text{Uniform}[\frac{1}{B+1},\dots,\frac{B+1}{B+1}]$\\ (Theorem \ref{thm:correct_type1_opt_independent_censoring})\end{tabular}                                                                                   &  \begin{tabular}[c]{@{}l@{}} Yes \\ (Sec \ref{sec:type1error}) \\ \end{tabular}                                                                                     \\ \midrule
\begin{tabular}[c]{@{}l@{}}$ C \not \independent X$ \\ \\ \\ \\  \end{tabular} & \begin{tabular}[c]{@{}l@{}} Unknown \\ \\ \\ \\  \end{tabular}                                                                                                                                                                 & \begin{tabular}[c]{@{}l@{}}Unkown, but experiments \\ suggest approximately\\  $\text{Uniform}[\frac{1}{B+1},\dots,\frac{B+1}{B+1}]$ \\ (Figure \ref{fig:lin_censoring_uniform_pvalues} Appendix)\end{tabular} & \begin{tabular}[c]{@{}l@{}} Yes \\ (Sec \ref{sec:type1error}) \\ \\ \\  \end{tabular}                                                                                     \\ \bottomrule
\end{tabular}
\caption{An overview of the type 1 error results obtained for optHSIC in the case when $C\independent X$ and when $C\not \independent X$. $B$ denotes the number of permutations. The $p$-value is defined in Algorithm  \ref{algorithm:opthsic}. \label{table:overview_type_1_error}}
\end{table}

\subsection{Using multiple transformations \label{sec:multiple_transformations}}

Algorithm \ref{algorithm:transformation} defines the transformation used in optHSIC. In line 7, a covariate is sampled from the multiset $\LL$ with distribution $v$. The sampled element may differ for different iterations of the algorithm. Consequently, given a fixed dataset $D$, the transformed dataset will look different for different iterations of the algorithm. The optHSIC algorithm proposes to use only one transformation, resulting in a single $p$-value. 

Say that instead of applying the optHSIC algorithm once, we repeat the optHSIC algorithm $m$ times, resulting in $m$ $p$-values $p_1,\dots,p_m$. An example of the distribution of $p$-values for a fixed dataset is given in Figure \ref{fig:dist_p_values_given_data}.

\begin{figure}[H]
\centering
\begin{subfigure}{.8\textwidth}
  \centering
\includegraphics[width=1\linewidth]{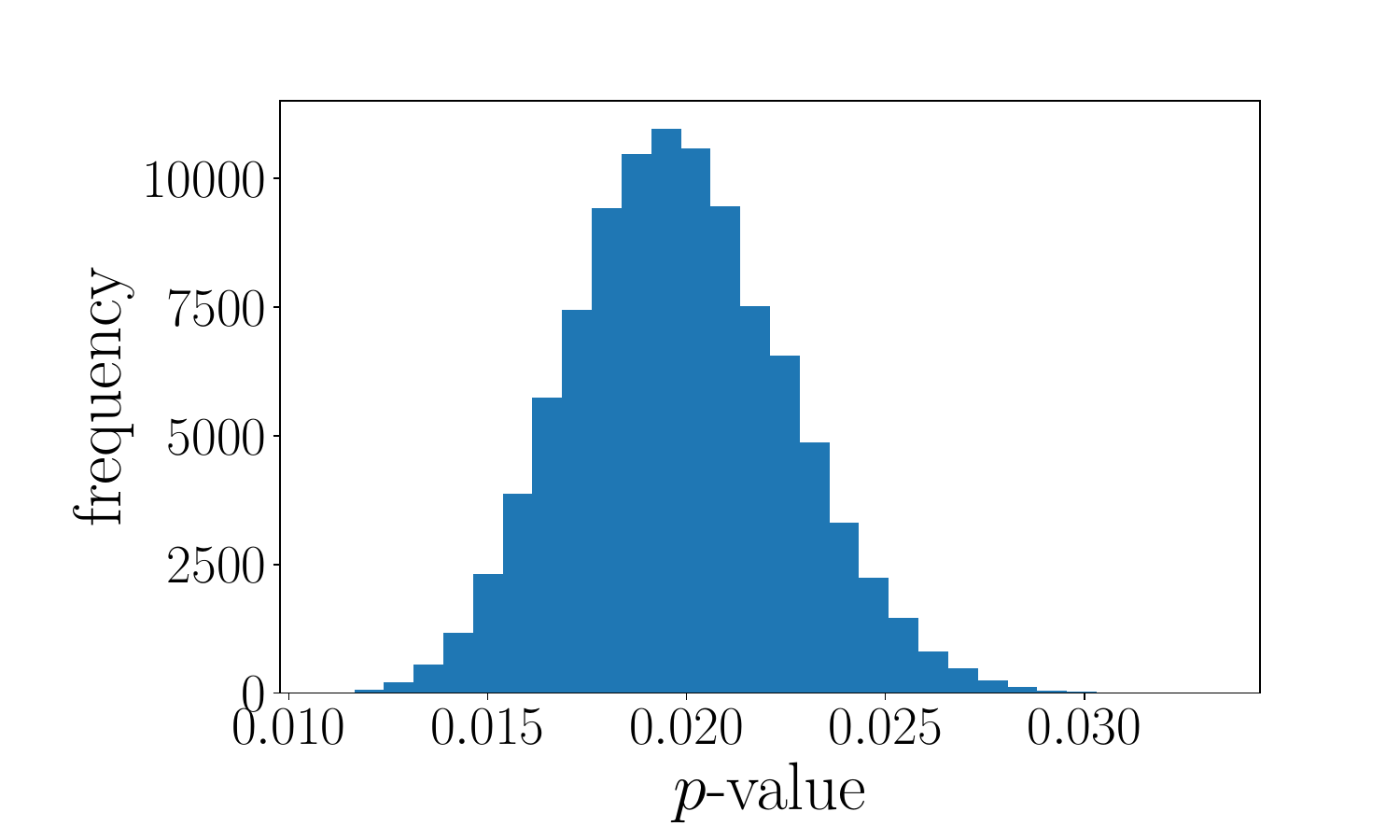}
\end{subfigure}%
\caption{\label{fig:dist_p_values_given_data} A histogram of the $p$-values obtained from optHSIC for a fixed dataset $D$ sampled from D.1 of Table \ref{table:scenarios_power}. The sample was of size $n=300$ and $61\%$ of the individuals was observed $(\delta=1)$. The variability in $p$ comes from the fact that the transformation in optHSIC is random. }
\end{figure}

Figure \ref{fig:dist_p_values_given_data} shows that in the used dataset, where $39\%$ of observations is censored and $n=300$, the variance in the obtained $p$-values is not excessively high, and in particular all $p$-values would have led to the decision to reject the null-hypothesis. 

We also studied several techniques one may use to combine $p$-values, typically at the cost of being more conservative when there is little variance among $p$-values, like in the example above. This is discussed in Section \ref{sec:appendix:multiple_transformations}.

\section{Alternative approaches when censoring is independent of the covariate \label{sec:alternative_methods_when_censoring_independent}}
We are not aware of fully nonparametric methods to test independence between right-censored times and continuous covariates. When $C$ is independent of $X$ the challenge is mitigated, since in that case any statistic can be combined with a permutation test that permutes the covariates (see Theorem \ref{thm:permutation_test_if_independent_censoring}). Using that approach, this Section proposes some alternative tests in the case when $C\independent X$ that will serve as benchmarks for studying power performance of optHSIC in addition to the Cox proportional hazards likelihood ratio test. In Section \ref{sec:experiments} we compare the performance of these methods with optHSIC.\\ 
Before we propose alternative test statistics, we state formally why standard permutation tests can be used when $C \independent X$. We first state a Theorem on the correctness of the type 1 error rate of permutation tests without censoring. We include the proof following \cite{berrettinformation}  in the Appendix for completeness.
 \begin{theorem}\label{thm:standard_permutation_test} Let $D= (X_i,Y_i) _{i=1}^n$ be an i.i.d. sample of size $n$ from distribution $P_{XY}$ on $\mathcal X \times \mathcal Y$. Denote $\pi D =(X_{\pi(i)},Y_i) _{i=1}^n$. Let $\pi_1,\dots\pi_B$ be permutations sampled uniformly and independently from $S_n$. Let $H$ be any statistic of the data. Let $R$ be the rank of the first coordinate in the vector:
\begin{equation*}
\bigl(H(D),H(\pi_1D),\dots,H(\pi_BD) \bigr)
\end{equation*}
when ties are broken at random and where $R=1$ is the rank of the largest element and $R=B+1$ the rank of the smallest element.  Define $p\coloneqq R/(B+1)$. Then under the null hypothesis $H_0:X \independent T$ it holds that: $p\sim\text{Uniform}[\frac{1}{B+1},\dots,\frac{B+1}{B+1}]$. So in particular for $\alpha \in [0,1]$
\begin{align*}
P_{H_0}(p\leq \alpha )\leq \alpha.
\end{align*}
\end{theorem}
\begin{proof}
See Appendix.
\end{proof}

In survival analysis we apply the above, setting $Y=(Z,\Delta)$, where $Z=\min\{T,C \}$ and $\Delta=1\{T\leq C\}$.
\begin{theorem}\label{thm:permutation_test_if_independent_censoring}
	Assume $C \independent X$ and $T \independent C \vert X$. Write $D=(X_i,Z_i,\Delta_i) _{i=1}^n$, 
	and let $\pi_1,\dots,\pi_B$ be sampled i.i.d.\ uniformly from $S_n$. Write
	$\pi D:=(X_{\pi(i)},Z_i,\Delta_i)_{i=1}^n$. Let $H(D)$ be any statistic of the data. Let $R$ be the rank of the first coordinate in the vector:
\begin{equation*}
\bigl(H(D),H(\pi_1D),\dots,H(\pi_BD) \bigr)
\end{equation*}
where ties are broken at random and $R=1$ is the rank of the largest element and $R=B+1$ is the rank of the smallest element. Define $p\coloneqq R/(B+1)$. Then under the null hypothesis $H_0:X \independent T$ it holds that $p\sim \text{Uniform}[\frac{1}{B+1},\dots,\frac{B+1}{B+1}]$. So in particular for $\alpha \in [0,1]$
\begin{align*}
P_{H_0}(p\leq \alpha )\leq \alpha.
\end{align*}
\end{theorem}
\begin{proof}
See Appendix.
\end{proof}

Having shown that we can use a permutation test on any test statistic when $C \independent X$, the next two sections explore meaningful measures of dependency on right-censored data. Indeed such methods may lead to type 1 errors when $C\not \independent X$.

\subsection{wHSIC}\label{sectionwhsic}

HSIC relies on estimating the mean embedding of the joint distribution $P_{XT}$. The estimated embedding is then compared to the embedding of the product of marginal distributions. In the uncensored case the empirical distribution equals
$
 \sum_{i=1}^n \delta_{(x_i,t_i)}/n
$
which has corresponding mean embedding
$
\sum_{i=1}^n K((x_i,t_i),\cdot)/n
$.

Since we do not observe the $t_i$'s we could consider replacing the empirical distribution by a weighted version $\sum_{i=1}^n w_i \delta_{(x_i,z_i)}/n$
where we try to find weights $w_i$ such that
$\sum_{i=1}^n w_i f(x_i,z_i) \approx Ef(X,T)
$ for every measurable function $f$ such that the expectation exists. A natural idea is to give an observed point $(x,z,\delta)$ a weight of zero if it is censored, and a weight equal to the inverse probability of being uncensored otherwise. This can be motivated by the following lemma, that applies also if $C \not \independent X$. We first define
$$g(t,x)= P(C>t|X=x).$$
The following lemma proposes a weight function $W$ in terms of $g$.
\begin{lemma} \label{lemmaexpectationweight} Let $f$ be an $(X,T)$- measurable function such that $Ef(X,T)< \infty$ and define $W$ by
\begin{equation*}W= \begin{cases}  0 & \text{if}  \ \Delta=0 \\ \frac{1}{g(Z,X)} & \text{if}  \ \Delta=1. \end{cases} \end{equation*} Then  $E Wf(X,Z)=Ef(X,T).$
\end{lemma}
\begin{proof}
See Appendix. \end{proof}

We would thus like to use weights $ w_i= 0 \ \text{if} \ \delta_i=0$ and $ 1/{ng(x_i,z_i)} \ \text{if}  \ \delta=1.$
As we will be working under the assumption $C\independent X$ we may estimate ${1/P(C>t \vert X)}={1/P(C>t)}$ using Kaplan--Meier weights, which we define now. Assume that there are no ties in the data and that $z_i<z_{i+1}$ for $i=1,\dots,n$. Then the Kaplan--Meier survival curve is given by:
\begin{align*}
\hat S_T(z_k)=\prod_{i=1}^k \left( \frac{n-i}{n-i+1} \right)^{\delta_i}.
\end{align*}
Kaplan--Meier weights are then defined by
\begin{align*}
w_k=\hat P(T=z_k) & = \hat S_T(z_{k-1})-\hat S_T(z_k).
\end{align*}
One can also estimate the survival probability of the censoring time using Kaplan--Meier weights (simply replace $\delta_i$ by $1-\delta_i$ in the above formula). The following lemma shows that the  weights $w_k$ defined above, correspond up to a constant of $1/n$, to the inverse of the estimated probability of being uncensored.
\begin{lemma} \label{lemmakaplanmeier}
Let $ \hat S_C(z_k)=\hat P(C > z_k)$ denote the Kaplan--Meier estimate of the survival probability of the censoring distribution. Then:

\begin{align*}
w_k=\frac{1}{n} \frac{1}{\hat P(C > z_k)}
\end{align*}
\end{lemma}
\begin{proof}
See Appendix.
\end{proof}

We use these weights to define the statistic wHSIC as the RKHS distance between the embedding of 
$ \sum_{i=1}^n w_i \delta_{(x_i,z_i)}
 $
and the embedding of the product of the marginals,
$ \sum_{i=1}^n\sum_{j=1}^n w_iw_j\delta_{x_i}\delta_{z_j}.
$ 
That is:
\begin{align*}
\text{wHSIC}(D) \coloneqq \bigg \vert  \bigg \vert \sum_{i=1}^{n}w_{i}K\bigl((x_i,z_i),\cdot \bigr)-\sum_{i=1}^{n}\sum_{j=1}^n w_i w_j K \bigl( (x_i,y_j),\cdot \bigr) \bigg \vert \bigg \vert ^2 _{H}.
\end{align*}
Here $K\bigl((x,y),(x',y')\bigr)=k(x,x')l(y,y')$. The following theorem shows how to compute $\whsic$ efficiently.
\begin{theorem}{ (Computation of wHSIC)}\label{fastcomputationhsic}
Given a dataset $D=\left( (x_i,y_i) \right)_{i=1}^n$ with a weight vector $w\in \mathbb R^n$,
\begin{equation*}
\whsic(D)=\tr\left(H_{w}KH_{w}L\right) ,
\end{equation*}
where $K_{ij}=k(x_i,x_j)$ and $L_{ij}=l(y_i,y_j)$ and $H_w=D_w-ww^{\top}$, where $D_w=\text{diag}(w)$, a diagonal matrix with $(D_w)_{ii}=w_i$.
\end{theorem}
\begin{proof}
See Appendix.
\end{proof}
It is not difficult to see wHSIC has the same computational time as (uncensored) HSIC. \\

\begin{figure}
\begin{algorithm}[H]
    \SetKwInOut{Input}{Input}
    \SetKwInOut{Output}{Output}
	\Input{  $D=((x_i,z_i,\delta_i))_{i=1}^{n}$, significance level $\alpha$, number of permutations $B$. } 
Sample permutations $\pi_1,\dots,\pi_B$ i.i.d. uniformly from $S_n$.\;
Breaking ties at random, compute the rank $R$ of $\whsic(D)$ in the vector $$\left(\whsic(D),\whsic (\pi_1 D),\whsic (\pi_2 D),\dots,\whsic (\pi_{B} D) \right)$$
where $R=1$ is the rank of the largest element. \;
\Output{The $p$-value $p\coloneqq R/(B+1)$. Reject $H_0$ if $p\leq \alpha$. }
    \caption{wHSIC}
\end{algorithm}
\caption{The wHSIC algorithm, resulting in a $p$-value and rejection decision of $H_0:X\independent T$.}
\end{figure}

\subsection{zHSIC}\label{sectionzhsic}

If $C \independent X$, then any dependence between $X$ and $Z$ must be due to dependence between $X$ and $T$. Hence we may estimate $\hsic(X,Z)$ to measure the strength of the dependence. This test is, indeed, expected to yield false rejections if  $C \independent X$ fails to hold, and to lack power when a large portion of the events is censored. We include this test mainly to see how the power of optHSIC and wHSIC compare with this approach. \\

\begin{figure}
\begin{algorithm}[H]
    \SetKwInOut{Input}{Input}
    \SetKwInOut{Output}{Output}
	\Input{  $ D=((x_i,z_i,\delta_i))_{i=1}^{n}$, significance level $\alpha$, number of permutations $B$.}
Denote  $\tilde D=((x_i,z_i))_{i=1}^n$ the sample as if there was no censoring. Sample permutations $\pi_1,\dots,\pi_B$ i.i.d. uniformly from $S_n$.\;
Breaking ties at random, compute the rank $R$ of $\hsic(D)$ in the vector $$\left(\hsic(\tilde D),\hsic (\pi_1 \tilde D),\hsic (\pi_2 \tilde D),\dots,\hsic (\pi_{B} \tilde D) \right)$$
where $R=1$ is the rank of the largest element. \;

	\Output{The $p$-value $p\coloneqq R/(B+1)$. Reject $H_0$ if $p\leq \alpha$.}
    \caption{zHSIC}
\end{algorithm}
\caption{The zHSIC algorithm, resulting in a $p$-value and rejection decision of $H_0:X\independent T$.}
\end{figure}

\section{Numerical evaluation of the methods\label{sec:experiments}}

We generate data from various distributions of $X$, $T$ and $C$ to compare the power and type 1 error rate of optHSIC, wHSIC, zHSIC and CPH. CPH stands for the Cox proportional hazards likelihood ratio test. In each scenario, we let the sample size range from $n=40$ to $n=400$ in intervals of $40$. To obtain $p$-values in the three HSIC-based methods we use a permutation test with 1999 permutations. We reject the null hypothesis if our obtained $p$-value is less than $0.05$. In the main paper, we present results of rejection rates under distributions that are chosen such that approximately $60\%$ of the observations are observed ($\delta=1$). We investigate the rejection rates under varying censoring regimes in Section \ref{sec:appendix:varying_censoring} of the Appendix. 

\subsection{Type 1 error rate \label{sec:type1error}}
We begin by investigating the type 1 error rate. The distributions in which we test the type 1 error rate are found in Table \ref{table:scenarios_type1error}. In these distributions the null hypothesis holds, i.e. $X\independent T$ and we consider both cases where $C\independent X$ and $C\not \independent X$, as well as the case of multidimensional covariates. We estimate the rejection rates by sampling 5000 times from each distribution for each sample size and applying the different tests to the samples. The obtained rejection rates of optHSIC are found in Table \ref{table:type1error_opthsic}. The type 1 error rate of the remaining methods is found in Section \ref{sec:appendix:type1error}. The most important finding is that optHSIC is found to have correct type 1 error rate of $\alpha=0.05$ both when $C \independent X$ and when $C\not \independent X$. In contrast, wHSIC and zHSIC yield many false rejections when $C \not \independent X$, as expected, but have the correct type 1 error rate when $C\independent X$. Investigation of the $p$-values of optHSIC furthermore showed that $p$-values are distributed approximately according to $\text{Uniform}[\frac{1}{B+1},\dots,\frac{B+1}{B+1}]$ under the null hypothesis, even when $C\not \independent X$ (see Figure \ref{fig:lin_censoring_uniform_pvalues} in Section \ref{sec:appendix:type1error}).

\begin{table}[H]
\begin{center}
    \begin{tabular}{  l  l  l  l  l   p{5cm} }
 
    D. &  X & $ T \vert X$ & $C \vert X$ & $\% \ \delta=1$ \\ \hline
    1   & $ \text{Unif}[-1,1]$ & $ \text{Unif}[0,1]$ & $ \text{Unif}[0,1.5]$ & $66\%$\\ 
	    2  &$ \text{Unif}[-1,1]$ & $\text{Exp}(\text{mean}=5/2)$ &  $\text{Exp}(\text{mean}=5/3) $ & $40 \%$ \\ 
    3  & $ \text{Unif}[-1,1]$  &  $  \text{Exp}(\text{mean}=2/3) $ & $\text{Exp}(\text{mean}=\exp(X))$ & $60\%$   \\ 
    4  & $ \text{Unif}[-1,1]$  &  $  \text{Exp}(\text{mean}=1.6) $ & $\text{Exp}(\text{mean}=\exp(9X^2))$  & $60\%$  \\ 
    5  & $ \text{Unif}[-1,1]$   &  $  \text{Exp}(\text{mean}=0.9) $ & $\text{Weib}(\text{shape}=1.75X+3.25)$  & $60\%$ \\ 
            6  & $ \text{Unif}[-1,1]$ & $\text{Exp}(\text{mean}=0.9)$  & $1+X$  & $60\%$ \\ 
                7  & $ N_{10}(0,\Sigma)$  &$\text{Exp}(\text{mean}=0.6)$   & $\text{Exp}(\text{mean}=\exp(1^TX))$ & $60\%$ \\ 
                    8  & $N_{10}(0,\Sigma)$  & $\text{Exp}(\text{mean}=0.6)$   &  $\text{Exp}(\text{mean}=\exp(X_1/8))$ & $60\%$  \\ 
    \end{tabular}\\
\end{center}    
\caption{The 8 scenarios in which we simulate the type 1 error rate. $N_{10}(0,\Sigma)$ is a 10-dimensional Gaussian random variable, with $0$ denoting a vector of zeros of length 10, and $\Sigma=MM^T$ where $M$ is a $10\times 10$ matrix of i.i.d $N(0,1)$ entries. Note that $C$ depends on $X$ in D.3-8. \label{table:scenarios_type1error}}
\end{table}
\begin{table}[H]
\resizebox{\textwidth}{!}{
\begin{tabular}{@{}lllllllllll@{}}
\toprule
 $n=$   & $40$ & $80$ & $120$ & $160$ & $200$ & $240$ & $280$ & $320$ & $360$ & $400$ \\ \midrule
D.1 &0.047 & 0.053 & 0.051 & 0.050 & 0.054 & 0.051 & 0.045 & 0.056 & 0.049 & 0.048 \\ D.2 &
0.049 & 0.051 & 0.057 & 0.044 & 0.053 & 0.050 & 0.047 & 0.049 & 0.052 & 0.048 \\ D.3 &
0.052 & 0.052 & 0.048 & 0.049 & 0.044 & 0.050 & 0.048 & 0.054 & 0.054 & 0.048 \\ D.4 &
0.050 & 0.054 & 0.049 & 0.054 & 0.054 & 0.054 & 0.050 & 0.053 & 0.050 & 0.046 \\ D.5 &
0.050 & 0.056 & 0.056 & 0.051 & 0.053 & 0.051 & 0.046 & 0.055 & 0.050 & 0.049 \\ D.6 &
0.050 & 0.049 & 0.048 & 0.047 & 0.048 & 0.057 & 0.050 & 0.050 & 0.048 & 0.050 \\ D.7 &
0.050 & 0.055 & 0.047 & 0.048 & 0.056 & 0.054 & 0.054 & 0.052 & 0.051 & 0.048 \\ D.8 &
0.041 & 0.050 & 0.051 & 0.048 & 0.056 & 0.049 & 0.054 & 0.050 & 0.054 & 0.052 \\ \bottomrule
\end{tabular}}
\caption{The rejection rate of optHSIC for the distributions D.1-8 of Table \ref{table:scenarios_type1error}. Note the type 1 error rate is very close to the level $\alpha=0.05$ as desired, also in the scenarios where $C\not \independent X$. \label{table:type1error_opthsic}}
\end{table}

\subsection{Comparison of power}\label{sec:power}

To compare the power of the four methods we consider a number of distributions in which $T\not \independent X$. For each distribution we let the sample size range from $n=40$ to $n=400$ in steps of $40$. At each sample size we take 1000 samples to estimate the rejection rate. In these distributions censoring is independent of the covariate such that the methods wHSIC and zHSIC do not have inflated rejection rate due to dependencies between $C$ and $X$. Distributions with varying censoring rates and dependent censoring are investigated in Section \ref{sec:appendix:varying_censoring}. Parameters are chosen so that that $60\%$ of the observations is uncensored ($60\%$ has $\delta_i=1$), and rejection rates take a range of values (i.e. to exclude trivial distributions so that each method rejects with probability 1 for each sample size.). 

\begin{table}[H]
\begin{center}
    \begin{tabular}{  l  l  l  l  l   p{5cm} }
 
    D. &  X & $ T \vert X$ & $C \vert X$  \\ \hline
    1   & $  N(0,1) $ & $ \text{Exp}(\text{mean}=\exp(X/6)) $  & $ \text{Exp}(\text{mean}=1.5) $ \\ 
	    2  & $ N(0,1) $ & $\text{Exp}(\text{mean}=\exp(X^2/5)$ &  $\text{Exp}(\text{mean}=2.25) $ \\ 
    3  & $ \text{Unif}[-1,1]$  & $  \text{Weib}(\text{shape}=1.75X+3.25)$  &  $  \text{Exp}(\text{mean}=1.75) $ \\ 
    4  & $ \text{Unif}[-1,1]$  & $N(100-X,2X+5.5)$  &  $ 82+\text{Exp}(\text{mean}=35)$ \\ 
        5  & $ N_{10}(0,\Sigma)$  & $\text{Exp}(\text{mean}=\exp(1^T X/30))$  &  $ \text{Exp}(\text{mean}=1.5)$ \\ 
            6  & $ N_{10}(0,\Sigma)$  & $\text{Exp}(\text{mean}=\exp(X_1/8))$  & $\text{Exp}(\text{mean}=1.5)$  \\ 
                7  & $ N_{10}(0,\Sigma)$  &$\text{Exp}(\text{mean}=\exp(X^2_1/4)$   & $\text{Exp}(\text{mean}=10)$ \\ 
                    8  & $N_{10}(0,\Sigma)$  & $\text{Exp}(\text{mean}=\exp(X^2_1/4+X_2/7))$   &  $\text{Exp}(\text{mean}=8)$ \\ 
    \end{tabular}\\
\end{center}    
\caption{The 8 scenarios in which we test the power in the main text.  $N_{10}(0,\Sigma)$ is a 10-dimensional Gaussian random variable, with $0$ denoting a vector of zeros of length 10, and $\Sigma=MM^T$ where $M$ is a $10\times 10$ matrix of i.i.d $N(0,1)$ entries. In these distributions about $60\%$ of the individuals is observed. Distributions with dependent censoring and varying censoring percentages are discussed in Section \ref{sec:appendix:varying_censoring}.\label{table:scenarios_power}}
\end{table}

\subsubsection{Power for 1-dimensional covariates}

In distributions 1-4 of Table \ref{table:scenarios_power} the covariate is 1-dimensional. Scatterplots of the samples and rejection rates are displayed in Figure \ref{fig:power_1dimensional}. Note how in D.1 the CPH assumption holds, so the CPH method suits this example very well. We find however that the rejection rate of optHSIC is very similar to that of the CPH likelihood ratio test (first row, right of Figure \ref{fig:power_1dimensional}). D.2 features a case in which hazard is highest in the middle, and lower for extreme values of the covariate. The CPH likelihood ratio test does not have power to detect this relationship, and optHSIC is the most powerful method. In D.3 and D.4 the hazard functions of different covariates cross each other. In this case, wHSIC is the top-performing method, but optHSIC is also able to detect the more complicated relationship, while CPH is not able to do so.

\begin{figure}[H]
\centering
\begin{subfigure}{.5\textwidth}
  \centering
\includegraphics[width=1.1\linewidth]{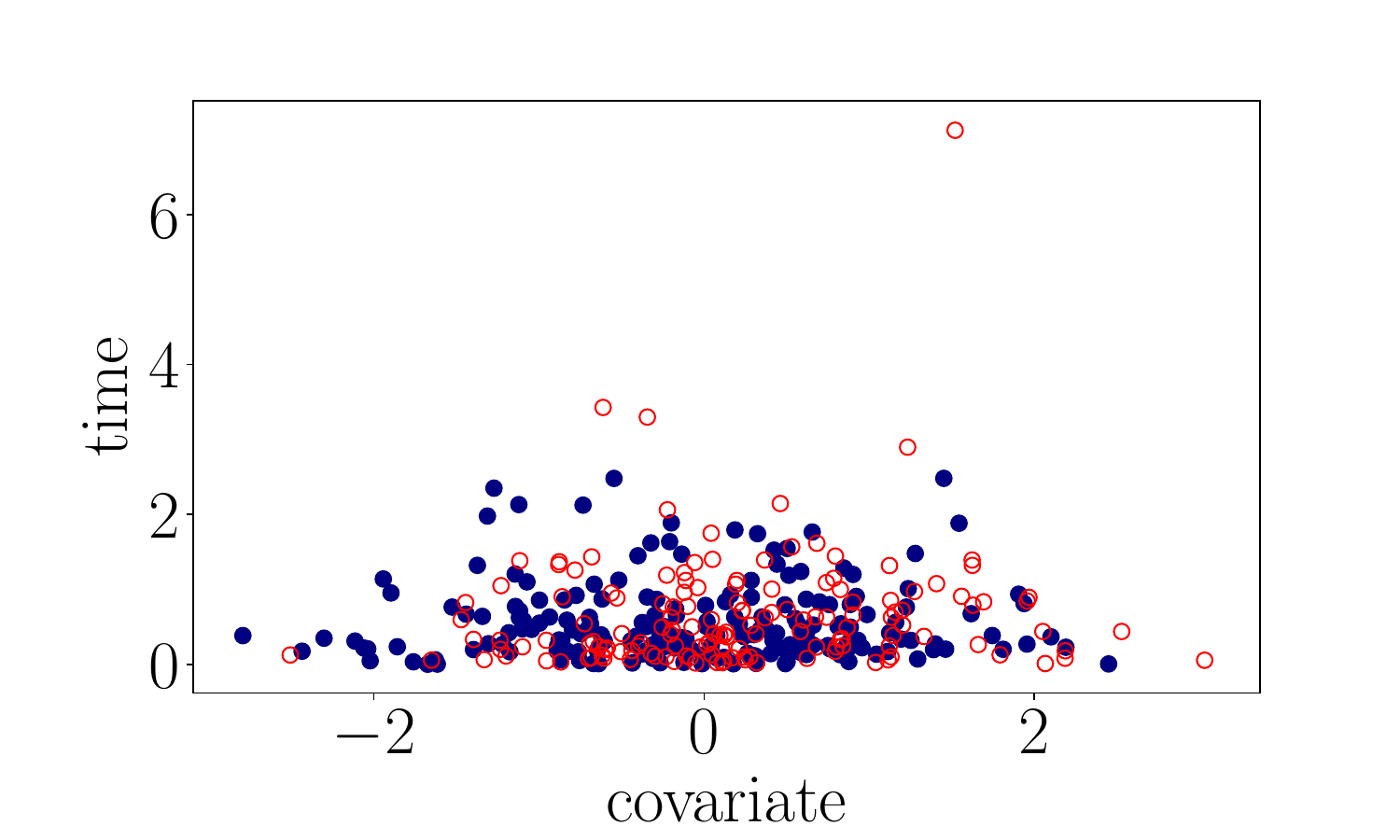}
\includegraphics[width=1.1\linewidth]{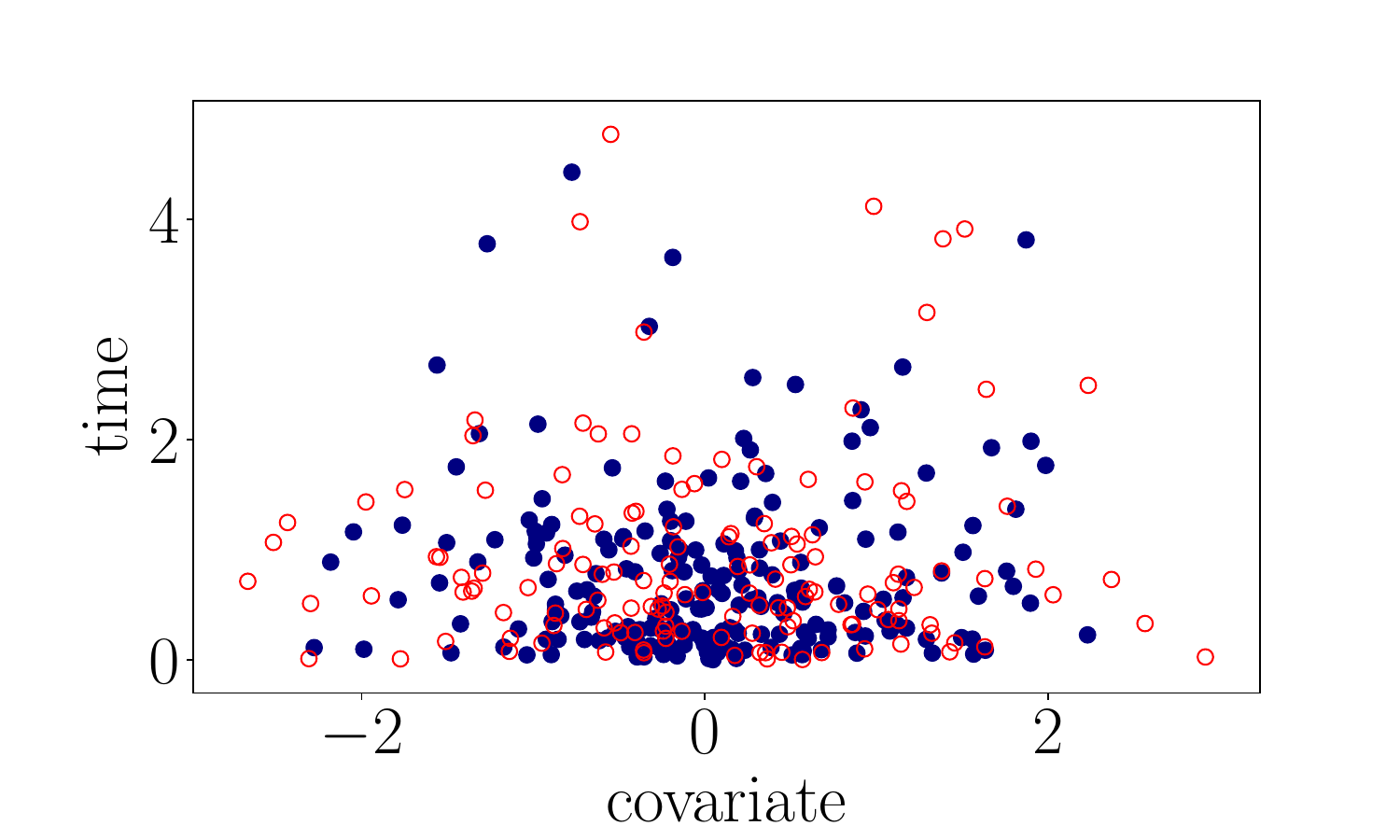}
\includegraphics[width=1.1\linewidth]{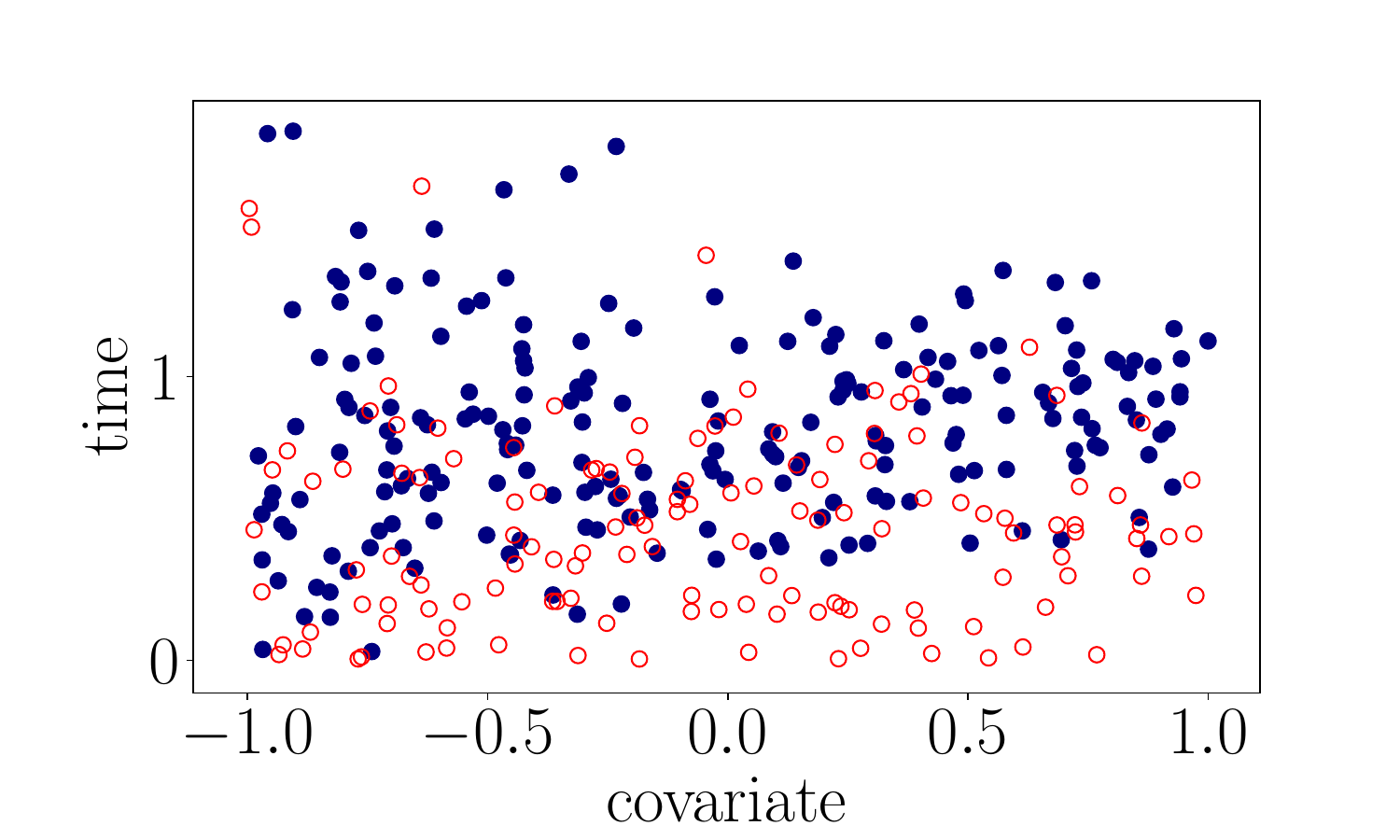}
\includegraphics[width=1.1\linewidth]{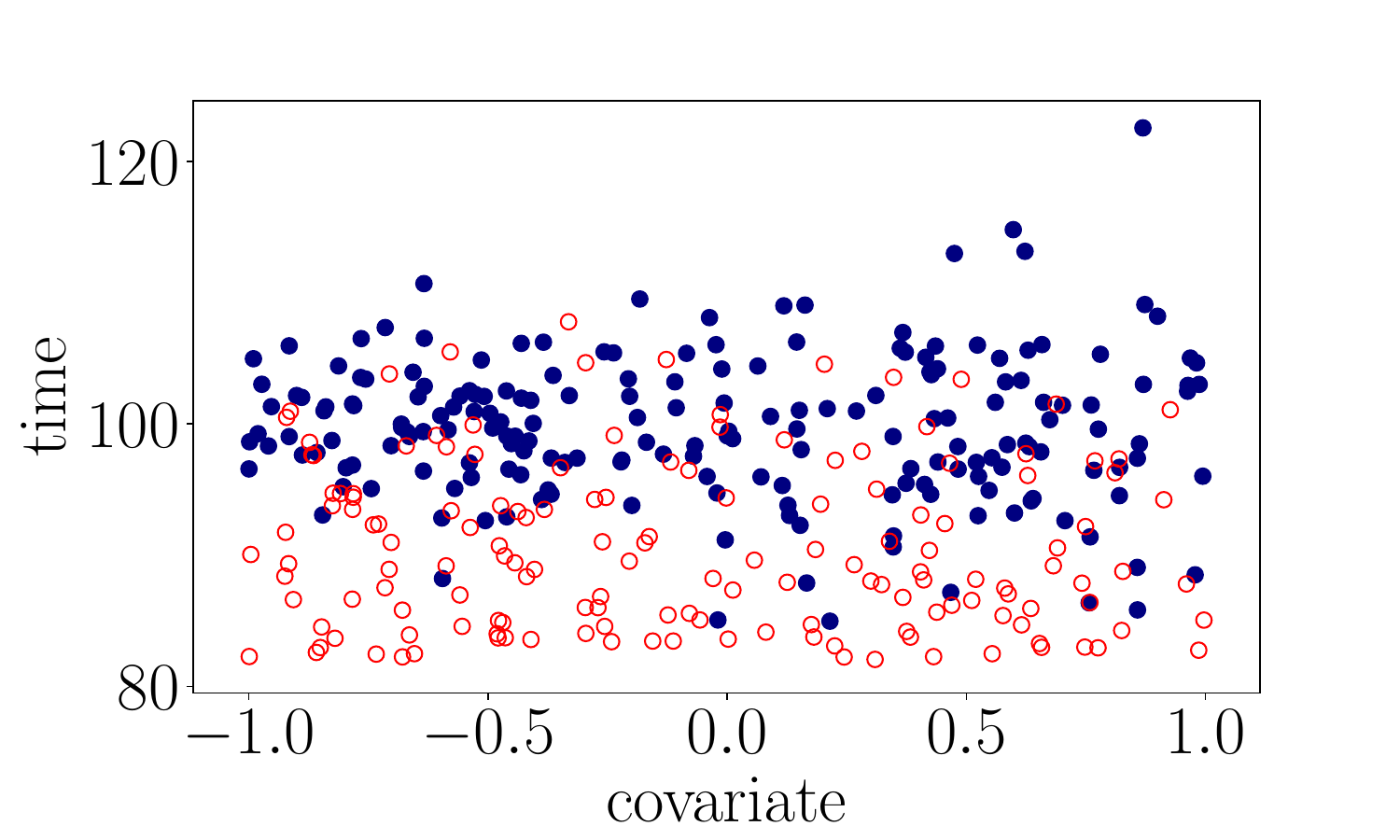}
\includegraphics[width=0.5\linewidth]{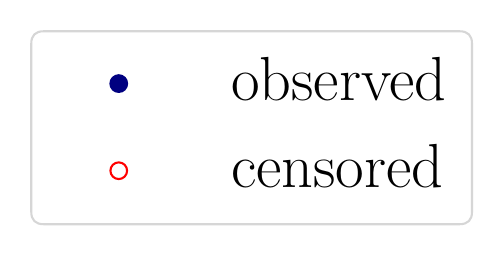}
\end{subfigure}%
\begin{subfigure}{.5\textwidth}
  \centering
\includegraphics[width=1.1\linewidth]{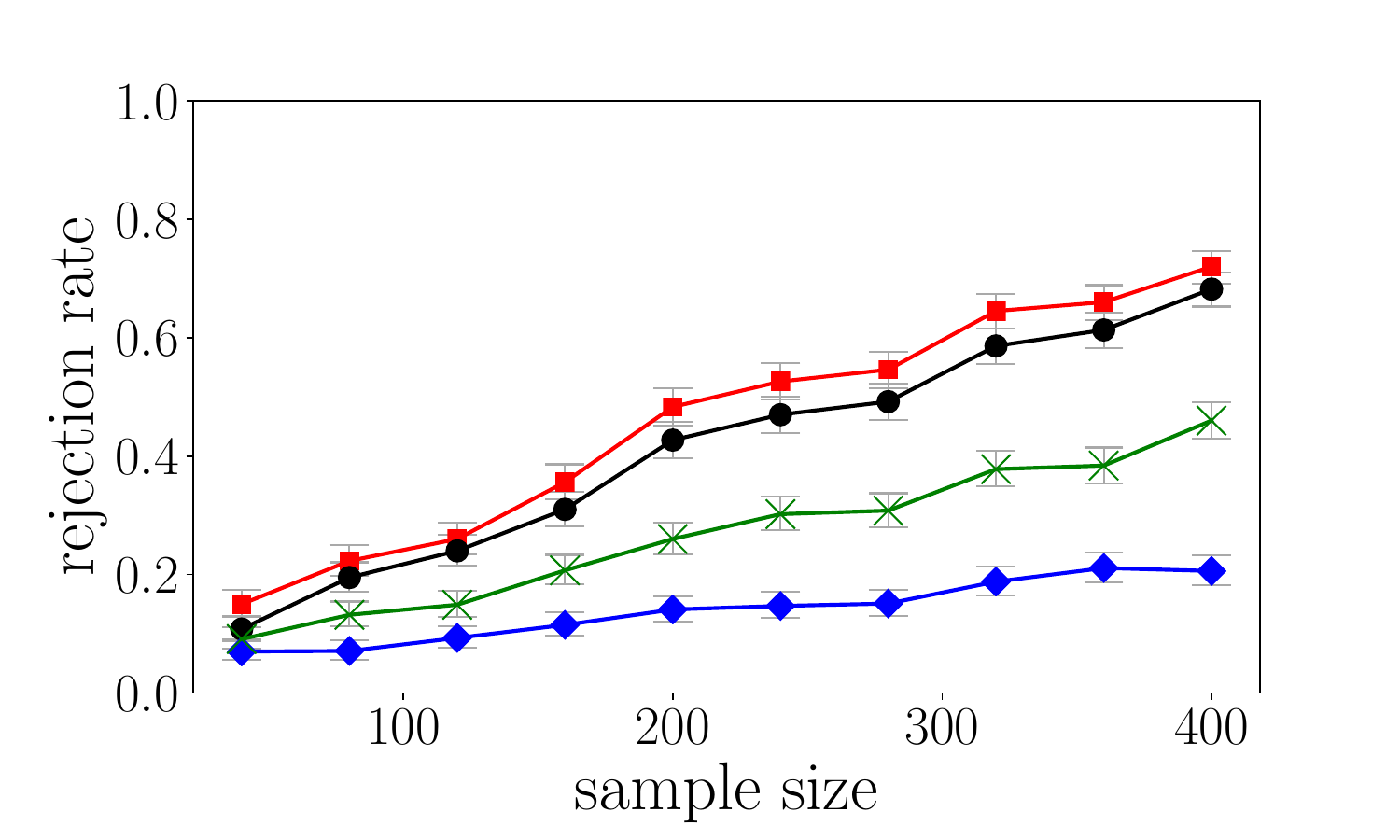}
\includegraphics[width=1.1\linewidth]{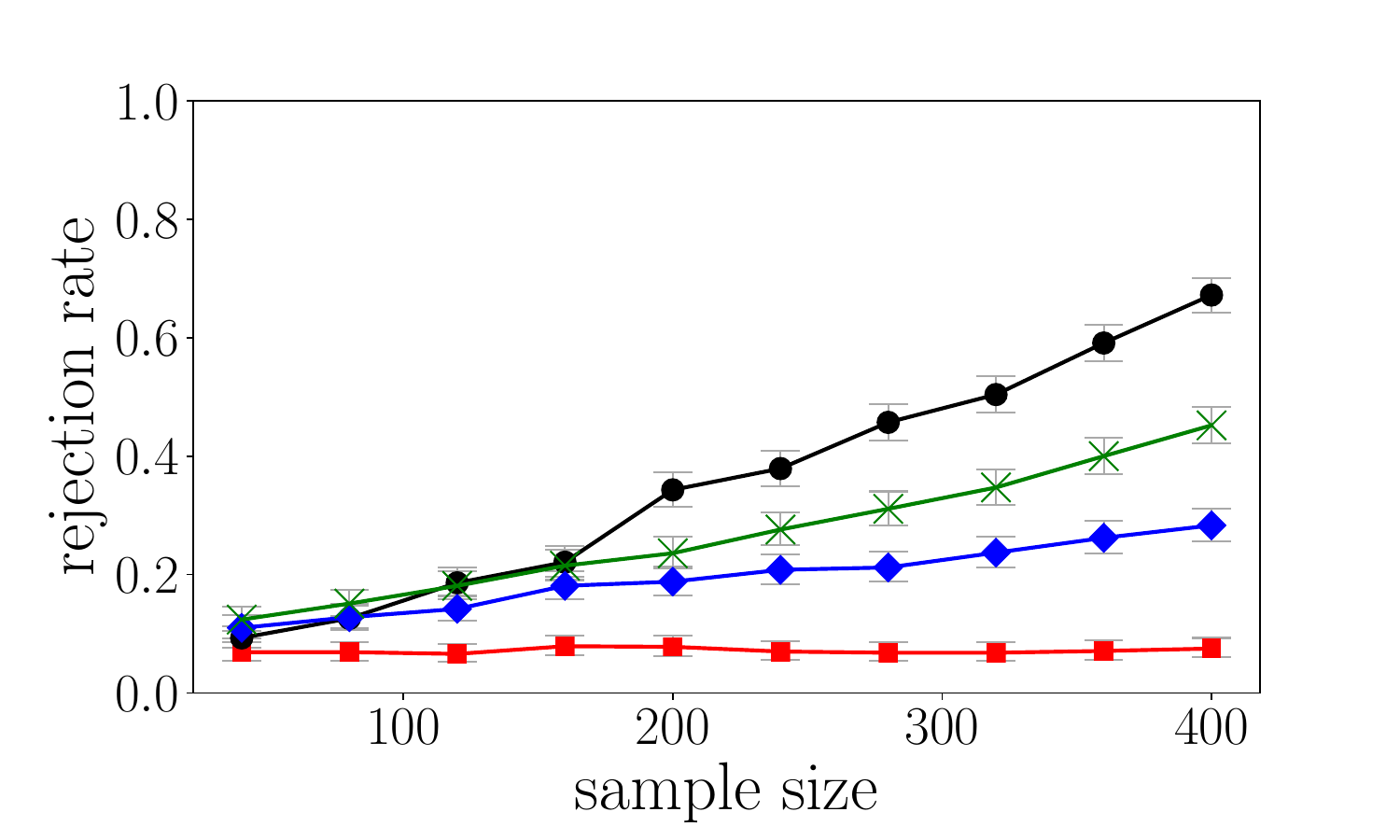}
\includegraphics[width=1.1\linewidth]{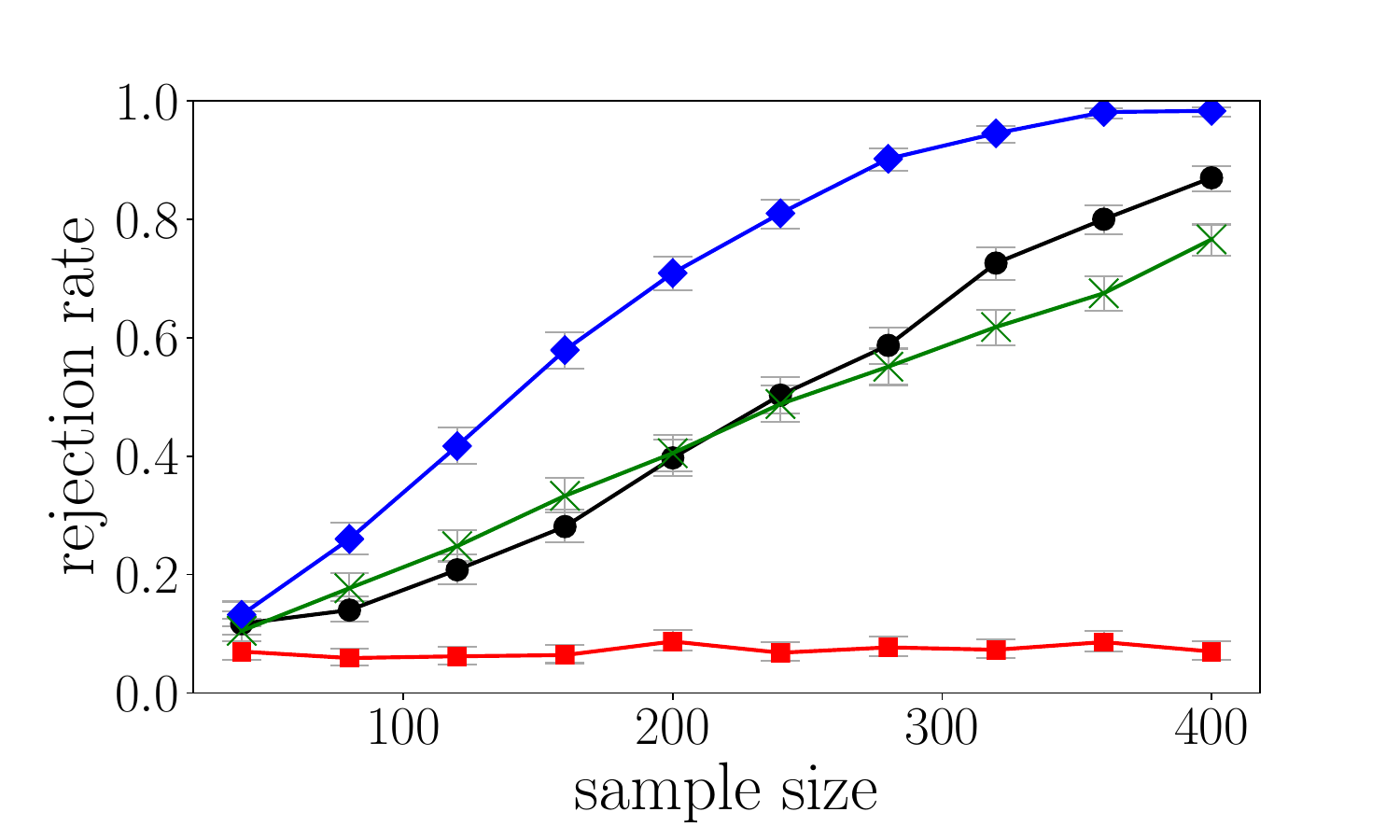}
\includegraphics[width=1.1\linewidth]{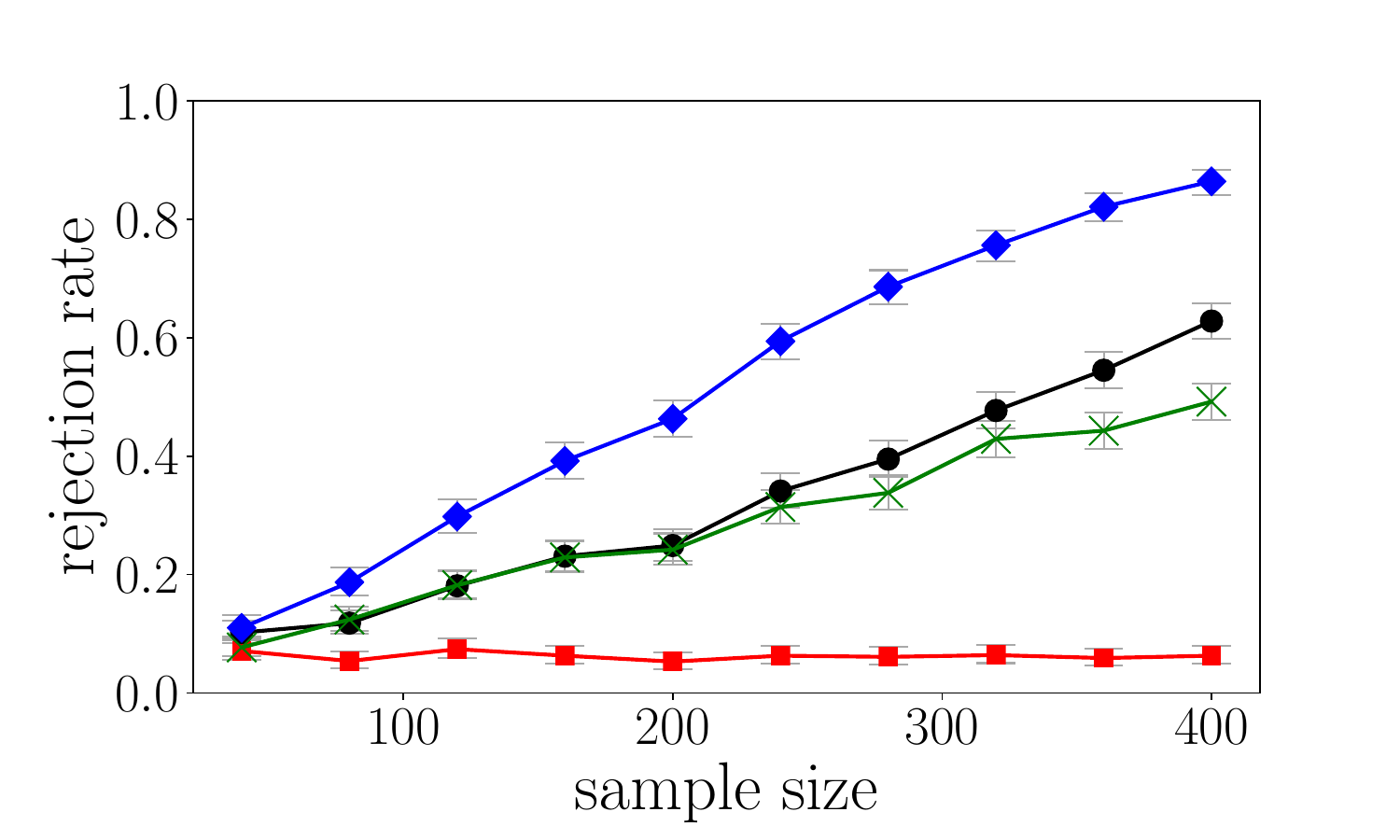}
\includegraphics[width=1\linewidth]{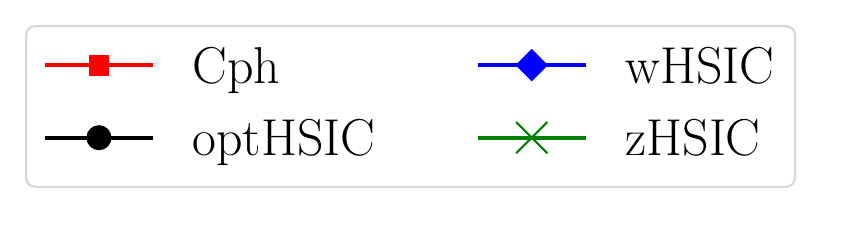}
\end{subfigure}

\caption{\label{fig:power_1dimensional} Scatterplots and rejection rates for D.1 (top row) - D.4 (bottom row) of Table \ref{table:scenarios_power}.}
\end{figure}

\subsubsection{Power for multidimensional covariates}

In D.5-8 of Table \ref{table:scenarios_power} the covariates are multidimensional. Figure \ref{fig:power_higherdimension} shows the rejection rates of the four methods, both as the dimension is fixed at $10$, and the sample size increases (left column) and when the sample size is fixed at 200, and the dimension increases (right column). In D.5 and D.6 the CPH assumption holds. Again we see that the power of optHSIC is relatively close to the power of the CPH likelihood ratio test, which is specifically designed for these assumptions. This holds both when the dependence is on a single sub-dimension of the covariate (D.6) and when the dependence is on all covariates (D.5). In D.7 and D.8 a non-linear term is present in the hazard rate. Together with zHSIC, optHSIC is now the best performing method. Note that as dimension increases, the dependence in D.6-D.8 becomes harder to detect, whereas the dependence in D.5 becomes easier to detect since in the latter $T$ depends on all covariates.

\begin{figure}[H]
\centering
\begin{subfigure}{.5\textwidth}
  \centering
\includegraphics[width=1.1\linewidth]{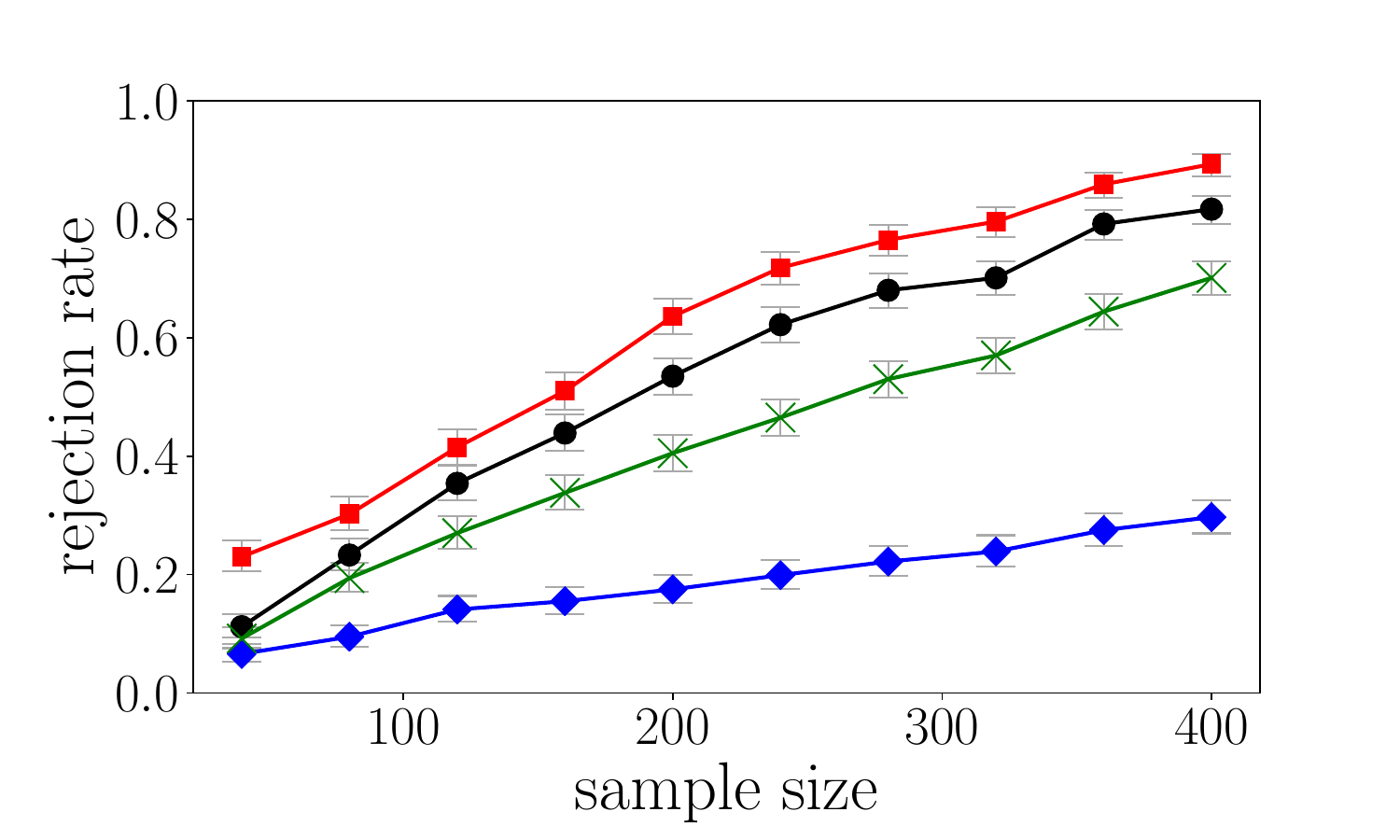}
\includegraphics[width=1.1\linewidth]{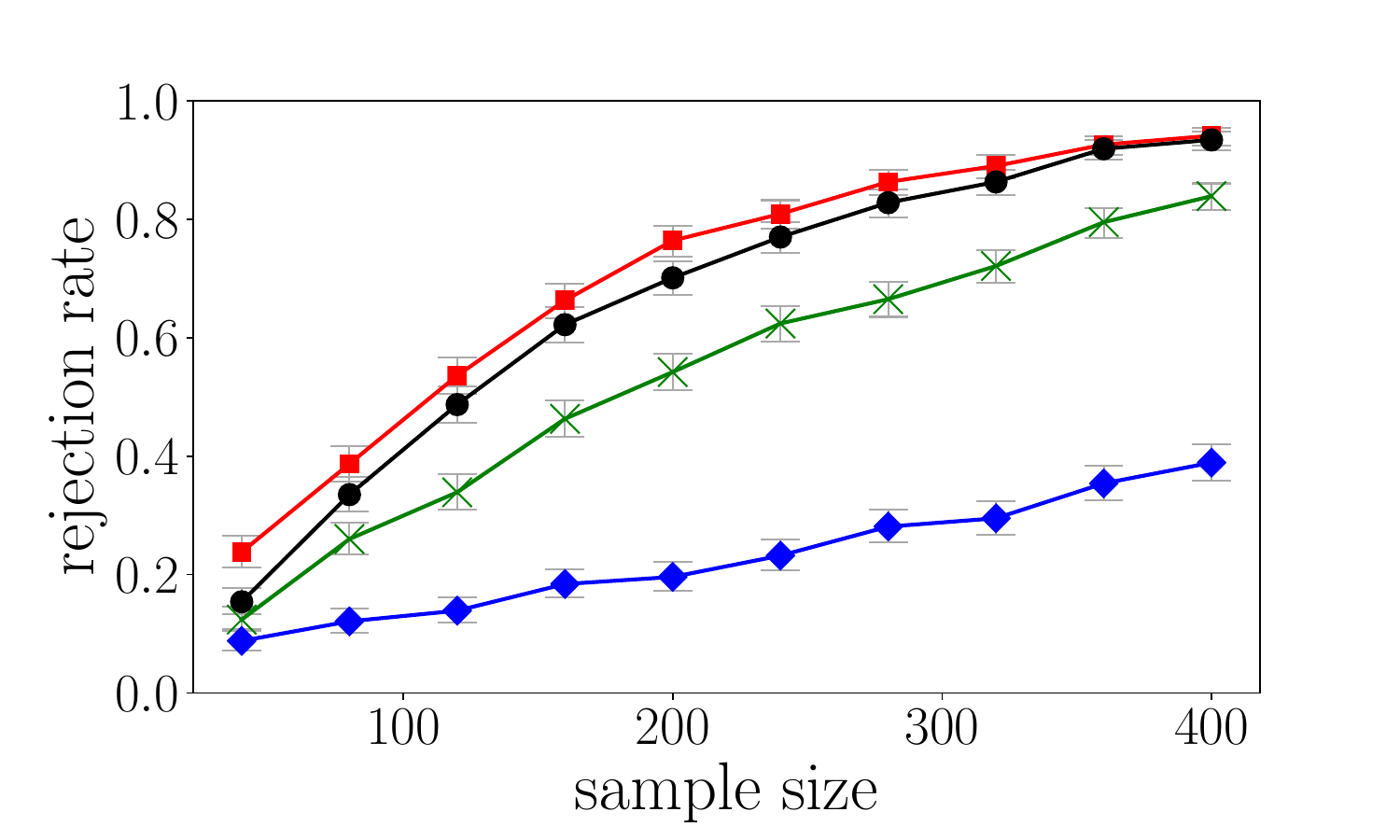}
\includegraphics[width=1.1\linewidth]{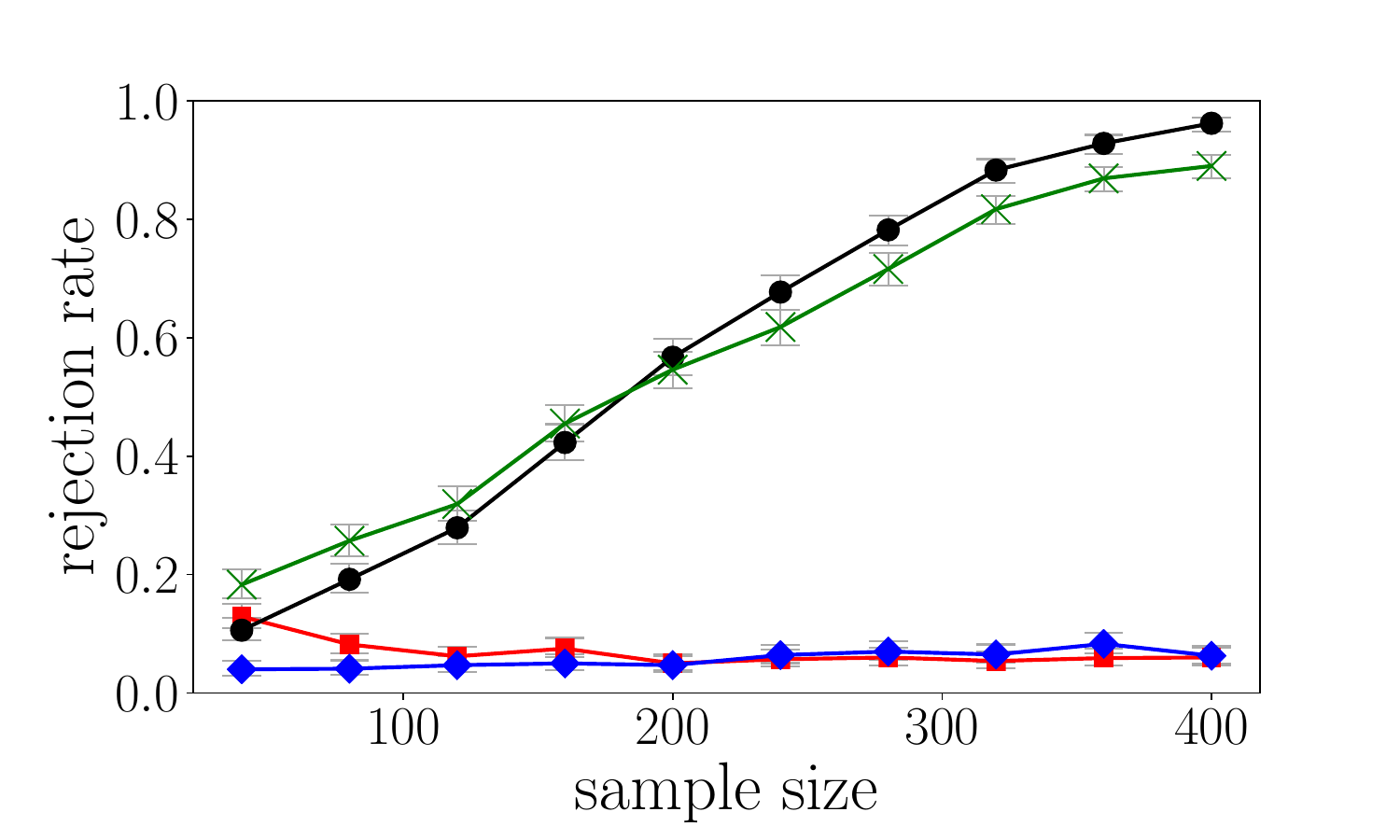}
\includegraphics[width=1.1\linewidth]{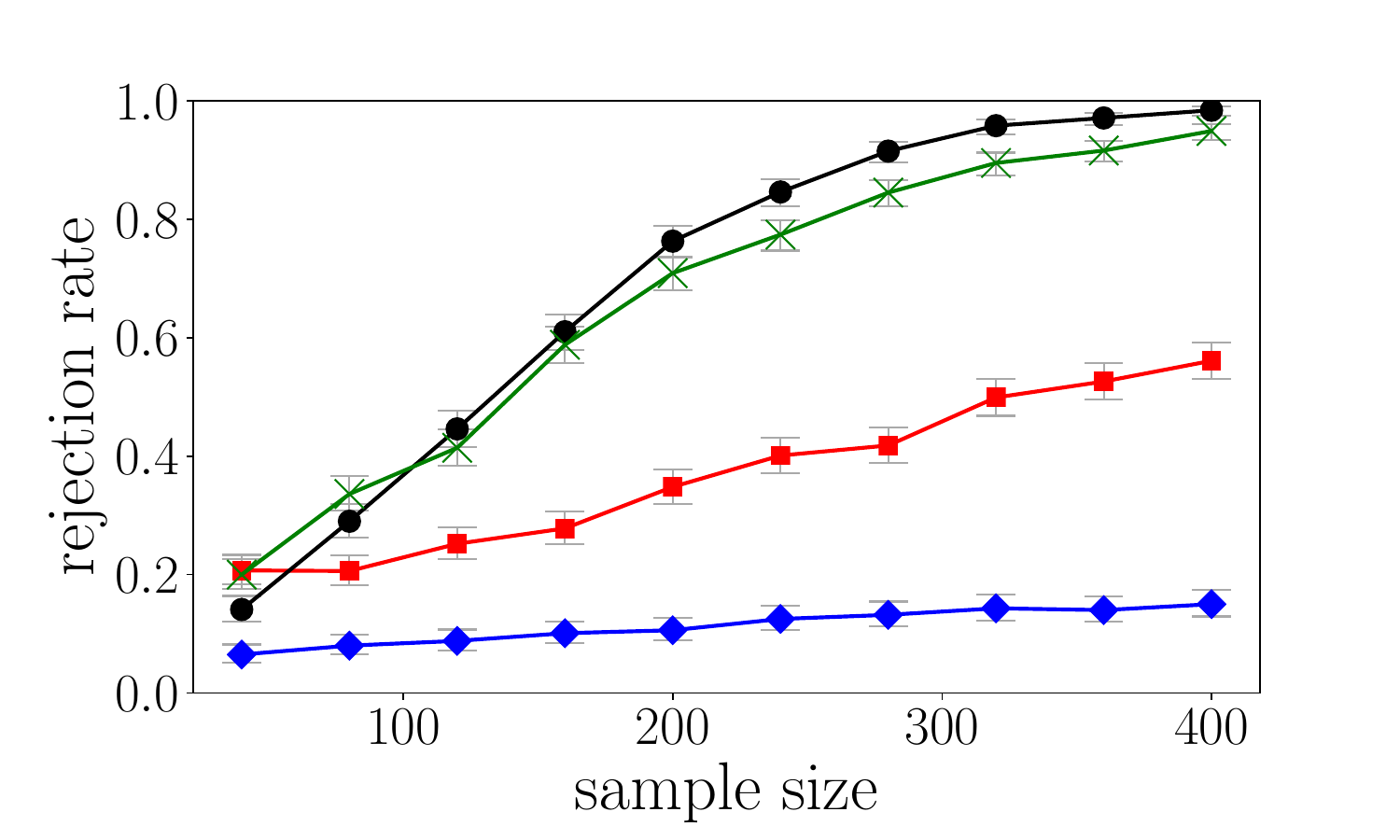}
\end{subfigure}%
\begin{subfigure}{.5\textwidth}
  \centering
\includegraphics[width=1.1\linewidth]{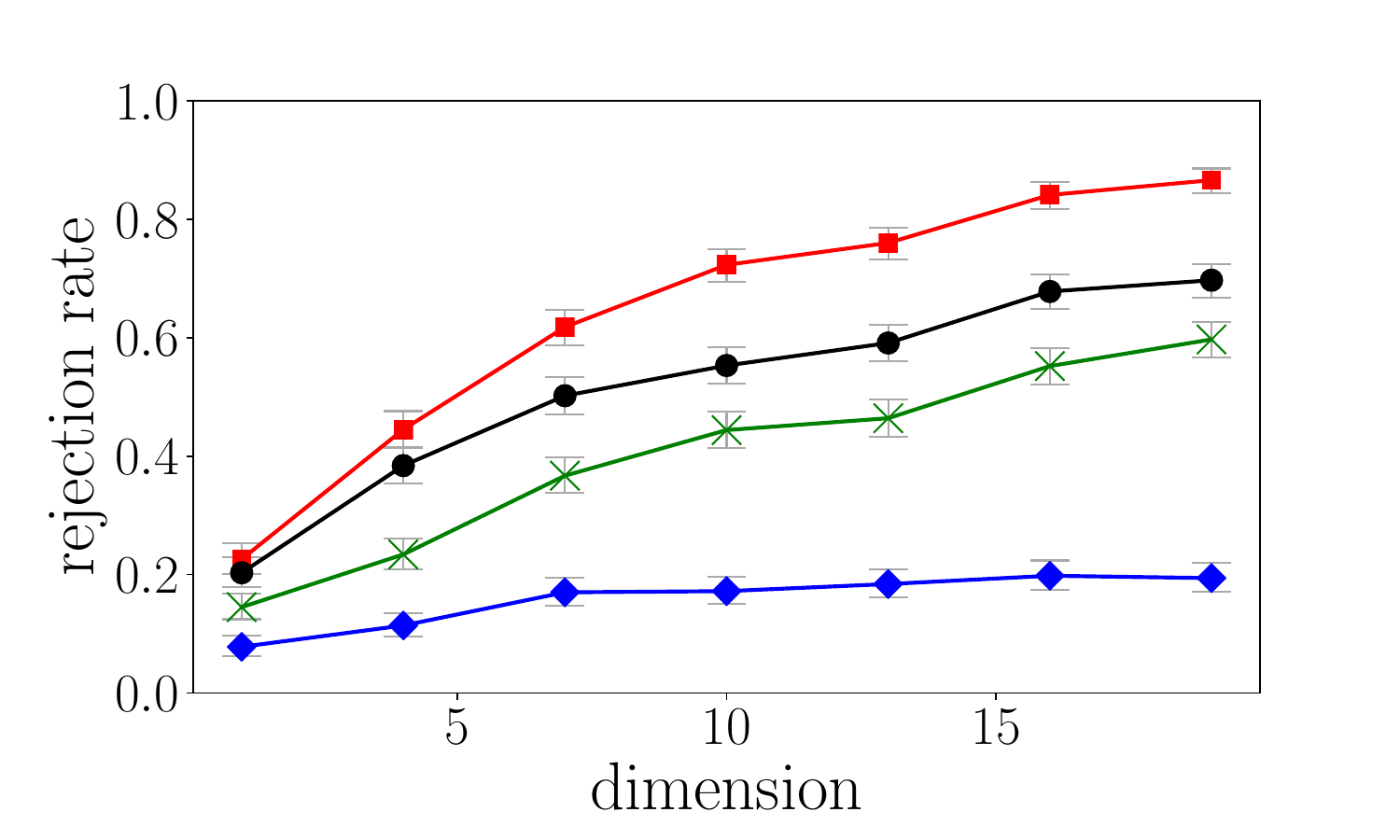}
\includegraphics[width=1.1\linewidth]{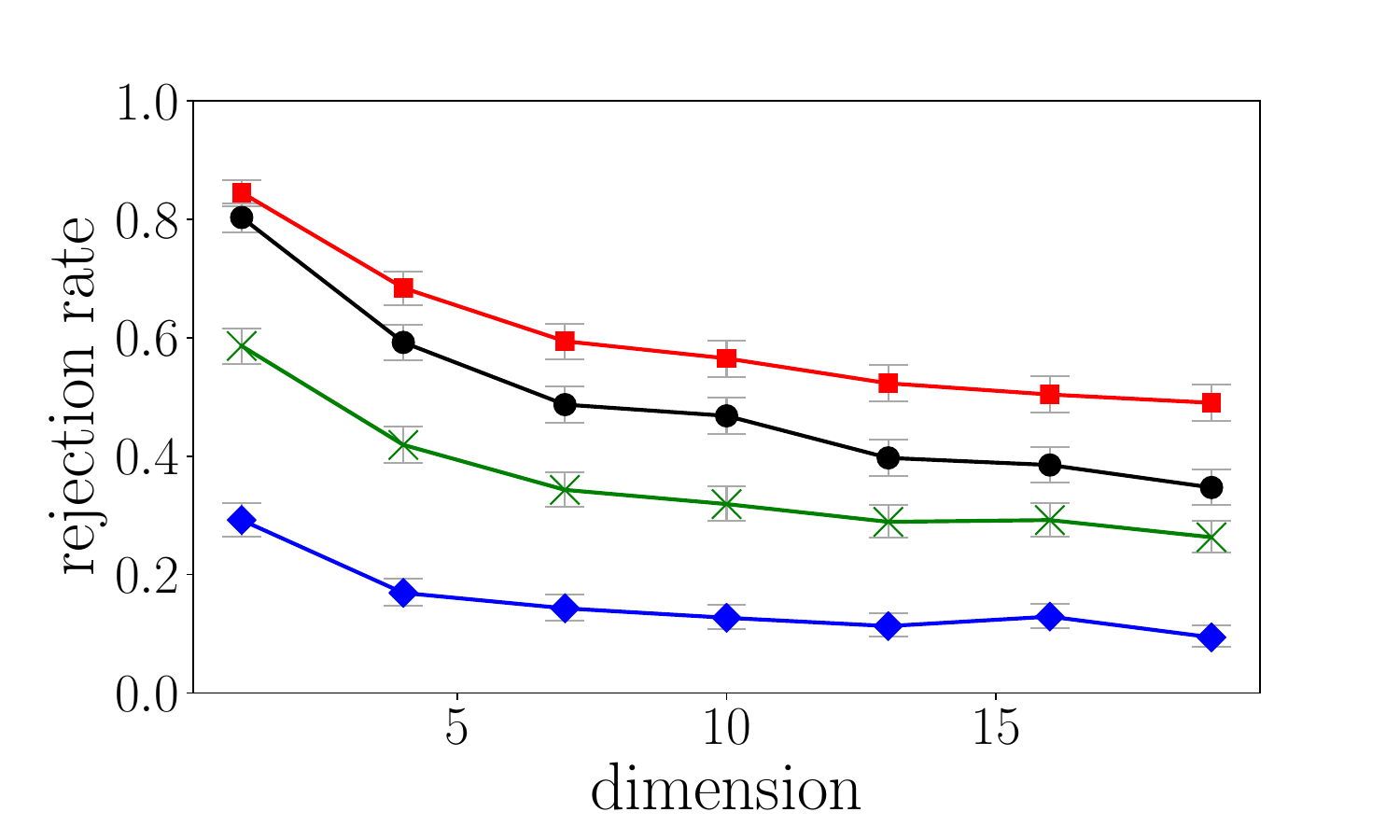}
\includegraphics[width=1.1\linewidth]{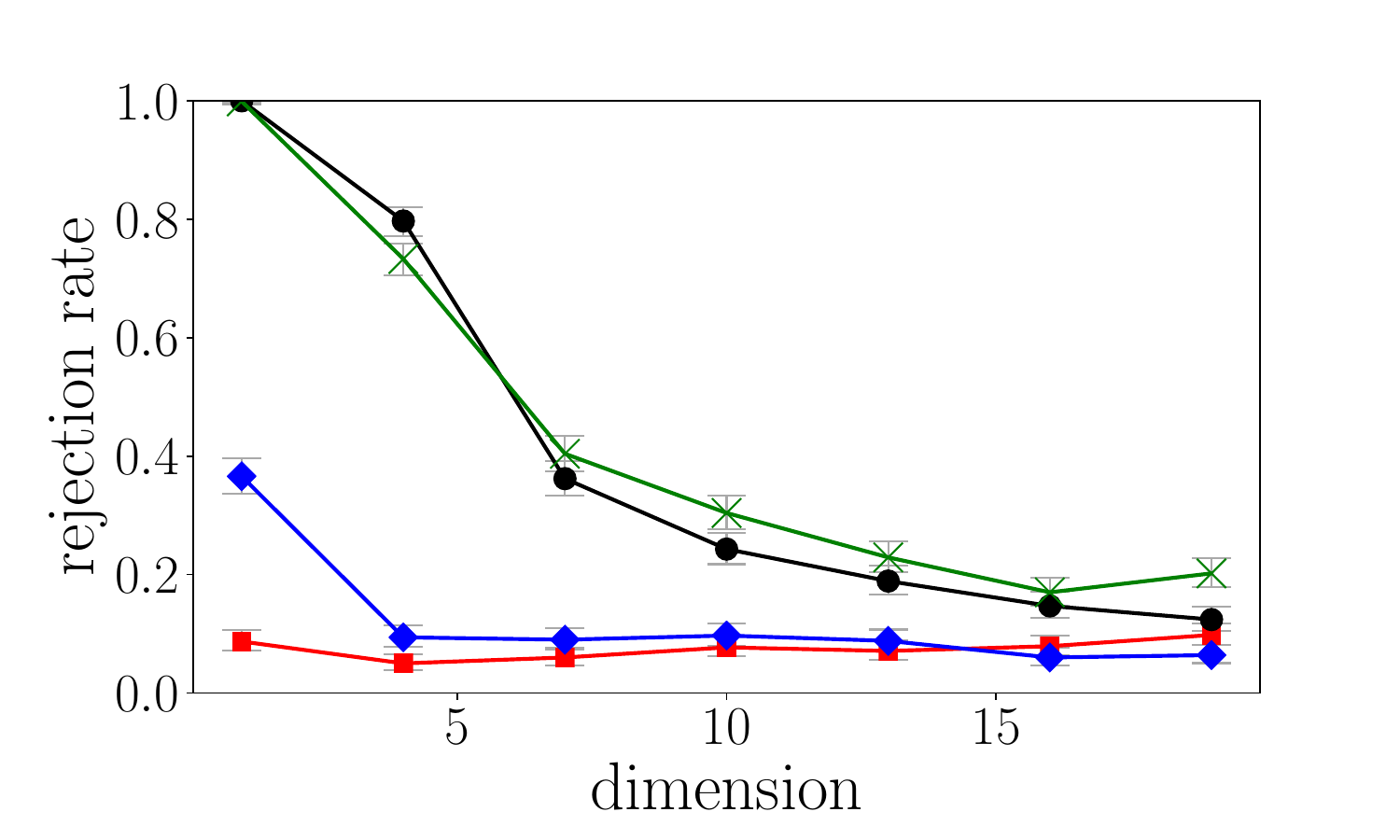}
\includegraphics[width=1.1\linewidth]{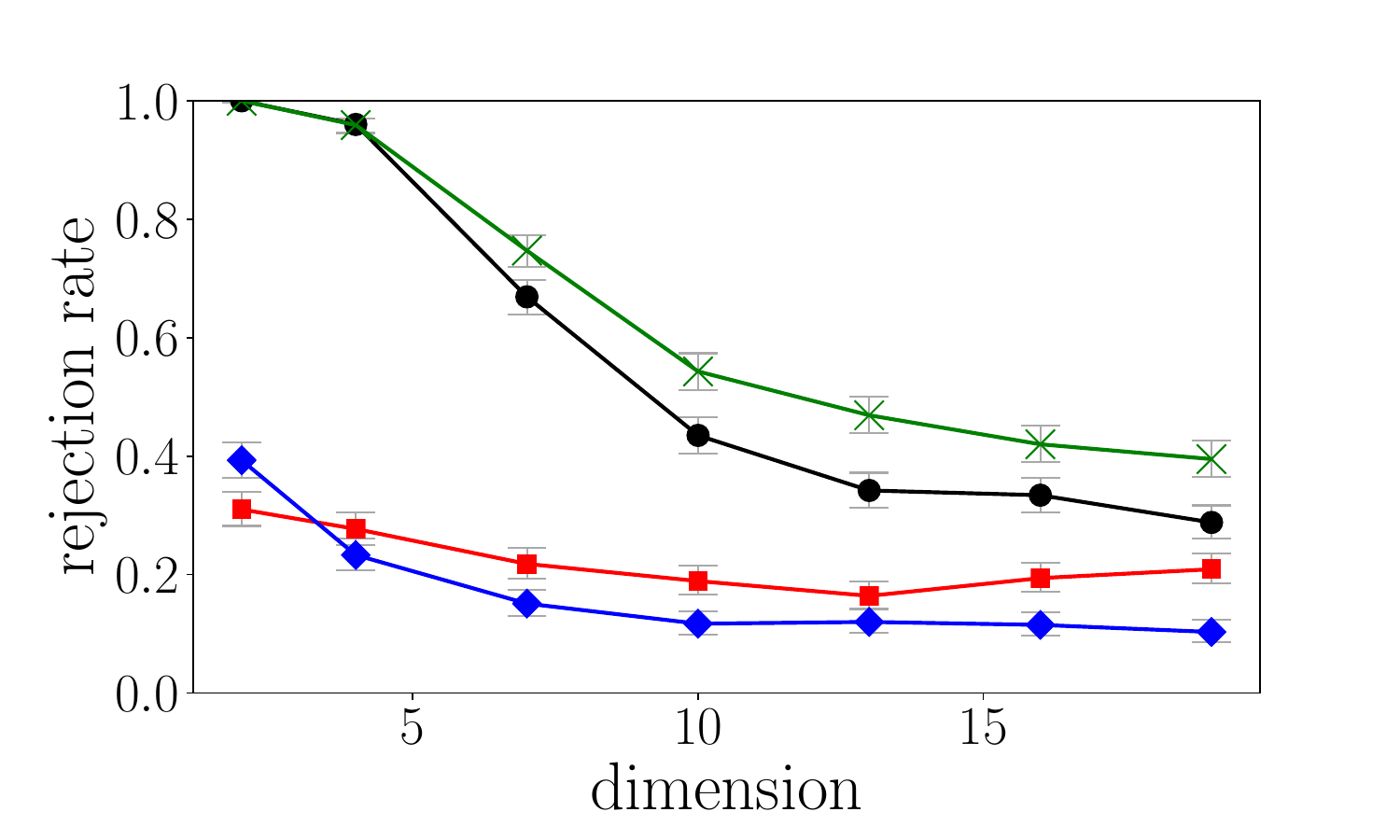}
\end{subfigure}
\begin{subfigure}{.7\textwidth}
  \centering
\includegraphics[width=1.1\linewidth]{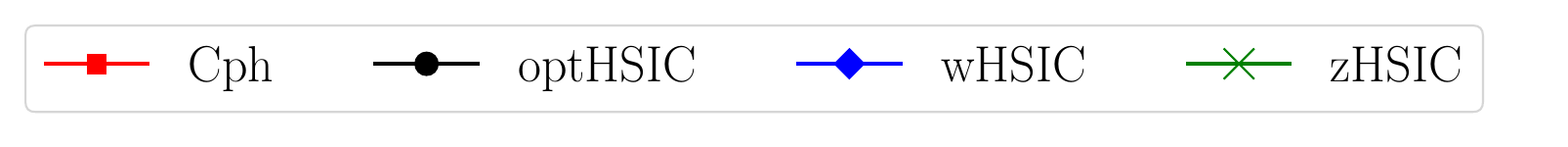}
\end{subfigure}
\caption{\label{fig:power_higherdimension} Rejection rates of the four methods for D.5 (top row) - D.8 (bottom row) of Table \ref{table:scenarios_power}. Left column: dimension fixed at 10, and the sample size increases. Right column: sample size fixed at 200, and the dimension increases. }
\end{figure}

We also investigated the power when $C\not \independent X$ and as the censoring percentage varies from $0\%$ to $80\%$. Distributions and rejection rates are presented in Section \ref{sec:appendix:varying_censoring}. The main observations are that, although all methods lose power as the censoring rate increases, the relative performance of the four methods remains similar to the results presented in the main text.

We thus find that for continuous covariates optHSIC is able to detect a wider range of dependencies than the CPH likelihood ratio test, while not losing much power when the CPH assumptions hold. This is true both when the covariate is 1-dimensional and when the covariate is multidimensional, even when the dependence arises from a lower-dimensional subspace of the covariate. Importantly, unlike the methods wHSIC and zHSIC, the type 1 error rate of optHSIC is found to be of the correct level both when $C\independent X$ and when $C\not \independent X$.

\subsection{Binary covariates \label{sec:additional_experiments}}

Lastly, we did experiments with binary covariates, in which case independence testing is equivalent to two-sample testing. For two-sample testing there are other alternative approaches, including \cite{weightedlogrank} and \cite{matabuenatwosample}. Furthermore, the wHSIC approach can be modified firstly to allow for the computation of weights within groups, and secondly a permutation test can be constructed that is valid also when censoring differs between groups. This permutation test was proposed by \cite{wangpermutation}. This special case of wHSIC coincides with the approach discussed in \cite{matabuenatwosample} for two-sample testing. 

Our main finding is that optHSIC has a decent all-round performance for two-sample testing. However, especially when the distributions meet additional assumptions used in semi-parametric tests, these semi-parametric tests outperform optHSIC. We also note that, as optHSIC relies on choosing `similar covariates' during the transformation of Algorithm \ref{algorithm:transformation}, it is not ideally suited for the two-sample case, as there is a chance of choosing the opposite covariate. 

We also provide a `failure mode' of optHSIC (we thank a reviewer for this example): a scenario in which optHSIC performs much worse than the CPH test. In this example, covariates are binary (i.e. there are two groups), group sizes are very unequal, $90\%$ of observations is censored, censoring appears only in the large group, and survival in the two groups is identical until beyond the censoring has occurred. See \ref{sec:appendix:exampletwosample} for a detailed description. 

 A full presentation of the experiments and methods used for the case of binary covariates is given in Section \ref{sec:appendix:binary_covariates}. We conclude that the main value of optHSIC lies in independence testing with continuous covariates.

\section{Discussion} 
The main contribution of this paper is the proposal of optHSIC, combining a novel way of using optimal transport to transform right-censored datasets, with
nonparametric permutation-based independence testing using HSIC/DCOV. We have shown optHSIC has power against a wider range of alternatives than the commonly used CPH model, while forfeiting little power when the CPH assumptions are
satisfied exactly. Extensive numerical simulations suggest that the approach is well calibrated, yielding reliable $p$-values even when censoring strongly depends on the covariate. Under the assumption that censoring does not depend on the covariate, we have proven correct type 1 error rate of optHSIC. Theoretical guarantees for the type 1 error under dependent censoring are a topic of future work. We furthermore proposed reweighting the original dataset, and measuring the distance between the resulting weighted mean embeddings. While these methods showed some promise, they do rely on the very strong assumption that $C$ and $X$ are independent.  An interesting future challenge is to develop nonparametric tests for covariates and right-censored times for mutual independence testing and conditional independence testing, two methods used in causal inference, and which are relevant to many problems in biostatistics.

\section{Supplementary materials and code}

Supplementary materials contain a worked out example of the transformation proposed in Algorithm \ref{algorithm:transformation}, proofs of all results and details on additional experiments. They also contain methods of combining $p$-values from multiple transformation of the same dataset. Code to replicate the experiments and of the tests is available on \url{www.github.com/davidrindt/opthsic}.

\section{Acknowledgment}
We thank the reviewers and editor for their helpful comments.

\newpage
\appendix
\section{Appendix}
\subsection{Example of transformation \label{sec:appendix:example_transformation}}
Let $D$ be the following dataset consisting of $n=5$ individuals:
\begin{table}[H]
\centering
\begin{tabular}{llll}
\hline
$i$ & $x_i$  & $z_i$  & $\delta_i$ \\ \hline
1 & $1.3$ & 13 & 1 \\ 
2 & $0.5$ & 22 & 0 \\
3 & $0.3$ & 24 & 1 \\
4 & $-1.1$ & 45  & 1 \\
5 & $-0.9$ & 81 & 0 \\
\hline
\end{tabular}
\caption{An example dataset $D$ for which we will demonstrate the transformation.}
\end{table}
We initialize $\tilde D \leftarrow [\ ]$ to be an empty multiset and set $\LL \leftarrow [1.3,0.5,0.3,-1.1,-0.9]$ and  $\AR \leftarrow [1.3,0.5,0.3,-1.1,-0.9]$. We loop over the events $i=1,\dots,4$.

At the first time $z_1=13$ it holds that $\delta_1=1$. We compute the joint distribution $P_{UU'}$ that solves the optimal transport problem between $U\sim \text{Uniform}(\AR)$ and $U'=\text{Uniform}(\LL)$. Since it holds that $\AR=\LL$, it follows that $P_{UU'}$ is the matrix:

\[  P \leftarrow 
\begin{blockarray}{cccccc}
x_1 & x_2 & x_3 & x_4 & x_5 \\
\begin{block}{(ccccc)c}
  0.2 & 0. & 0. & 0. & 0. & x_1\\
  0. & 0.2 & 0. & 0. & 0. & x_2\\
  0. & 0. & 0.2 & 0. & 0. & x_3\\
  0. & 0. & 0. & 0.2 & 0. & x_4\\
  0. & 0. & 0. & 0. & 0.2 & x_5\\
\end{block}
\end{blockarray}
 \]

Conditioning $P$ on $U=x_1$ yields $v\leftarrow [1,0,0,0,0]$, corresponding to the first row of $P$. Sampling from $\LL$ with distribution $v$ yields $\tilde x\leftarrow 1.3=x_1$ with probability 1. We update $\tilde D\leftarrow [(1.3,13)]$. We also replace $\LL \leftarrow [0.5,0.3,-1.1,-0.9]$ and  $\AR \leftarrow [0.5,0.3,-1.1,-0.9]$. We move to the next event time. 

At $z_2=22$ we note that $\delta_2=0$, so we only remove $x_2=0.5$ from $\AR$ and update $\AR\leftarrow [0.3,-1.1,-0.9]$, while leaving $\LL$ and $\tilde D$ unchanged. 

At the third event $z_3=24$ it holds that $\delta_3=1$ and $\AR= [0.3,-1.1,-0.9]$ and $\LL = [0.5,0.3,-1.1,-0.9]$. We couple a random variable $U\sim \text{Uniform}(\AR)$ and $U'=\text{Uniform}(\LL)$ using optimal transport. The resulting distribution equals: 

\[  P \leftarrow 
\begin{blockarray}{ccccc}
x_2& x_3 & x_4 & x_5 \\
\begin{block}{(cccc)c}
  0.25 & 0.083 & 0. & 0. & x_3\\
  0. & 0. & 0.25 & 0.083 & x_4\\
  0. & 0.167 & 0. & 0.167 & x_5\\
\end{block}
\end{blockarray}
 \]

We condition this distribution on $U=x_3=0.3$. This corresponds to the first row of $P$, and the resulting distribution over $\LL$ equals: $v\leftarrow [0.75,0.25,0,0]$. We now sample a point from this distribution and, suppose, it turns out to be $\tilde x\leftarrow 0.5=x_2$, which has $75 \%$ chance. We update $\tilde D\leftarrow [(1.3,13),(0.5,24)]$. We also replace $\LL \leftarrow [0.3,-1.1,-0.9]$ and  $\AR \leftarrow [-1.1,-0.9]$ before moving to the next event.

At $i=4$ it holds that $z_4=45$ and $\delta_4=1$. We note $\AR= [-1.1,-0.9]$ and $\LL = [0.3,-1.1,-0.9]$. We couple a random variable $U\sim \text{Uniform}(\AR)$ and $U'=\text{Uniform}(\LL)$ using optimal transport. The resulting distribution equals: 

\[  P \leftarrow 
\begin{blockarray}{cccc}
 x_3 & x_4 & x_5 \\
\begin{block}{(ccc)c}
  0. & 0.333 & 0.167 & x_4\\
  0.333 & 0. & 0.167 & x_5\\
\end{block}
\end{blockarray}
 \]

We condition $P$ on $U=x_4=-1.1$, resulting in $v\leftarrow [0,0.67,0.33]$. In this case our sample turns out to be $\tilde x\leftarrow -1.1=x_4$, which happens with probability $0.67$. Hence $\tilde D\leftarrow [(1.3,13),(0.5,24),(-1.1,45)]$. We also replace $\LL \leftarrow [0.3,-0.9]$ and  $\AR \leftarrow [-0.9]$. 

We have now finished the loop $i=1,\dots,4$. Since $z_5=81$ and $\LL\leftarrow [0.3,-0.9]$ we add the two datapoints $(0.3,81)$ and $(-0.9,81)$ to $\tilde D$. The finalized transformed dataset equals
$$ \tilde D\leftarrow [(1.3,13),(0.5,24),(1.1,45),(0.3,81),(-0.9,81)].$$

\subsection{Proof of Lemma \ref{lemma:permutedtransformedvstransformpermuted}}

 Let $D=((x_i,z_i,\delta_i))_{i=1}^n$ where $z_i$ is increasing, and assume for convenience there are no ties in $z$. Denote by $k\coloneqq|\{ i:\delta_i=1\}|$ the number of observed events. We do not view $D$ as random in this section. Applying the optimal transport algorithm results in a random, transformed dataset, which we denote by $T(D)$. Note that the times and covariates in $T(D)$ are not random, since they are determined by $D$, but the way in which they are paired up in the transformation $T$ may be random. The same set of times and covariates is obtained in $\pi(T(D))$ and $T(\pi(D))$ for any $\pi \in S_n$. 
 Denote the times in $T(D)$ by $t_1\leq \dots \leq  t_n$ and define a standard pairing $\tilde D=\left( (x_i,t_i) \right)_{i=1}^n$,. We will often use that $T(D),\pi(T(D)),T(\pi(D))$ are all permutations (possibly random) of $\tilde D$. 
 Finally, define $h:\{1,\dots,k \} \to \{1,\dots,n \} $, so that $t_i=z_{h(i)}$, which says that the $i$-th observed event is the $h(i)$-th overall event. As a last piece of notation, we will use $\Pi$ to denote a uniform random permutation, and $\pi$ to be a specific instance of a permutation. In particular we denote $\Pi_i=\Pi(i)$ and $\Pi_{1:h(i)-1}=[\Pi_1,\dots,\Pi_{h(i)-1}]$. This corresponds to the covariates in the permuted dataset $\Pi(D)$ until just before the time of the $i$-th observed event.

We prove the theorem by showing that the left- and right-hand sides of Lemma \ref{lemma:permutedtransformedvstransformpermuted} are both equal in distribution to
$$
\left[T(D),\Pi^1(\tilde D),\dots,\Pi^B(\tilde D) \right]. 
$$
This is done in separate lemmas.
\begin{lemma} \label{lemmapermutetransformed}
$$
	\left[ T(D),\Pi^1(T(D)),\dots,\Pi^B(T(D)) \right] 	\stackrel{d}{=}\left[T(D),\Pi^1(\tilde D),\dots,\Pi^B(\tilde D) \right] 
$$
\end{lemma}
\begin{proof}

By the above remarks we see that $T(D)=\Pi^D(\tilde D)$
for some random permutation $\Pi^D$. (Note: The randomness in $\Pi^D$ comes from the transformation 
$T$, not from the dataset $D$, which is fixed.)
It suffices to show that
$$
\left[\Pi^D,\Pi^1\circ \Pi^D ,\dots,\Pi^B \circ \Pi^D\right] 	
	\stackrel{d}{=}\left[\Pi^D,\Pi^1,\dots,\Pi^B\right] .
$$
This is easy to see by conditioning on $\Pi^D$. Let $\pi^0,\dots\pi^B$ be arbitrary permutations. Then
\begin{align*}
& P\bigl(\Pi^D=\pi^0,\Pi^1\circ \Pi^D=\pi^1,\dots,\Pi^B \circ\Pi^D=\pi^B \bigr) \\ = & P\bigl(\Pi^1 \circ \pi^0 = \pi^1,\dots,\Pi^B\circ \pi^0=\pi^B \, \, \bigm|\, \, \Pi^D=\pi^0\bigr)P(\Pi^D=\pi^0) \\ = & P\bigl(\Pi^1=\pi^1 \circ (\pi^0)^{-1},\dots,\Pi^B=\pi^B \circ (\pi^0)^{-1}\bigr)P(\Pi_D=\pi^0).
\end{align*}
Since $(\Pi^1,\dots,\Pi^B)$ are independent uniform permutations, 
this is the same as
$$
	P\bigl(\Pi^D=\pi^0,\Pi^1=\pi^1,\dots,\Pi^B =\pi^B \bigr).
$$
\end{proof}

We now consider the effect of first permuting and then transforming the data.

\begin{lemma}\label{lemmatransformpermuted} Let $\Pi$ be a uniformly chosen permutation of $S_n$ and let $T$ be defined through optHSIC. It holds that
$$T(\Pi(D))\stackrel{d}{=}\Pi(\tilde D).$$
\end{lemma}
\begin{proof}
By the comments above, we can define a random permutation $\Sigma$ by $\Sigma(\tilde D)\coloneqq T(\Pi(D))$.  We wish to show that $P(\Sigma=\sigma)=1/n!$ for all $\sigma \in S_n$. To do so, we will condition on events of the form
$$ 
\{\Pi_{1:h(i)-1}=\pi_{1:h(i)-1} \}, 
$$
which determines the covariates in the permuted dataset up to (just before) the time of the $i$-th observed event. We also condition on $\Sigma_{1:i-1}$, fixing the covariates in the transformed dataset, up to the $i$-th observed event.
Note that this conditioning fixes the coupling defined in the optimal transport algorithm. Namely, we let $\tilde Y$ and $\tilde X$ be the coupled random variables resulting from optimal transport between choosing uniformly from the covariates indexed by $[n]\setminus \{ \sigma_{1:i-1} \}$ and  choosing uniformly from the covariates indexed by $[n]\setminus\{\pi_{1:h(i)-1} \}$ respectively. Then, the transformation samples $U'$ conditional on $U=x_{\Pi_{h(i)}}$. Because $\Pi$ is a uniformly chosen permutation, given the events we conditioned on so far, $U$ is uniformly chosen from the covariates indexed by $[n]\setminus\{\pi_{1:h(i)-1} \}$. By the definition of the coupling, $U'$ is thus uniform on the covariates indexed by $[n]\setminus \{\sigma_{1:i-1}\}$. That is, $\Sigma_i$ is chosen uniformly from  $[n]\setminus \{\sigma_{1:i-1}\}$. Mathematically, for any $\sigma, \pi \in S_n$ so that the conditioning event has nonzero probability, it holds that
\begin{align*}
P(\Sigma_i=\sigma_i &\, \bigm|\,  \Pi_{1:h(i)-1}=\pi_{1:h(i)-1},\Sigma_{1:i-1}=\sigma_{1:i-1}) \\
=&\frac{1}{n-i+1}.
\end{align*}
Having shown that, conditioned on what happened in both the permuted dataset, and the synthetic dataset, the new synthetic covariate is chosen uniformly from those not chosen before, we aim to derive a recurrence relation so as to apply this result at each successive time. To this end note that
\begin{align*}
& P(\Sigma_{i:k}=\sigma_{i:k} \bigm|\  \Pi_{1:h(i)-1}=\pi_{1:h(i)-1},\Sigma_{1:i-1}=\sigma_{1:i-1})\\
 =&P(\Sigma_{i+1:k}=\sigma_{i+1:k} \, \bigm|\,  \Pi_{1:h(i)-1}=\pi_{h(i)-1},\Sigma_{1:i}=\sigma_{1:i}) \\ & \times
P(\Sigma_i=\sigma_i \, \bigm|\,  \Pi_{1:h(i)-1}=\pi_{1:h(i)-1},\Sigma_{1:i-1}=\sigma_{1:i-1})\\
 =&\frac{1}{n-i+1}P(\Sigma_{i+1:k}=\sigma_{i+1:k} \, \bigm|\,  \Pi_{1:h(i)-1}=\pi_{h(i)-1},\Sigma_{1:i}=\sigma_{1:i})
 \\ =&\frac{1}{n-i+1}\sum_{\pi_{h(i):h(i+1)-1}}P(\Sigma_{i+1:k}=\sigma_{i+1:k} \, \bigm|\,  \Pi_{1:h(i+1)-1}=\pi_{1:h(i+1)-1},\Sigma_{1:i}=\sigma_{1:i})\\ & 
\times P(\Pi_{h(i):h(i+1)-1}=\pi_{h(i):h(i+1)-1}\, \bigm|\,   \Pi_{1:h(i)-1}=\pi_{1:h(i)-1},\Sigma_{1:i}=\sigma_{1:i})
 \\ 
\end{align*}
where we use the previously established equality in the first equality. This allows us to compute
\begin{align*}
& P(\Sigma_{1:k}=\sigma_{1:k}) \\ = & \sum_{\pi_{1:h(1)-1}}P(\Sigma_{1:k}=\sigma_{1:k}\, \bigm|\,  \Pi_{1:h(1)-1}=\pi_{1:h(1)-1})\\ & \times P(\Pi_{1:h(1)-1}=\pi_{1:h(1)-1}) \\ = & \frac{1}{n}
 \sum_{\pi_{1:h(1)-1}}\sum_{\pi_{h(1):h(2)-1}}P(\Sigma_{2:k}=\sigma_{2:k}\, \bigm|\,  \Pi_{1:h(2)-1}=\pi_{1:h(2)-1},\Sigma_1=\sigma_1)\\ & \times P(\Pi_{h(1):h(2)-1}=\pi_{h(1):h(2)-1}\, \bigm|\,  \Pi_{1:h(1)-1}=\pi_{1:h(1)-1},\Sigma_1=\sigma_1) \\ &\times P(\Pi_{1:h(1)-1}=\pi_{1:h(1)-1}) \\ = & \frac{1}{n(n-1)\cdots(n-k+2)} \\ & \times \sum_{\pi_{1:h(1)-1}}\cdots\sum_{\pi_{h(k-1):h(k)-1}}P(\Sigma_{k}=\sigma_k\, \bigm|\,  \Pi_{1:h(k)-1}=\pi_{1:h(k)-1},\Sigma_{1:k-1}=\sigma_{1:k-1}) \\ & \times P(\Pi_{h(k-1):h(k)-1}=\pi_{h(k-1):h(k)-1} \, \bigm|\,  \Pi_{1:h(k-1)-1}=\pi_{1:h(k-1)-1},\Sigma_{1:k-1}=\sigma_{1:k-1})\\ & \times \cdots\\ & \times P(\Pi_{1:h(1)-1}=\pi_{1:h(1)-1})\\ =& \frac{1}{n(n-1)\cdots(n-k+1)}
\end{align*}
Since the indices $\Sigma_{(k+1):n}$ are added in uniform random order by definition of the transformation algorithm, this concludes the lemma. \end{proof}
\begin{lemma} \label{jointtransformpermuted}
$$
\left[ T(D),T_1(\Pi^1(D)),\dots,T_B(\Pi^B(D)) \right] \stackrel{d}{=} \left[T(D),\Pi^1(\tilde D),\dots,\Pi^B(\tilde D) \right] 
$$ 
\end{lemma}
\begin{proof} The left hand side can be written as $[\Pi^D(\tilde D),\Sigma^1(\tilde D),\dots,\Sigma^B(\tilde D)]$. The lemma above shows that the $\Sigma^i$ for $i\geq 1$ have the correct distributions. We only need to show they and $\Pi^D$ are a sequence of mutually independent permutations. But $\Sigma^i$ is determined completely by $\Pi^i$ and $T_i$, and $\Pi^D$ is determined by $T$. The proof follows since all these variables are mutually independent.\end{proof}

Lemma A.1 and A.3 together prove the theorem. \\

\subsection{Proof of Theorem \ref{thm:correct_type1_opt_independent_censoring}}

 The proof of Theorem \ref{thm:permutation_test_if_independent_censoring} shows that, if $C \independent X$, then 
$$(D,\Pi_1(D),\dots,\Pi_B(D)) $$
is an exchangeable vector. In particular, if $T,T_1\dots,T_B$ are independent identically distributed transformations of the data, then also 
$$ (T(D),T_1(\Pi_1(D)),\dots,T_B(\Pi_B(D)))$$
is exchangeable. We let $T$ be the transformation of the data using the transformation of the data. By Lemma \ref{lemma:permutedtransformedvstransformpermuted} the above vector is equal in distribution to
$$ (T(D),\Pi_1(T(D)),\dots,\Pi_B(T(D))),$$
implying that the latter is also exchangeable. For an arbitrary statistic $H$,
$$[H(T(D)),H(\Pi_1(T(D))),\dots,H(\Pi_B(T(D)))]$$
is thus exchangeable too. In particular, the rank of the first entry is uniformly distributed on $1,\dots,B+1$, which proves the theorem.
\\

\subsection{Proof of Theorem \ref{thm:standard_permutation_test}}

\begin{proof}
This  proof is based on the proof of Lemma 3 of \cite{berrettinformation}. Since $H_0:X \independent Y$ implies that $(X_i,Y_j)\stackrel{d}{=} (X_i,Y_i)$ it is easy to see that $\pi D\stackrel{d}{=}\text{D}$, for any permutation $\pi$. Writing $\Pi^0=id$, and $\Pi^1,\dots,\Pi^B$ for i.i.d. uniform permutations, we aim to show that, for any permutation $\sigma$ of $\{0,1,\dots,B\}$
$$
\left(\Pi^0(D),\Pi^1(D),\dots,\Pi^B(D)\right) \stackrel{d}{=} \left(\Pi^{\sigma_0}(D), \Pi^{\sigma_1}(D),\dots,\Pi^{\sigma_B}(D)         \right);
$$
that is, that the random vector on the left is exchangeable. We observed above that the first entries are equal in distribution. It remains to show that the other entries of the right-hand side are uniform and independently chosen permutations of the first entry. Indeed, writing $\Pi^{\sigma_0}(D)=\tilde D$, we can rewrite the right-hand side as:
$$ 
\left(\tilde D,\Pi^{\sigma_1}(\Pi^{\sigma_0})^{-1}(\tilde D),\dots,\Pi^{\sigma_B}(\Pi^{\sigma_0})^{-1}(\tilde D)\right).
$$
So it remains to show that $\bigl(\Pi^{\sigma_j}(\Pi^{\sigma_0})^{-1}, 1 \leq j \leq B\bigr)$ are independent uniformly chosen permutations of $S_n$. If $\sigma_0=0$, then $\Pi^{\sigma_0}=id$ and $\tilde D=D$ and the result is obvious. Now assume that $\sigma_i=0$ for $i\geq1$.

\begin{align*} 
P\bigl (\Pi^{\sigma_1}(\Pi^{\sigma_0})&^{-1}=\pi^1,\dots,(\Pi^{\sigma_0})^{-1}=\pi^i,\dots,\Pi^{\sigma_B}(\Pi^{\sigma_0})^{-1}=\pi^B \bigr) \\
= & P\left(\Pi^{\sigma_1}\pi^i=\pi^1,\dots,\Pi^{\sigma_B}\pi^i=\pi^B \, \bigm|\, (\Pi^{\sigma_0})^{-1}=\pi^i \right)P \bigl((\Pi^{\sigma_0})^{-1}=\pi^i \bigr) \\ 
= & P\left(\Pi^{\sigma(1)}=\pi^1 (\pi^i)^{-1} \right) \dots P\left(\Pi^{\sigma(B)}=\pi^B (\pi^i)^{-1}\right) P\left (\Pi^{\sigma_0})^{-1}=\pi^i\right) \\ =& (n!)^{-B} \end{align*}
It follows that the vector $$\left(D,\Pi^1(D),\dots,\Pi^B(D)\right) $$ is indeed exchangeable. Letting $H$ denote any arbitrary function on data, it follows that:
$$
\bigl(H(D),H(\Pi^1(D)),\dots,H(\Pi^B(D))\bigr) 
$$ 
is also exchangeable.  If we break ties at random, this implies that every ordering of the $B+1$ elements is equally likely. In particular, the rank of an individual element is uniformly distributed on $\{1,\dots,B+1\}$,
and the result follows.

\end{proof}

\subsection{Proof of Theorem \ref{thm:permutation_test_if_independent_censoring}}

\begin{proof}

When we assume that $C \independent X$ then, under the null hypothesis $H_0:T \independent X$, it follows that the pair $(T,C) \independent X$.
As $(Z,D)$ is $(T,C)$--measurable, also $(Z,D) \independent X$. If we write
$Y=(Z,D)$, then Theorem \ref{thm:standard_permutation_test} applies. 

\end{proof}
\subsection{Proof of Lemma \ref{lemmaexpectationweight}}

The following computation shows that $E Wf(X,Z)=Ef(X,T)$ for all functions $f$. We denote the distribution of $(X,T,C)$ on $\mathcal{X} \times \mathbb R_{\geq 0} \times \mathbb R_{\geq 0}$ by   $\mu_{XTC}$. As we are assuming independence of $T$ and $C$ given $X$ we can decompose $\mu_{XTC}=\mu_{XT}\times \mu_{C|X}$.\\
\begin{align*} E Wf(X,Z) &= E 1\{W\neq 0 \}f(X,Z) \\ &= \int_{\mathcal{X} \times \mathbb R_{\geq 0} \times \mathbb R_{\geq 0}}1\{c \geq t\}\frac{1}{g(t,x)} f(x,t)\mu_{XTC}(dx,dt,dc)  \\ &= \int_{\mathcal{X} \times \mathbb R_{\geq 0}} \int_{t}^{\infty} \frac{1}{g(t,x)} f(x,t)\mu_{C|x}(dc)\mu_{XT}(dx,dt) \\ &= \int_{\mathcal{X} \times \mathbb R_{\geq 0}} \frac{1}{g(t,x)} f(x,t) \int_{t}^{\infty} \mu_{C|x}(dc)\mu_{XT}(dx,dt) \\ &=  \int_{\mathcal{X} \times \mathbb R_{\geq 0}}  f(x,t) \mu_{XT}(dx,dt)  \\ &=Ef(X,T).
\end{align*}
where the penultimate equality follows because $\int_{t}^{\infty} \mu_{C|x}(dc)= \mathbb P(C>t|X=x)=g(t,x)$. \\

\subsection{Proof of Lemma \ref{lemmakaplanmeier}}

Estimating the survival of the censoring distribution amounts to replacing $\delta$ by $1-\delta$ in the Kaplan Meier Survival curve. This yields:
\begin{align*}
\hat P(C > z_k) &= \prod_{i=1}^k \left( \frac{n-i}{n-i+1}  \right)^{1-\delta_i} \end{align*}
Thus the probability of being uncensored by time $z_k$ equals:
\begin{align*}
\hat P(C\geq z_k ) &= \prod_{i=1}^{k-1} \left( \frac{n-i}{n-i+1}  \right) \prod_{i=1}^{k-1} \left( \frac{n-i}{n-i+1}  \right)^{-\delta_i} \\ &=  \frac{n-k+1}{n}  \prod_{i=1}^{k-1} \left( \frac{n-i}{n-i+1}  \right)^{-\delta_i}
\end{align*}
Note now that 
\begin{align*}
\frac{ 1}{ \hat P(C\geq z_k ) } &= n \times \prod_{i=1}^{k-1} \left( \frac{n-i}{n-i+1}  \right)^{\delta_i} \left( \frac{1}{n-k+1} \right) \\ &= n \times w_k
\end{align*}
for points that are uncensored. That is, Kaplan--Meier weights equal a re-scaled inverse of the probability of being uncensored by that time.

\subsection{Proof of Theorem \ref{fastcomputationhsic}}
\begin{proof}The squared norm, written as the inner product with itself, can be expanded into three terms $a_{1}+a_{2}-2a_{3}$ that we compute in turn. We denote by
$A \circ B$ the entrywise product of the matrices $A$ and $B$. Using the Hadamard product
property $\alpha^{\top}\left(A\circ B\right)\beta=\tr\left(D_{\alpha}AD_{\beta}B^{\top}\right)$
where $D_{\alpha}=\text{diag}(\alpha)$, $D_{\beta}=\text{diag}(\beta)$,
we have the following identities:
\begin{align*}
a_{1}&= \sum_{i=1}^n\sum_{j=1}^n w_iw_j k(x_i,x_j)l(z_i,z_j)\\&=w^{\top}\left(K\circ L\right)w\\ & =\tr\left(D_wKD_wL\right);
\end{align*} 
\begin{align*}
a_{2}&=\sum_{i=1}^n\sum_{j=1}^n\sum_{r=1}^n\sum_{s=1}^n w_iw_jw_rw_s k(x_i,x_j)l(z_r,z_s)\\ &=w^{\top}Kww^{\top}Lw\\
&=\tr\left(ww^{\top}Kww^{\top}L\right);
\end{align*}
\begin{align*}
a_{3} & =  \left\langle \sum_{i=1}^{n}w_{i}K((x_{i},z_{i}),\cdot) , \sum_{r=1}^n\sum_{s=1}^n w_r w_s K((z_r,z_s),\cdot) \right\rangle \\
 &=\sum_{i=1}^{n}w_{i}\left(\sum_{j=1}^{n}w_jk\left(x_{i},x_{j}\right)\right)\left(\sum_{r=1}^{n}w_{r}l\left(z_{i},z_{r}\right)\right)\\&=w^{\top}\left(Kw\circ Lw\right)\\ 
 &=\tr\left(D_{w}Kww^{\top}L\right).
\end{align*}
As the entrywise product is symmetric in its arguments, we see that also $${a_{3}=\tr\left(D_{w}Lww^{\top}K\right)=\tr\left(ww^{\top}KD_{w}L\right)}.$$
Thus the weighted HSIC is
\begin{align*}
a_1+a_2-2a_3 & = \tr(D_wKD_wL)-\tr(D_wKww^{\top}L)-\tr(ww^{\top}KD_wL)\\ &\quad +\tr(ww^{\top}Kww^{\top}L)\\
&= \tr\left(\left(D_w-ww^{\top}\right)K\left(D_w-ww^{\top}\right)L\right)\\
&= \tr\left(H_{w}KH_{w}L\right),
\end{align*}
with $H_{w}=\left(D_{w}-ww^{\top}\right)$. In the standard
HSIC case $w=\frac1n(1,1,\dots,1)\coloneqq 1_n$ and, $D_{}=\frac{1}{n}I$, so that $H_{w}=\frac{1}{n}I-1_{n}1_{n}^{\top}$
is the standard (scaled) centering matrix. 
\end{proof}
\subsection{Using multiple transformations \label{sec:appendix:multiple_transformations}} 

We list 4 ways of combining $p$-values.
\begin{itemize}
\item[\textbf{Method 1:}] Use a Bonferroni correction and reject $H_0$ if for the smallest $p$-value, denoted by $p_{(1)}$, it holds that $p_{(1)}\leq \alpha/m$. 
\item[\textbf{Method 2:}] Make the following (random) rejection decision: reject $H_0$ with probability $ \sum_{i=1}^m 1\{p_i\leq \alpha\}/m$, and accept $H_0$ otherwise. 
\item[\textbf{Method 3:}] Fix $\beta\leq\alpha$ and reject if $ \sum_{i=1}^m 1\{p_i\leq \beta \}/m \geq \beta/\alpha$. For example, reject if $\sum_{i=1}^m 1\{p_i\leq 3\alpha/4\}/m \geq 3/4$.
\item[\textbf{Method 4:}] Reject if $2\sum_{i=1}^m p_i/m\leq \alpha$. This has the advantage that it results in a quantity that can be used as a $p$-value: $2\sum_{i=1}^m p_i/m$.  
\end{itemize}

Throughout this section, assume that the null hypothesis holds. Let $p$ be the $p$-value resulting from sampling a dataset $D$ once, followed by running optHSIC once (so exactly one transformation and one permutation test on the transformed data). We aim to show that the methods 1 and 2 above have correct type 1 error under the assumption that $P_{H_0}(p\leq \alpha)\leq \alpha$ for $\alpha \in [0,1]$ which we proved for $C\independent X$ and expect to be (approximately) true for $C\not \independent X$. We aim to show that methods 3 and 4 have asymptotically (as  the number of $p$-values goes to infinity) correct type 1 error rate under the assumption that $p\sim \text{Unif}[\frac{1}{B+1},\dots,\frac{B+1}{B+1}]$, which we proved for $C\independent X$ and expect to be (approximately) true also for $C\not \independent X$. See Table \ref{table:overview_type_1_error}. While method 2 is less conservative, it is a random rejection decision which is less desirable. We can imagine Method 3 being not too conservative when $\beta=3\alpha/4$.

\subsubsection{Method 1}
Assume it holds that $P_{H_0}(p\leq \alpha) \leq \alpha$ (see comments at the start of the section). Let $ p_{(1)},\dots,p_{(m)}$ be the $p$-values obtained from applying optHSIC $m$ times to $D$, in ascending order. The Bonferroni correction procedure rejects $H_0$ if $p_{(1)}\leq \alpha/m$. This has the correct type 1 error probability because by the union bound under the null hypothesis
\begin{align*}P_{H_0}(\text{reject})&=P_{H_0}(p_{(1)}\leq \alpha/m) \\ &\leq  m P_{H_0}(p_1\leq \alpha/m) \\ &\leq \alpha. \end{align*}

\subsubsection{Method 2} Assume it holds that $P_{H_0}(p\leq \alpha)\leq \alpha$ (see comments at the start of the section). Given the $p$-values $p_1,\dots,p_m$, the second method makes a random rejection decision in the following way: Reject $H_0$ with probability $\sum_{i=1}^m1\{p_i\leq \alpha\}/m$, and accept $H_0$ otherwise. This has correct type 1 error because
\begin{align*}
P_{H_0}(\text{reject})&=E_{H_0}[P(\text{reject}|p_1,\dots,p_m)] \\ &=
E_{H_0}[\sum_{i=1}^m1\{p_i\leq \alpha\}/m]\\ &=P_{H_0}(p_1\leq \alpha)\\ &\leq \alpha.
\end{align*}
\subsubsection{Method 3}
Fix $\beta\leq\alpha$ and reject if $ \sum_{i=1}^m 1\{p_i\leq \beta \}/m \geq \beta/\alpha$. An example would be to set $\beta=\alpha/2$ in which case we reject if $\sum_{i=1}^m1\{ p_i\leq \alpha/2 \}\geq \frac12$. Assume $P_{H_0}(p\leq \alpha)\leq \alpha$. 

This is an approximate method. The `ideal' and practically impossible method is to reject if $P(p\leq \beta |D)\geq \beta/\alpha$. We show that this `ideal' method has the correct type 1 error:
\begin{align*}
P_{H_0}(\text{reject})=P_{H_0}(A)
\end{align*}
where $A$ is the event that
\begin{align*}
A=\{ D: P(p\leq \beta|D)\geq \beta/\alpha \}.
\end{align*}
Assume by contradiction that $P_{H_0}(A)>\alpha$. Then it must hold that 
\begin{align*}
P_{H_0}(p\leq \beta)&=E_{H_0}[P(p\leq \beta |D)] \\ &\geq E_{H_0}[1_{A}P(p\leq\beta|D)] \\ & >\alpha \beta / \alpha \\& = \beta. 
\end{align*}
which contradicts that $P_{H_0}(p\leq \beta)\leq \beta$. Hence it must hold that $P_{H_0}(A)\leq \alpha$. Because in practice $P(p\leq \beta|D)$ is unkown, we can estimate it by $ \sum_{i=1}^m 1\{p_i\leq \beta \}/m $ and reject if $ \sum_{i=1}^m 1\{p_i\leq \beta \}/m \geq \beta/\alpha$. Since  $ \sum_{i=1}^m 1\{p_i\leq \beta \}/m  \to P(p\leq \beta|D)$  as $m \to \infty$ it is easy to see the approximate method is asymptotically correct.

\subsubsection{Method 4} Method 4 is to reject if $2\sum_{i=1}^m p_i/m\leq \alpha$. This is an approximation of the `ideal' and practically impossible method of rejecting $H_0$ if $D$ is such that $E(p|D)\leq \alpha/2$. We assume it holds that $p\sim \text{Uniform}[0,1]$: if we prove it under the assumption $p\sim \text{Uniform}[0,1]$ the result also follows under the assumption $p\sim \text{Uniform}[\frac{1}{B+1},\dots,\frac{1}{B+1}]$ since the latter distribution corresponds to a more conservative test. We now show that this `ideal' method has the correct type 1 error rate. Note
\begin{align*}
P_{H_0}(\text{reject})=P_{H_0}(A)
\end{align*} 
where $A$ is the event that 
\begin{align*}
A=\{D:E(p|D)\leq \alpha/2 \}.
\end{align*}

Define the following family of distributions:
\begin{align*}
M_A=(\mu_D)_{D\in A}
\end{align*}
where 
\begin{align*}\mu_D([a,b])\coloneqq P(p\in[a,b]|D).
\end{align*}
We verify that the family $M_A$ and the set $A$ satisfy three conditions:\\
\textbf{Condition 1:} For all $\mu_D \in M_A$ it holds that $$\mu_D([0,1])=1.$$ 
\textbf{Condition 2:} For all $\mu_D \in M_A$  it holds that
\begin{align*}
\int_{[ 0,1]}x \mu_D(dx)\leq \alpha/2
\end{align*}
by definition of $A$. \\
\textbf{Condition 3:} For all $0\leq a \leq b \leq 1$
\begin{align*}
E_{H_0}[1\{D\in A\}\mu_D([a,b])] & = E_{H_0}[1\{D\in A \}P(p\in[a,b]|D)] \\ & \leq E_{H_0}[P(p\in [a,b]|D)] \\ & = (b-a) 
\end{align*}

We now define the $\nu_A$ to be an `average' of the distributions in $M_A$:
\begin{align*}
\nu_A([a,b])=E_{H_0}[1\{D\in A \}\mu_D([a,b])]/P_{H_0}(A) 
\end{align*}
It is easy to see $\nu_A $ satisfies condition 1:
\begin{align*}
\nu_A[0,1]&=E_{H_0}[1\{D\in A\}\mu_D([0,1])]/P_{H_0}(A)\\ & =1.
\end{align*}
To see $\nu_A$ satisfies condition 2 note that 
\begin{align*}
\int_{[0,1]}x\nu_{A}(dx) & = \int_{[0,1]} x E_{H_0}[1\{D \in A \} \mu_D(dx)]/P_{H_0}(A) \\ & = E_{H_0}[1\{D\in A\}\int_{[0,1]}x \mu_D(dx)]/P_{H_0}(A)\\  & \leq E_{H_0}[ 1\{D\in A\} \alpha/2 ]/P_{H_0}(A) \\ & = \alpha/2.
\end{align*}
To see $\nu_A$ satisfies condition 3 note that 
\begin{align*}
E_{H_0}'(1\{D'\in A \}\nu_A[a,b])&= P_{H_0}(A)\nu_A[a,b]\\ &= E_{H_0}'(1\{D'\in A \} E_{H_0}[1\{D\in A \} \mu_D([a,b])])/P_{H_0}( A) \\ & \leq E_{H_0}(1\{D'\in A \} (b-a))/P_{H_0}( A) \\ &=b-a.
\end{align*}
Here $E'_{H_0}$ denotes expectation with respect to $D'$ and $E_{H_0}$ with respect to $D$. Note condition 1 says that $\nu_A$ is a probability measure, condition $2$ says its expectation is less than $\alpha/2$ and condition 3 that $v_A$ is dominated by the measure defined by the uniform density $1/ P_{H_0}(A)$.

Thus, if $ P_{H_0}(A)=\beta$, then $\nu_A$ satisfies the three conditions above, with $\beta$ in the third condition. We now show that there is a maximum value $\beta^\star$ so that if $P_{H_0}(A)=\beta>\beta^\star$, then it is impossible for any distribution $\nu$ to satisfy the three conditions above.

We first show $\beta^\star \geq \alpha.$ Assume that $P_{H_0}(A)=\alpha$. If we let $\nu_{\alpha}$ be the uniform probability measure on $[0,\alpha]$, then it is clear that the first two conditions are met: it is a valid probability distribution (condition 1), the expectation is exactly $\alpha/2$ (condition 2). The third condition is met since  for $0\leq a\leq b \leq \alpha$ it holds that
\begin{align*}
E_{H_0}[1\{D\in A \}\nu_{\alpha}([a,b])]=P_{H_0}(A)(b-a)/\alpha = b-a
\end{align*}
because by assumption $P_{H_0}(A)=\alpha$.

We now need to show that if $\beta > \alpha$ there does not exist a distribution $\nu$ that satisfies the three conditions. To that end, note first that if $\nu$ satisfies the three conditions for $\beta$, then it also satisfies the conditions for any $\beta'$ such that $\beta'<\beta$ (note $P_{H_0}(A)$ only appears in the third condition). So in particular such $\nu$ would have to satisfy the conditions also with $\beta=\alpha$. However, to change the distribution $\nu_{\alpha}$ defined above, one cannot place more mass in the region $[0,\alpha]$ by condition 3, which says $\nu$ needs to be dominated by the measure defined by the uniform density $1/P_{H_0}(A)$. On the other hand, if one removes mass from $[0,\alpha]$ then one automatically increases the mean of the distribution, which violates condition 2, since the mean of $\nu_{\alpha}=\alpha/2$. We conclude that $\beta^\star=\alpha$. The type 1 error of the `ideal' method is thus at most $\alpha$. 

Since $\sum_{i=1}^m p_i/m\to E(p|D)$ as $m\to \infty$ it is easy to see the approximate method has asymptotically correct type 1 error rate. This method has the advantage that it results in a combined $p$-value: $2\sum_{i=1}^mp_i/m$, whereas the other methods only lead to rejection decisions. The $p$-value will be conservative if there is little randomness in the dataset. In the case there is no censoring, the $p$-value is a factor 2 bigger than necessary. However, the Bonferroni correction would result in a $p$-value that is a factor $m$ bigger than necessary (where $m$, the number of transformations used, which may be much larger than $2$).

\subsection{Tables}
\subsubsection{ Type 1 error rates \label{sec:appendix:type1error}}

\begin{figure}[H]
\centering
\begin{subfigure}{.5\textwidth}
  \centering
\includegraphics[width=1.1\linewidth]{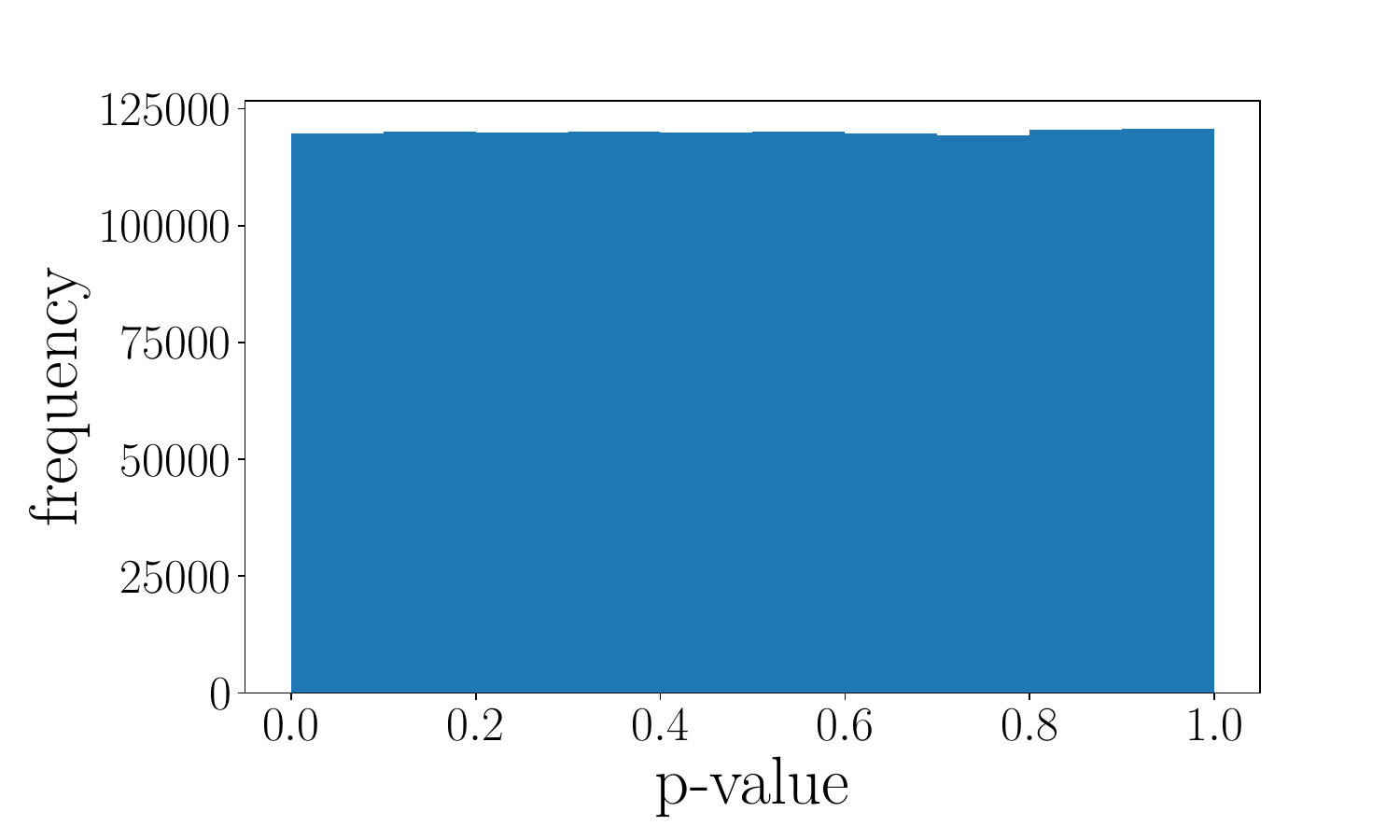}
\end{subfigure}%
\begin{subfigure}{.5\textwidth}
  \centering
\includegraphics[width=1.1\linewidth]{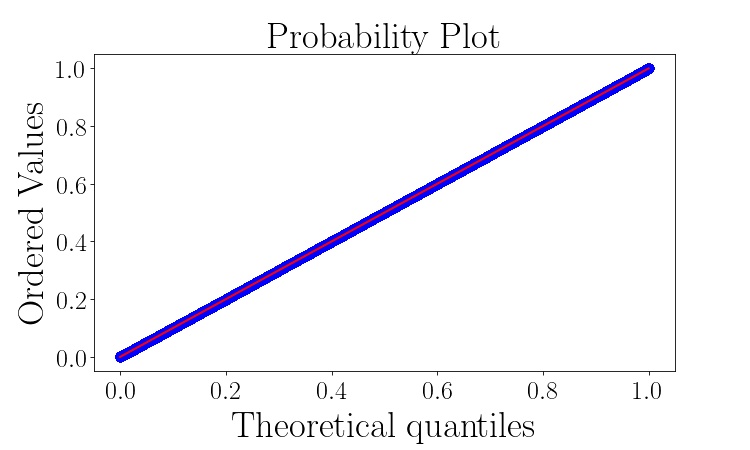}
\end{subfigure}
\caption{\label{fig:lin_censoring_uniform_pvalues} A histogram and a qq-plot of the $p$-values obtained from optHSIC for distribution D.6 of Table \ref{table:scenarios_type1error}, in which $C\not \independent X$, with a sample size of $n=200$. Data was sampled $1.2$ million times and on each sample the optHSIC test was performed, resulting in 1.2 million $p$-values. These plots indicate that despite the fact that there was a strong dependence between $C$ and $X$, the $p$-values returned by optHSIC are approximately $\text{Uniform}[0,1]$.  }
\end{figure}

\begin{table}[H]
\resizebox{\textwidth}{!}{
\begin{tabular}{@{}lllllllllll@{}}
\toprule
 $n=$   & $40$ & $80$ & $120$ & $160$ & $200$ & $240$ & $280$ & $320$ & $360$ & $400$ \\ \midrule
D.1 & 0.048 & 0.050 & 0.051 & 0.052 & 0.052 & 0.049 & 0.047 & 0.054 & 0.051 & 0.049 \\ D.2 &
0.048 & 0.051 & 0.047 & 0.047 & 0.049 & 0.051 & 0.048 & 0.054 & 0.050 & 0.047 \\ D.3 &
0.243 & 0.461 & 0.630 & 0.774 & 0.860 & 0.909 & 0.953 & 0.970 & 0.985 & 0.991 \\ D.4 &
0.142 & 0.232 & 0.343 & 0.487 & 0.610 & 0.734 & 0.812 & 0.880 & 0.932 & 0.959 \\ D.5 &
0.075 & 0.116 & 0.142 & 0.168 & 0.210 & 0.243 & 0.278 & 0.316 & 0.357 & 0.397 \\ D.6 &
0.932 & 1.000 & 1.000 & 1.000 & 1.000 & 1.000 & 1.000 & 1.000 & 1.000 & 1.000 \\ D.7&
0.308 & 0.594 & 0.761 & 0.856 & 0.906 & 0.937 & 0.960 & 0.971 & 0.980 & 0.984 \\ D.8&
0.078 & 0.122 & 0.152 & 0.211 & 0.264 & 0.307 & 0.355 & 0.388 & 0.439 & 0.467 \\ \bottomrule
\end{tabular}}
\caption{The rejection rate of zHSIC in against the distributions D.1-8 of Table \ref{table:scenarios_type1error}. \label{table:rejectionrates_type1error}}
\end{table}

\begin{table}[H]
\resizebox{\textwidth}{!}{
\begin{tabular}{@{}lllllllllll@{}}
\toprule
 $n=$   & $40$ & $80$ & $120$ & $160$ & $200$ & $240$ & $280$ & $320$ & $360$ & $400$ \\ \midrule
D.1 & 0.045 & 0.047 & 0.049 & 0.050 & 0.049 & 0.051 & 0.049 & 0.054 & 0.048 & 0.049 \\ D.2 &
0.050 & 0.047 & 0.048 & 0.047 & 0.050 & 0.048 & 0.047 & 0.045 & 0.048 & 0.049 \\ D.3 &
0.079 & 0.166 & 0.235 & 0.326 & 0.410 & 0.466 & 0.540 & 0.597 & 0.658 & 0.700 \\ D.4&
0.161 & 0.204 & 0.232 & 0.270 & 0.309 & 0.350 & 0.394 & 0.456 & 0.506 & 0.549 \\ D.5&
0.057 & 0.084 & 0.110 & 0.131 & 0.163 & 0.191 & 0.216 & 0.258 & 0.274 & 0.299 \\ D.6 &
0.267 & 0.672 & 0.900 & 0.971 & 0.990 & 0.998 & 0.999 & 0.999 & 1.000 & 1.000 \\ D.7&
0.071 & 0.107 & 0.143 & 0.192 & 0.260 & 0.305 & 0.369 & 0.412 & 0.480 & 0.527 \\ D.8&
0.057 & 0.078 & 0.081 & 0.095 & 0.120 & 0.142 & 0.153 & 0.162 & 0.177 & 0.190 \\ \bottomrule
\end{tabular}}
\caption{The rejection rate of wHSIC in against the distributions D.1-8 of Table \ref{table:scenarios_type1error}. \label{table:rejectionrates_whsic_type1error}}
\end{table}

\begin{table}[H]
\resizebox{\textwidth}{!}{
\begin{tabular}{@{}lllllllllll@{}}
\toprule
 $n=$   & $40$ & $80$ & $120$ & $160$ & $200$ & $240$ & $280$ & $320$ & $360$ & $400$ \\ \midrule
D.1 & 0.056 & 0.054 & 0.050 & 0.047 & 0.055 & 0.048 & 0.048 & 0.055 & 0.050 & 0.051 \\ D.2 &
0.055 & 0.056 & 0.055 & 0.050 & 0.055 & 0.051 & 0.049 & 0.050 & 0.052 & 0.050 \\ D.3 &
0.057 & 0.056 & 0.051 & 0.053 & 0.054 & 0.051 & 0.046 & 0.057 & 0.048 & 0.050 \\ D.4 &
0.057 & 0.061 & 0.050 & 0.051 & 0.054 & 0.058 & 0.049 & 0.053 & 0.051 & 0.051 \\ D.5 &
0.058 & 0.058 & 0.055 & 0.052 & 0.053 & 0.053 & 0.048 & 0.054 & 0.048 & 0.049 \\ D.6 &
0.053 & 0.051 & 0.051 & 0.050 & 0.046 & 0.048 & 0.057 & 0.053 & 0.054 & 0.055 \\ D.7 &
0.148 & 0.094 & 0.084 & 0.062 & 0.063 & 0.061 & 0.058 & 0.055 & 0.060 & 0.058 \\ D.8&
0.143 & 0.086 & 0.074 & 0.064 & 0.067 & 0.058 & 0.059 & 0.058 & 0.061 & 0.060\\ \bottomrule
\end{tabular}}
\caption{The rejection rate of the Cox proportional hazards likelihood ratio test in against the distributions D.1-8 of Table \ref{table:scenarios_type1error}. \label{table:rejectionrates_type1error_cph}}
\end{table}

\subsubsection{Rejection rate under varying censoring regimes \label{sec:appendix:varying_censoring}}

\begin{table}[H]
\centering
\resizebox{\textwidth}{!}{
\begin{tabular}{  l l l l  }
\toprule
    D. &  $ Z \vert X$  & $C \vert X$  & X \\ \hline
    1   & $ \text{Exp}(\text{mean}=\exp(X/5)) $ & $ \text{Exp}(\text{mean}=\theta) $  & $N(0,1)$ \\ 
    2   & $ \text{Exp}(\text{mean}=\exp(X^2)/5) $ & $ \text{Exp}(\text{mean}=\theta \exp(X)) $  & $N(0,1)$ \\ 
    3  & $ \text{Weib}(\text{shape}= 1.75X+3.25) $ & $ \text{Exp}(\text{mean}=\theta X^2) $  & $\text{Unif}[-1,1]$ \\ 
   4 & $ {N}(\text{mean}=100-X,\text{var}= 2X+5.5) $ & $ 82+ \text{Exp}(\text{mean}=\theta) $  & $\text{Unif}[-1,1]$ \\ 
    5 & $ \text{Exp}(\text{mean}=\exp(1^T X/30))$ & $ \text{Exp}(\text{mean}=\theta) $   & $N_{10}(0,\text{cov}=\Sigma_{10})$  \\
    6  & $ \text{Exp}(\text{mean}=\exp(X_4/7)) $ & $ \text{Exp}(\text{mean}=\theta \exp(1^T X/30)) $  & $N_{10}(0,\text{cov}=\Sigma_{10})$   \\
    7  & $ \text{Exp}(\text{mean}=\exp(X_4^2/20)) $ & $ \text{Exp}(\text{mean}=\theta\exp(X^2_2)/20) $  & $N_{10}(0,\text{cov}=\Sigma_{10})$   \\
        8  & $ \text{Exp}(\text{mean}=\exp(X_{10}^2+2X_8)/20) $ & $ \text{Exp}(\text{mean}=\theta\exp(X_2/7)) $  & $N_{10}(0,\text{cov}=\Sigma_{10})$   \\ \bottomrule
\end{tabular}}\caption{\label{table:appendix:varyingcensoringdistributions}The parametrized distributions to test the power under different censoring rates. Here $\Sigma_{10}=MM^T$ where $M$ is a $10\times 10$ matrix of i.i.d. standard normal entries. M is sampled once and then kept fixed. The parameter $\theta$ varies such that $20,40,60,80,100\%$ of the individuals are observed (i.e. $\Delta=1$). The sample size is $n=200$ in each case.  }    
\end{table}

\begin{table}[H]
\centering
\begin{tabular}{@{}lllllll@{}}
\toprule
 $\% \Delta=1$ &         & $20\%$ & $40\%$ & $60\%$ & $80\%$ & $100\%$ \\ \midrule
D.1 & Cph     & 0.243  & 0.422  & 0.565  & 0.705  & 0.753   \\
  & optHSIC & 0.229  & 0.382  & 0.501  & 0.634  & 0.699   \\
  & wHSIC   & 0.038  & 0.062  & 0.182  & 0.395  & 0.703   \\
  & zHSIC   & 0.066  & 0.168  & 0.329  & 0.525  & 0.701   \\ \midrule
D.2 & Cph     & 0.180  & 0.267  & 0.268  & 0.225  & 0.108   \\
  & optHSIC & 0.087  & 0.171  & 0.258  & 0.378  & 0.686   \\ \midrule
D.3 & Cph     & 0.073  & 0.056  & 0.107  & 0.223  & 0.288   \\
  & optHSIC & 0.242  & 0.177  & 0.399  & 0.886  & 0.968   \\ \midrule
D.4 & Cph     & 0.187  & 0.091  & 0.064  & 0.046  & 0.039   \\
  & optHSIC & 0.346  & 0.224  & 0.275  & 0.509  & 0.779   \\
  & wHSIC   & 0.138  & 0.285  & 0.452  & 0.654  & 0.770   \\
  & zHSIC   & 0.105  & 0.172  & 0.274  & 0.410  & 0.759   \\ \midrule
D.5 & Cph     & 0.315  & 0.487  & 0.610  & 0.705  & 0.836   \\
  & optHSIC & 0.268  & 0.439  & 0.546  & 0.629  & 0.775   \\
  & wHSIC   & 0.055  & 0.072  & 0.169  & 0.409  & 0.786   \\
  & zHSIC   & 0.083  & 0.229  & 0.362  & 0.605  & 0.760   \\ \midrule
D.6 & Cph     & 0.461  & 0.732  & 0.834  & 0.916  & 0.939   \\
  & optHSIC & 0.396  & 0.681  & 0.801  & 0.876  & 0.952   \\ \midrule
D.7 & Cph     & 0.055  & 0.068  & 0.077  & 0.078  & 0.107   \\
  & optHSIC & 0.043  & 0.100  & 0.134  & 0.289  & 0.669   \\ \midrule
D.8 & Cph     & 0.162  & 0.313  & 0.431  & 0.517  & 0.572   \\
  & optHSIC & 0.164  & 0.335  & 0.498  & 0.619  & 0.916   \\ \bottomrule
\end{tabular}
\caption{The rejection rates of the various methods against distributions D.1-D.8 given in Table \ref{table:appendix:varyingcensoringdistributions}. When $C\not \independent X$, we only show rejection rates of the CPH test and optHSIC, because wHSIC and zHSIC have high inflated rejection rates due to the dependency of $C$ and $X$. The top row shows the percentage of observed events $(\Delta=1)$. }
\end{table}

\subsection{Binary covariates\label{sec:appendix:binary_covariates}}

As a special case of independence testing we consider the case of a single
binary covariate, i.e., $X\in \{0,1\}$. If one groups the data by covariate,
then testing independence of $T$ and $X$ is equivalent to testing equality 
of lifetime distribution between the two groups. This is known as two-sample testing on right censored data. Popular approaches to this challenge
are the logrank test and various weighted logrank tests. optHSIC can be applied to this problem without any adjustments, while wHSIC can be improved in this case in two ways: first, the weights can be estimated even when the censoring distribution differs between the two groups; and second, there exists an alternative permutation strategy that, experiments show, seems to
control the type 1 error effectively even under dependent censoring. These adjustments are described in Section \ref{adjustingwhsic} and Section \ref{alternativepermutation} respectively. We omit consideration of zHSIC,
as it is fundamentally more limited, given the larger number of available methods.

\subsubsection{wHSIC for two-sample testing \label{adjustingwhsic}}
Let $P_0$ and $P_1$ denote the distribution of $T\vert X=0$ and $T\vert X=1$ respectively. Let the total sample be $D=\left( (x_i,z_i,\delta_i) \right)_{i=1}^n$ as before, and write $\left( (z_i^0,\delta^0_i) \right)_{i=1}^{n_0}$ and $\left((z_i^1,\delta^1_i)\right)_{i=1}^{n_1}$ for the event times and indicators of individuals with covariate $X=0$ and $X=1$ respectively. We want to asses if $P_0 = P_1$. We again use the covariance kernel of Brownian motion. If all of the $n$ times were observed ($\delta=1$), we could measure the difference in empirical distributions between both groups by the MMD between the two distributions:
\begin{align*}
\bigg\vert \bigg \vert \frac{1}{n_0}\sum_{i=1}^{n_0}k(z^0_i,\cdot) - \frac{1}{n_1}\sum_{j=1}^{n_1} k(z^1_j,\cdot) \bigg \vert \bigg \vert_{H}.
\end{align*}
Similar to Section \ref{sectionwhsic}, when some observations are censored, we might reweight the empirical distributions, and instead compare the weighted empirical distributions 
$$ \sum_{i=1}^{n_0} w_i^0 k(z_i^0,\cdot) \qquad \text{and} \qquad \sum_{i=1}^{n_0} w_i^1 k(z_i^1,\cdot).$$
We propose that the weights $w_i$ are computed by the Kaplan--Meier weights \emph{within} each group. The test statistic thus becomes:
\begin{align*}
\text{wHSIC}(D) \coloneqq \bigg\vert \bigg \vert\sum_{i=1}^{n_0} w_i^0 k(z_i^0,\cdot) - \sum_{i=1}^{n_0} w_i^1 k(z_i^1,\cdot) \bigg \vert \bigg \vert^2_{H}.
\end{align*}
This statistic was also, independently, proposed by \cite{matabuenatwosample}, and can be seen as a special case of wHSIC in the case of binary covariates. Under the hypothesis that $C\independent X$, one can obtain $p$-values using a permutation test, resulting in the following algorithm. Section \ref{alternativepermutation} provides an alternative permuation strategy under dependent censoring, that was proposed by \cite{wangpermutation}. It was proposed in the context of the logrank test, but can equally be used for other statistics.\\

\begin{algorithm}[H]
    \SetKwInOut{Input}{Input}
    \SetKwInOut{Output}{Output}
	\Input{  $D=((x_i,z_i,\delta_i))_{i=1}^{n}$, significance level $\alpha$, number of permutations $B$. } 
Sample permutations $\pi_1,\dots,\pi_B$ i.i.d. uniformly from $S_n$. \;
Breaking ties at random, compute the rank $R$ of $\whsic(D)$ in the vector $$\left(\whsic(D),\whsic (\pi_1 (D)),\whsic (\pi_2 (D)),\dots,\whsic (\pi_{B} (D)) \right)$$ where wHSIC is as defined above.  \;
	\Output{Reject if $p\coloneqq R/(B+1)\leq \alpha.$}
    \caption{wHSIC for two-sample data}
\end{algorithm}

\subsubsection{ipxHSIC \label{alternativepermutation}}

This subsection overviews a test we name ipxHSIC, which uses the same statistic $\text{wHSIC}(D)$ defined in Section \ref{adjustingwhsic} above, but a different permutation strategy that is robust against differences in the censoring distributions of both groups. The permutation strategy was proposed in \cite{wangpermutation} to provide reliable $p$-values for the logrank statistic in the case of small or unequal sample sizes. In fact \cite{wangpermutation} propose two permutation strategies: the first one, which they call `ipz' (section 2.1.1), permutes group membership and the second, which they call `ipt'(section 2.1.2), permutes survival times. These permutation strategies were proposed in the context of logrank tests - but can equally be applied to other statistics, such as wHSIC. The first strategy, which permutes the covariates, is referred to in their work as `ipz' since the procedure first imputes several unobserved times, and then permutes the covariate, which in their work is denoted by $z$. We refer to it as `ipx', as our covariate is denoted by $x$. The algorithm uses the Kaplan--Meier estimator to estimate three distributions: 1) $G^0$, the censoring distribution in group 0, based on the data observed in group 0; 2) $G^1$, the censoring distribution in group 1, based on the data observed in group 1; 3) the distribution of the lifetimes $F$ based on the pooled dataset containing both groups. With these estimates, a new dataset is constructed, consisting of $n$ observations, each consisting of a covariate, an event time, and two censoring times, one for each censoring distribution. This larger dataset is then permuted, and transformed back to a censored dataset.  \cite{wangpermutation} describe the algorithm in full detail. This method thus combines the wHSIC statistic with an alternative permutation strategy. Because this method relies on explicitly estimating censoring distributions in each group, it is difficult to extend this to the continuous case, where for each covariate we only have one individual in the study with that exact covariate. 

\subsubsection{Numerical comparison of methods in the two-sample case}

We generate data from four different distributions for each of $X$, $T$, and $C$ to compare the power and type 1 error of the proposed methods optHSIC, wHSIC, ipxHSIC to the power and type 1 error of the classic logrank test and a weighted logrank test proposed by \cite{weightedlogrank}. The classical logrank test is known to have low power against certain alternatives, such as crossing survival curves. A weighted logrank test assigns weights to data, giving the logrank test power against different alternatives. In \cite{weightedlogrank} a combination of weights is proposed, so as to achieve power against a wider class of alternatives. In particular \cite{weightedlogrank} propose a combination of two sets of weights, corresponding to proportional and crossing hazards. As this section mostly serves to provide an example of our methods, we simulate fewer scenarios than in Section \ref{sec:experiments}. In each scenario we let the $n$ values range from $n=20$ to $n=400$ in intervals of $20$. To obtain $p$-values in the three HSIC based methods as well as the weighted logrank test we use a permutation test with 1999 permutations. We reject the null hypothesis if our obtained $p$-value is less than $0.05$. \\

\begin{table}[H]
\begin{center}
    \begin{tabular}{  l  l  l  l  l  l  p{5cm} }
    
    D. & $T_0$ & $T_1$ & $C_0$ & $C_1$ & \% Observed \\ \hline
    1  & $\text{Exp}(1)$ &   $\text{Exp}(1/1.6)$   & $ \text{Exp}(1/2) $  & $ \text{Exp}(1/2) $ & 60 \% \\ 
     2  & $\text{Weib}(1,5)$  & $\text{Weib}(1,1.5) $ & $ \text{Exp}(1/2) $  & $ \text{Exp}(1/2) $ & 60 \%  \\ 
3 & $\text{Exp}(1)$  & (0.43, 1.39+$\text{Exp}(1))$  & $1+\text{Exp}(1/2)$  &$1+\text{Exp}(1/2)$  & 90 \% \\ 
4 & $\text{Exp}(1)$ & $\text{Exp}(1)$ & $\text{Exp}(2)$ & None & 65 \% \\ 
    \end{tabular}\\
\end{center}
\caption{\label{tabletwosample}The 4 scenarios in which in which we perform two-sample tests. $T_1$ is $0.43$ w.p. $0.75$ and $1.39+\text{Exp}(1)$ w.p. 0.25. Note that in D.4 the null hypothesis holds.}    
\end{table}

\begin{figure}[H]
\begin{subfigure}{.5\textwidth}
  \includegraphics[width=1.1\linewidth]{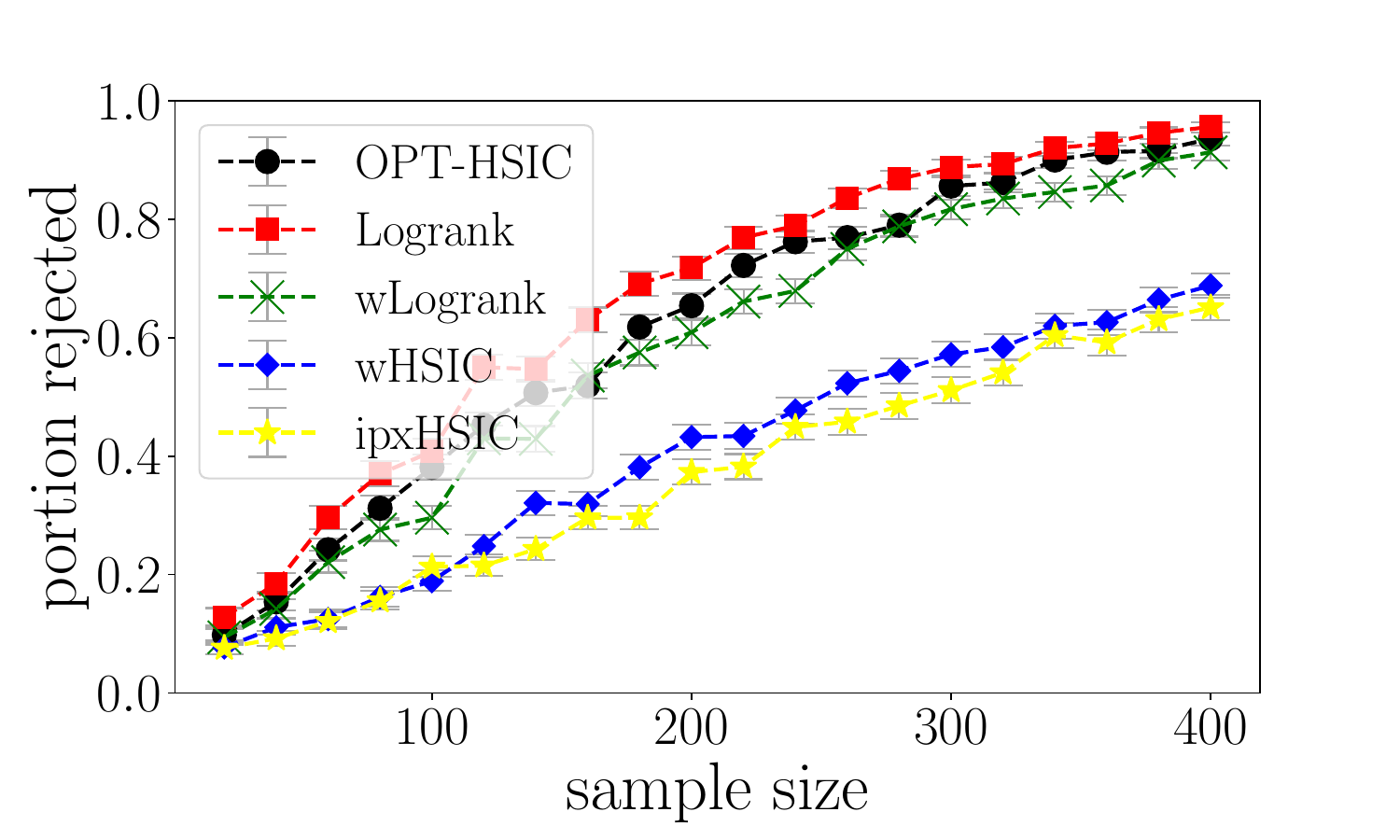}
  \label{fig:sub1}
  \caption{Scenario 1}
\end{subfigure}%
\begin{subfigure}{.5\textwidth}
    \includegraphics[width=1.1\linewidth]{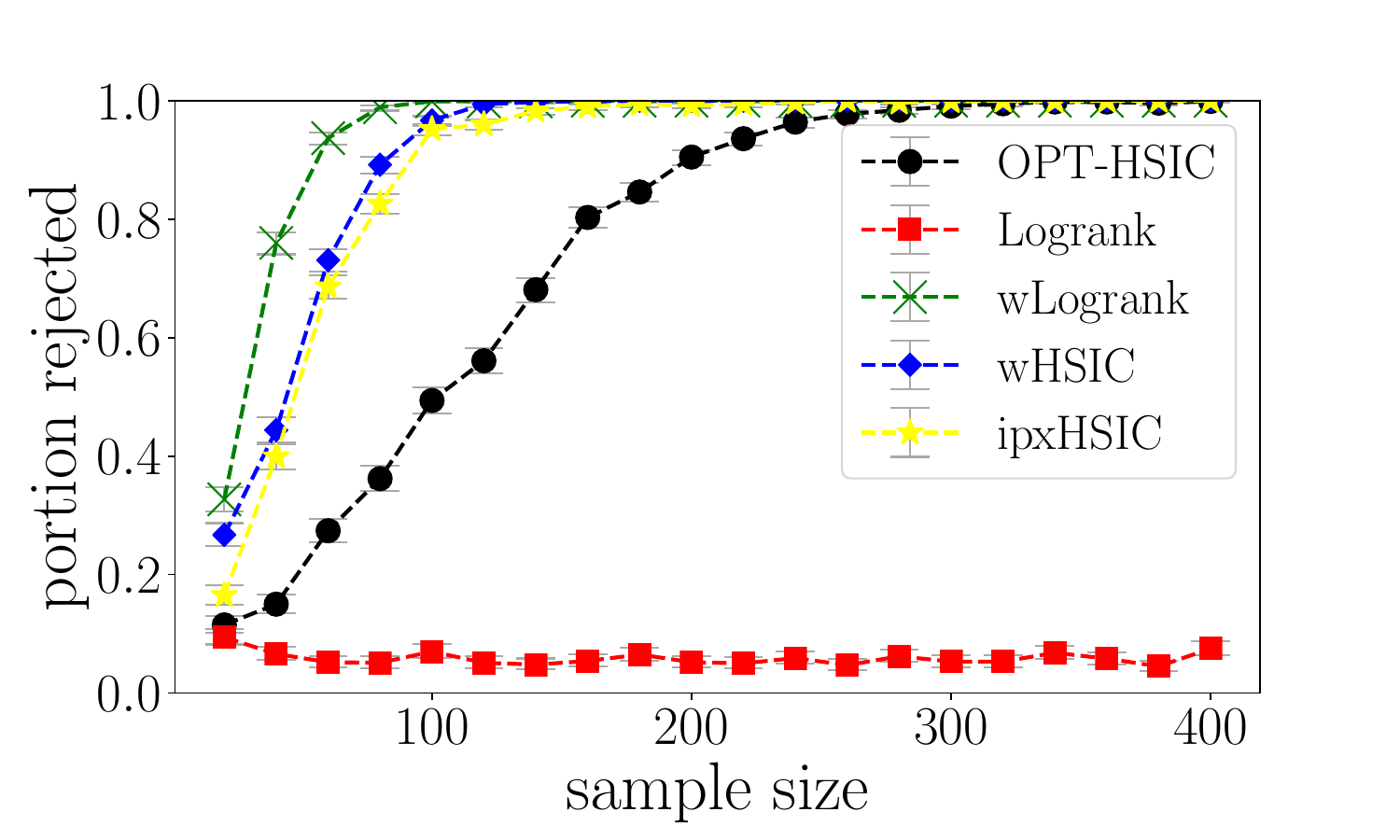}
  \label{fig:sub1}
  \caption{Scenario 2}
\end{subfigure}\\
\centering
\begin{subfigure}{.5\textwidth}
 \includegraphics[width=1.1\linewidth]{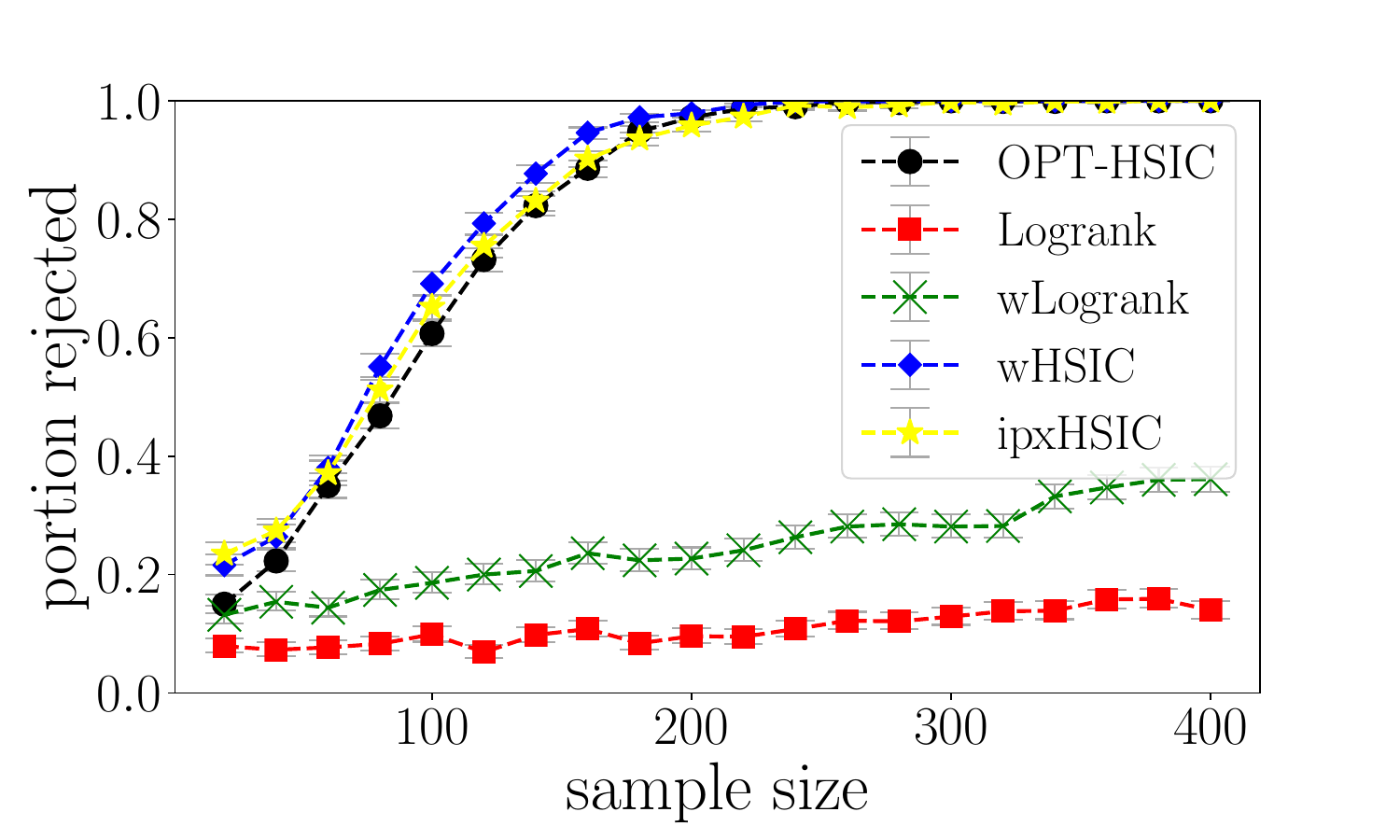}
  \label{fig:sub1}
      \caption{Scenario 3}
\end{subfigure}%
\begin{subfigure}{.5\textwidth}
\includegraphics[width=1.1\linewidth]{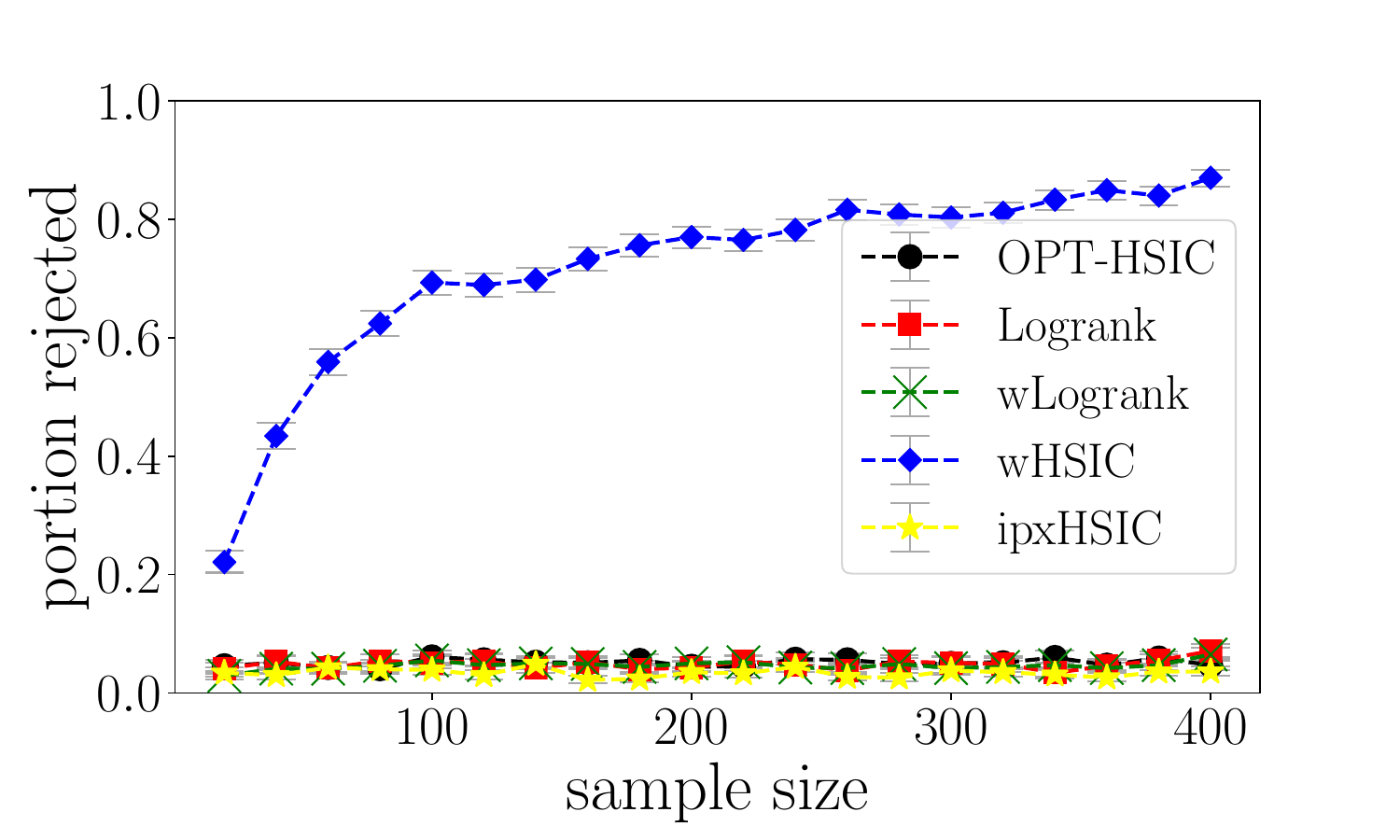}
\label{fig:sub1}
\caption{Scenario 4}
\end{subfigure}%
\caption{\label{plottwosample}Rejection rates of the various two-sample tests. Note that in Scenarios 1-3 the alternative hypothesis holds, implying high rejection rate is desirable. In Scenario 4, the null hypothesis holds, so a rejection rate of $0.05$ is desirable. wHSIC in that case thus wrongly rejects the null: this reflects the crucial assumption of wHSIC that the groups have identical censoring distributions. }
\end{figure}

\subsection{Example of data with binary covariates in which optHSIC does not perform well \label{sec:appendix:exampletwosample}}
Consider the following case. Group $X=0$ contains 1050 individuals. Group $X=1$ contains 50 individuals. Up to time $t=50$, no events occur. At time $t=50$, 1000 individuals of group $X=0$ are censored. There are now 50 individuals remaining in each of the groups. The $50$ individuals of group $X=0$ have event time $100+\text{Exp}(\text{mean}=2)$ and the $50$ individuals of group $X=1$ have event time $100+\text{Exp}(\text{mean}=1)$. In this example we find the logrank test to have power of $89\%$ and optHSIC to have power of only $12 \%$. 

What happens is the following: At time $t=100$ there are 100 individuals at risk. The individuals of group $X=1$ are likely to have their event first, due to the higher rate in the corresponding exponential distribution. Because in group $X=0$ 1000 individuals have been censored, the optimal transport map has a high chance of choosing $\tilde x=0$ when $x_i=1$. So while in the resulting dataset a slight bias will remain towards individuals in group $X=1$ having their event first, this bias is much less clear than before the transformation. (We thank a reviewer for proposing this scenario.)

There are several characteristics that make the difference in this example so large. Firstly, as mentioned before, optimal transport relies on the ability to choose a `similar covariate'. When covariates are binary it may happen that $\tilde x=0$ while $x_i=1$. Secondly, in this case all the censoring happens in group $X=0$, causing optimal transport to send mass from group $X=1$ to group $X=0$. Furthermore, the censoring rate is high ($91\%$ of all individuals). Lastly, before the censoring occurs there is no evidence of a difference in distribution.

\subsubsection{Comments on two-sample simulations}

The results show that the logrank test and the weighted logrank test have little power in scenario 2 and 3 and scenario 3 respectively, even though large differences between the samples are present. The logrank is designed to detect differences as in scenario 1, and the weighted logrank is designed to detect differences as in scenario 1 and 2, sacrificing power slightly compared to the logrank test in the first. Scenario 3 is designed to defeat the weighted logrank test, since we constructed an extreme version of an early crossing survival curve, and the test does not contain weights for early crossing. 
The kernel methods are fully nonparametric, but do lose power in certain scenarios, most notably in Scenario 2 and the example provided. We believe optHSIC is not ideally suited to the case of binary covariates, since optimal transport relies on choosing a `similar' covariate. Furthermore, while there are no fully nonparametric alternatives for independence testing for continuous covariates, there are more alternative two-sample tests. We thus believe the main value of optHSIC lies in the case of continuous covariates.

\bibliographystyle{DeGruyter}
\bibliography{sample}

\end{document}